\documentclass[12pt,reqno]{amsart}
\usepackage{amssymb,amsmath,amsthm,hyperref}
\newcommand{\sg}{\textnormal{sg}}

\newtheorem{theorem}{Theorem}
\newtheorem{lemma}[theorem]{Lemma}
\newtheorem{corollary}[theorem]{Corollary}
\newtheorem{proposition}[theorem]{Proposition}

\theoremstyle{remark}
\newtheorem*{remark}{Remark}
\theoremstyle{definition}

\newtheorem{definition}[theorem]{Definition}
\newtheorem{example}{Example}

\numberwithin{theorem}{section} \numberwithin{equation}{section}
\numberwithin{example}{section}

\renewcommand{\Re}{\text {\rm Re}}

\title{Hecke-type double sums, Appell-Lerch sums, and~mock~theta~functions (I)}

\author{Dean Hickerson and Eric Mortenson}

\dedicatory{For George Andrews in honor of his 70th birthday}

\begin{document}

\date{7 August  2012}



\begin{abstract}
By developing a connection between partial theta functions and Appell-Lerch sums, we find and prove a formula which expresses Hecke-type double sums in terms of Appell-Lerch sums and theta functions.  Not only does our formula prove classical Hecke-type double sum identities such as those found in work Kac and Peterson on affine Lie Algebras and Hecke modular forms, but once we have the Hecke-type forms for Ramanujan's mock theta functions our formula gives straightforward proofs of many of the classical mock theta function identities.  In particular, we obtain a new proof of the mock theta conjectures.  Our formula also applies to positive-level string functions associated with admissable representations of the affine Lie Algebra $A_1^{(1)}$ as introduced by Kac and Wakimoto.
\end{abstract}
\address{Eureka, California, 95501}
\address{School of Mathematics and Physics, The University of Queensland,
Brisbane, Australia 4072}
\email{dean.hickerson@yahoo.com, etmortenson@gmail.com}
\maketitle
\setcounter{section}{-1}

\section{introduction}\label{section:introduction}

 Let $q$ be a complex number with $0<|q|<1$ and define $\mathbb{C}^*:=\mathbb{C}-\{0\}$.  We recall some basics:
\begin{gather*}
(x)_n=(x;q)_n:=\prod_{i=0}^{n-1}(1-q^ix), \ \ (x)_{\infty}=(x;q)_{\infty}:=\prod_{i\ge 0}(1-q^ix),\\
 j(x;q):=(x)_{\infty}(q/x)_{\infty}(q)_{\infty}=\sum_{n}(-1)^nq^{\binom{n}{2}}x^n,\\
 {\text{and }}\ \ j(x_1,x_2,\dots,x_n;q):=j(x_1;q)j(x_2;q)\cdots j(x_n;q).
\end{gather*}
where in the last line the equivalence of product and sum follows from Jacobi's triple product identity.  We also keep in mind the easily deduced fact that $j(q^n,q)=0$ for $n\in \mathbb{Z}.$  The following are special cases of the above definition.  Let $a$ and $m$ be integers with $m$ positive.  Define
\begin{gather*}
J_{a,m}:=j(q^a;q^m), \ \ \overline{J}_{a,m}:=j(-q^a;q^m), \ {\text{and }}J_m:=J_{m,3m}=\prod_{i\ge 1}(1-q^{mi}).
\end{gather*}

In his last letter to Hardy, Ramanujan gave a list of seventeen functions which he called ``mock theta functions.''   Each mock theta function was defined by Ramanujan as a $q$-series convergent for $|q|<1$.  He stated that they have certain asymptotic properties as $q$ approaches a root of unity, similar to the properties of ordinary theta functions, but that they are not theta functions.  He also stated several identities relating some of the mock theta functions to each other.  Later, many more mock theta function identities were found in the lost notebook \cite{RLN}.

Mock theta functions have many representations---Eulerian forms, Appell-Lerch sums, Hecke-type double sums, and Fourier coefficients of meromorphic Jacobi forms.  In terms of Eulerian forms, the mock theta conjecture for the fifth order mock theta function $f_0(q)$ is stated
\begin{equation}
f_0(q):=\sum_{n\ge 0}\frac{q^{n^2}}{(-q)_n}=2-2\sum_{n\ge 0}\frac{q^{10n^2}}{(q^2;q^{10})_{n+1}(q^8;q^{10})_{n}}+\frac{(q^5;q^{5})_{\infty}(q^5;q^{10})_{\infty}}{(q;q^{5})_{\infty}(q^4;q^{5})_{\infty}}.\label{equation:mtcf0}
\end{equation}
To facilitate studying mock theta functions, it is useful to translate the Eulerian form into other representations.  Translating from one form to another has been an historically difficult problem; moreover, unifying any two of the various guises with a single equation has been an unsolved problem. 

Early attempts to unify Eulerian forms and Appell-Lerch sum forms used the theory of basic hypergeometric series.  In some cases, this led to identities between the mock theta functions and helped determine modular properties \cite{W3}.  We will use the following definition of an Appell-Lerch sum.  
\begin{definition}  \label{definition:mdef} Let $x,z\in\mathbb{C}^*$ with neither $z$ nor $xz$ an integral power of $q$. Then
\begin{equation}
m(x,q,z):=\frac{1}{j(z;q)}\sum_r\frac{(-1)^rq^{\binom{r}{2}}z^r}{1-q^{r-1}xz}.\label{equation:mxqz-def}
\end{equation}
\end{definition}
\noindent These sums were first studied by Appell \cite{Ap} and then by Lerch \cite{L1}.  Appell-Lerch sums appear in the literature under various names such as Lerch sums, Appell functions, Appell theta functions, and mock Jacobi forms.  In Section \ref{section:known}, we will see that (\ref{equation:mtcf0}) is equivalent to
\begin{equation*}
f_0(q)=2m(q^{14},q^{30},q^4)+2q^{-2}m(q^{4},q^{30},q^4)+\frac{(q^5;q^{5})_{\infty}(q^5;q^{10})_{\infty}}{(q;q^{5})_{\infty}(q^4;q^{5})_{\infty}}.
\end{equation*}
In attempts to prove the mock theta conjectures, Andrews introduced techniques to translate mock theta functions into two new representations.  He used Bailey's Lemma to go from Eulerian forms to Hecke-type sums \cite{A} and then used the constant term method to go from Hecke-type sums to Fourier coefficients of meromorphic Jacobi forms \cite{A4}.  In \cite{A}, he showed
\begin{equation}
f_0(q)\cdot (q)_{\infty}=\sum_{\substack{n\ge0\\ |j|\le n}}(-1)^jq^{\tfrac{5}{2}n^2+\tfrac12 n-j^2}(1-q^{4n+2}).\label{equation:heckef0}
\end{equation}
For this paper, we will use the following definition of the building block of Hecke-type double sums.  
\begin{definition} \label{definition:fabc-def}  Let $x,y\in\mathbb{C}^*$ and define $\sg (r):=1$ for $r\ge 0$ and $\sg(r):=-1$ for $r<0$. Then
\begin{equation}
f_{a,b,c}(x,y,q):=\sum_{\substack{\sg (r)=\sg(s)}} \sg(r)(-1)^{r+s}x^ry^sq^{a\binom{r}{2}+brs+c\binom{s}{2}}.\label{equation:fabc-def}\\
\end{equation}
\end{definition}
\noindent In terms of our building block, (\ref{equation:heckef0}) becomes
\begin{equation*}
 f_0(q)\cdot (q)_{\infty}=f_{3,7,3}(q^2,q^2,q)+q^3f_{3,7,3}(q^7,q^7,q).
\end{equation*}
\noindent Hecke \cite{He} conducted the first systematic study of such double sums, special cases of which appeared earlier in work of Rogers \cite{R}. 

Although using the methods of basic hypergeometric series, constant term, and Bailey pairs have provided a great deal of insight, researchers have been unable to unify all of the mock theta functions under a single formula.  Existing results are limited to functions within the same order, and no available method works for all orders.  Also, the known methods of basic hypergeometric series, constant term, and Bailey pairs are not robust---slight changes in the Eulerian form result in dramatic changes in the technique which needs to be applied.  Polishchuk \cite{P1} used homological mirror symmetry to show that one can always expand an indefinite theta series of signature (1,1) in terms of Appell-Lerch sums; however, examples could only be produced on a case-by-case basis.  Zwegers \cite{Zw} showed that indefinite theta series, Appell-Lerch sums, and Fourier coefficients of meromorphic Jacobi forms exhibit the same near-modular behavior.  Zwegers thus established the modularity theory for mock theta functions, and his work allows mock theta functions to be cast as the holomorphic parts of weak Maass forms \cite{BruF, Za, BrO1, BrO2}).  

Another important problem is understanding the various types of Eulerian forms and how they relate to each other.  For example, in \cite{RLN} one finds scores of Eulerian forms in mock theta function identities, Rogers-Ramanujan type identities, and partial theta function identities.  In this direction, Andrews \cite{A5} has recently produced $q$-hypergeometric formulas which simultaneously prove mock theta function identities and Rogers-Ramanujan type identities.

By developing a connection between partial theta functions and Appell-Lerch sums---the building block of mock theta functions---we obtain a master formula that expands a certain family of Hecke-type double sums in terms of Appell-Lerch sums and theta functions.  In Section \ref{section:known}, we use results from the literature to write the classical mock theta functions in terms of Appell-Lerch sums.  Given the Hecke-type sum form of a mock theta function,  our formula produces the same Appell-Lerch sum forms found in Section \ref{section:known}.    In particular, for the fifth order  $f_0(q)$ we will show
\begin{align*}
f_{3,7,3}(q^2,q^2,q)+q^3f_{3,7,3}(q^7,q^7,q)= (q)_{\infty}\cdot \Big [&2m(q^{14},q^{30},q^4)+2q^{-2}m(q^{4},q^{30},q^4)\\
&\ \ \ \ \ +\frac{(q^5;q^{5})_{\infty}(q^5;q^{10})_{\infty}}{(q;q^{5})_{\infty}(q^4;q^{5})_{\infty}}\Big ].
\end{align*}
As a consequence, our formula combined with Appell-Lerch sum properties, such as Theorems \ref{theorem:changing-z-theorem} and \ref{theorem:msplit-general-n}, reduces many of the mock theta function identities to straightforward exercises.  In particular, our formula gives a new proof of the mock theta conjectures.  Indeed, given the Hecke-type sum form, it gives us the mock theta conjecture.  Our formula also proves the Hecke-type double sum identities found in Polishchuk \cite{P1} and Kac and Peterson \cite{KP}.  One compares this with the approach using Zwegers' near-modularity result and the theory of weak Maass forms. Here one computes the correction term (i.e. a period integral of some weight $3/2$ unary theta series) for each mock theta function to make it transform like a modular form, one verifies that the modular transformation properties match those of the theta function provided by Ramanujan, one verifies that the correction terms cancel, and finally one computes the first few coefficients and compares.

In future work, we will use the techniques of this paper to develop a duality between identities involving mock theta functions and identities involving partial theta functions.  We will see that not only do the dual partial theta function expansions appear to have a structure similar to the Appell-Lerch sum expansions in this paper, but also that the dual identities in terms of partial theta functions are much easier to prove.

For the statements of our theorems, $x$ and $y$ are generic, i.e. $x,y,q$ do not cause poles in the Appell-Lerch sums or in the quotients of theta functions.   To state our results, we define the following expression:
\begin{align}
g_{a,b,c}&(x,y,q,z_1,z_0)\notag\\
:=&\sum_{t=0}^{a-1}(-y)^tq^{c\binom{t}{2}}j(q^{bt}x;q^a)m\Big (-q^{a\binom{b+1}{2}-c\binom{a+1}{2}-t(b^2-ac)}\frac{(-y)^a}{(-x)^b},q^{a(b^2-ac)},z_0\Big )\label{equation:mdef-2}\\
&+\sum_{t=0}^{c-1}(-x)^tq^{a\binom{t}{2}}j(q^{bt}y;q^c)m\Big (-q^{c\binom{b+1}{2}-a\binom{c+1}{2}-t(b^2-ac)}\frac{(-x)^c}{(-y)^b},q^{c(b^2-ac)},z_1\Big ).\notag
\end{align}

\begin{theorem}   \label{theorem:masterFnp} Let $n$ and $p$ be positive integers with $(n,p)=1$.  For generic $x,y\in \mathbb{C}^*$
\begin{align*}
f_{n,n+p,n}(x,y,q)=g_{n,n+p,n}(x,y,q,-1,-1)+\frac{1}{\overline{J}_{0,np(2n+p)}}\cdot \theta_{n,p}(x,y,q),
\end{align*}
where
\begin{align*}
&\theta_{n,p}(x,y,q):=\sum_{r^*=0}^{p-1}\sum_{s^*=0}^{p-1}q^{n\binom{r-(n-1)/2}{2}+(n+p)\big (r-(n-1)/2\big )\big (s+(n+1)/2\big )+n\binom{s+(n+1)/2}{2}}  (-x)^{r-(n-1)/2}\notag\\
 & \cdot   \frac{(-y)^{s+(n+1)/2}J_{p^2(2n+p)}^3j(-q^{np(s-r)}x^n/y^n;q^{np^2})j(q^{p(2n+p)(r+s)+p(n+p)}x^py^p;q^{p^2(2n+p)})}{j(q^{p(2n+p)r+p(n+p)/2}(-y)^{n+p}/(-x)^n,q^{p(2n+p)s+p(n+p)/2}(-x)^{n+p}/(-y)^n;q^{p^2(2n+p)})}.\notag
\end{align*}
Here $r:=r^*+\{(n-1)/2\}$ and $s:=s^*+\{ (n-1)/2\}$, with $0\le \{ \alpha\}<1$ denoting the fractional part of $\alpha$. 
\end{theorem}
\noindent For our secondary result, we consider the special case of (\ref{equation:fabc-def}) in which  $b$ is divisible by $a$ and $c$.  
\begin{theorem} \label{theorem:main-acdivb}Let $a,b,$ and $c$ be positive integers with $ac<b^2$ and $b$ divisible by $a$ and $c$. Then
\begin{align*}
& f_{a,b,c}(x,y,q)=h_{a,b,c}(x,y,q,-1,-1)-\frac{1}{\overline{J}_{0,b^2/a-c}\overline{J}_{0,b^2/c-a}}\cdot \theta_{a,b,c}(x,y,q),
\end{align*}
where 
\begin{align*}
h_{a,b,c}(x,y,q,z_1,z_0):&=j(x;q^a)m\Big( -q^{a\binom{b/a+1}{2}-c}{(-y)}{(-x)^{-b/a}},q^{b^2/a-c},z_1 \Big )\\
& \ \ \ \ \ \ +j(y;q^c)m\Big( -q^{c\binom{b/c+1}{2}-a}{(-x)}{(-y)^{-b/c}},q^{b^2/c-a},z_0 \Big ),
\end{align*}
and
\begin{align*}
&\theta_{a,b,c}(x,y,q):=\sum_{d=0}^{b/c-1}\sum_{e=0}^{b/a-1}\sum_{f=0}^{b/a-1}
q^{(b^2/a-c)\binom{d+1}{2}+(b^2/c-a)\binom{e+f+1}{2}+a\binom{f}{2}}j\big (q^{(b^2/a-c)(d+1)+bf}y,q^{b^2/a}\big ) \\
&\ \ \ \ \ \ \ \ \ \  \cdot(-x)^{f} j\big (q^{b(b^2/(ac)-1)(e+f+1)-(b^2/a-c)(d+1)+b^3(b-a)/(2a^2c)}(-x)^{b/a}y^{-1};q^{(b^2/a)(b^2/(ac)-1)}\big ) \\
& \cdot \frac{J_{b(b^2/(ac)-1)}^3j\big (q^{ (b^2/c-a)(e+1)+(b^2/a-c)(d+1)-c\binom{b/c}{2}-a\binom{b/a}{2}}(-x)^{1-b/a}(-y)^{1-b/c};q^{b(b^2/(ac)-1)}\big )}
{j\big (q^{(b^2/c-a)(e+1)-c\binom{b/c}{2}}(-x)(-y)^{-b/c},q^{(b^2/a-c)(d+1)-a\binom{b/a}{2}}(-x)^{-b/a}(-y);q^{b(b^2/(ac)-1)}\big )}.
\end{align*}

\end{theorem}
\noindent Although all eight Hecke-type double sum identities found in Kac and Peterson \cite{KP} and Polishchuk \cite{P1} are special cases of the above theorems, we will only demonstrate one namely \cite[$(5.22)$]{KP}:
\begin{gather}
\sum_{\substack{k,\ell\in \mathbb{Z}\\2k\ge \ell\ge 0}}(-1)^{k}q^{[5(2k+1)^2-(2\ell+1)^2]/4}=q\prod_{n\ge 1}(1-q^{4n})(1-q^{20n}).\label{equation:KP-(5.22)}
\end{gather}
\begin{corollary}  \label{corollary:cor-P1-KP} Identity  (\ref{equation:KP-(5.22)}) is a special case of Theorem \ref{theorem:main-acdivb}.
\end{corollary}

We compute three examples using Theorem \ref{theorem:masterFnp} with the $n=1$, $p=1$ specialization:
\begin{align}
f_{1,2,1}(x,y,q)&=j(y;q)m\big (\tfrac{q^2x}{y^2},q^3,-1\big )+j(x;q)m\big (\tfrac{q^2y}{x^2},q^3,-1\big )
- \frac{yJ_{3}^3j(-x/y;q)j(q^2xy;q^3)}
{\overline{J}_{0,3}j(-qy^2/x,-qx^2/y;q^3)}.\label{equation:n1p1}
\end{align}
The first two examples will involve two of Ramanujan's sixth order mock theta functions:
\begin{equation*}
\phi(q):=\sum_{n\ge 0}\frac{(-1)^nq^{n^2}(q;q^2)_n}{(-q)_{2n}} \ \ {\text{  and  }}\ \ 
\sigma(q):=\sum_{n\ge 0}\frac{q^{\binom{n+2}{2}}(-q)_n}{(q;q^2)_{n+1}}.
\end{equation*}
\begin{example}We recall two expressions for $\phi(q)$ obtained via constant term and Bailey pair techniques respectively \cite[$(0.10)_R$]{AH}, \cite[$(2.19)$]{AH}:
\begin{align}
&\phi(q)\cdot \prod_{n\ge1}(1-q^{3n})(1+q^{3n-1})(1+q^{3n-2})=2\sum_{n}\frac{q^{n(3n+1)}}{1+q^{3n}},\label{equation:phiex1a}\\
&\phi(q)\cdot \prod_{n\ge1}(1-q^{4n})(1+q^{4n-1})(1+q^{4n-3}) =\sum_{n}(-1)^nq^{3n^2+n}\sum_{j=-n}^{n}(-1)^jq^{-j^2}.\label{equation:phiex1b}
\end{align}
It is straightforward to rewrite (\ref{equation:phiex1a}) and (\ref{equation:phiex1b}) in our notation as
\begin{align*}
\phi(q)=&2m(q,q^3,-1) \ \ {\text{ and }}\ \ \overline{J}_{1,4}\cdot \phi(q)=f_{1,2,1}(q,-q,q).
\end{align*}
Using (\ref{equation:n1p1}), it is easy to translate from the Hecke form to the Appell-Lerch form,
\begin{align}
f_{1,2,1}(q,-q,q)&=j(-q;q)m(q,q^3,-1)+j(q;q)m(-q,q^3,-1)+ \frac{qJ_{3}^3j(1;q)j(-q^4;q^3)}
{\overline{J}_{0,3}j(-q^2,q^2;q^3)}\notag\\
&=j(-q;q)m(q,q^3,-1)+0+0=2\overline{J}_{1,4}m(q,q^3,-1).\label{equation:phi-lerch}
\end{align}
\end{example}

\begin{example}   From \cite[$(2.21)$]{AH} we know that $J_{1,2}\sigma(q)=qf_{1,2,1}(q^4,q^3,q^2).$  Using (\ref{equation:n1p1}) gives
\begin{align}
qf_{1,2,1}(q^4,q^3,q^2)&=-J_{1,2}m(q^2,q^6,-1)+0
+\frac{J_6^3\overline{J}_{1,2}J_{5,6}}{\overline{J}_{0,6}\overline{J}_{4,6}\overline{J}_{1,6}}.\label{equation:psi-lerch}
\end{align}
Using  (\ref{equation:phi-lerch}) and (\ref{equation:psi-lerch}), we immediately obtain
\begin{align*}
\phi(q^2)+2\sigma(q)&=2m(q^2,q^6,-1)-2m(q^2,q^6,-1)+\frac{2}{J_{1,2}}\cdot \frac{J_6^3\overline{J}_{1,2}J_{5,6}}{\overline{J}_{0,6}\overline{J}_{4,6}\overline{J}_{1,6}}\\
&=\prod_{n\ge 1}(1+q^{2n-1})^2(1-q^{6n})(1+q^{6n-3})^2,
\end{align*}
 a sixth order mock theta function identity \cite[$(0.19)_R$]{AH}.  For the general mock theta function identity, use Theorem \ref{theorem:masterFnp} and Theorem \ref{theorem:msplit-general-n} to line up all of the Appell-Lerch sum expressions such that they cancel. Then all that remains to be verified is a theta function identity, which can be done by classical means. 
\end{example}
\begin{example}  We discuss the integral-level string functions associated with admissable representations of the affine Lie Algebra $A_1^{(1)}$ \cite{KW1,KW2}.  Here $m\in \mathbb{Z}$, $N\in\mathbb{N}$,  $\ell \in \{ 0,1,2,\dots, N\}$, $m\equiv \ell \pmod 2$, where $N$ is the level.  In \cite[p. 236]{SW}, one finds the Hecke-type  form for the general integral-level string function:
\begin{align}
C_{m,\ell}^N(q)&=\frac{1}{(q)_{\infty}^3}\Big \{ \sum_{\substack{j\ge 1\\ k\le 0}} - \sum_{\substack{j\le 0 \\ k\ge 1}} \Big \}(-1)^{k-j}q^{\binom{k-j}{2}-Njk+\tfrac{1}{2}k(m-\ell)+\tfrac{1}{2}j(m+\ell)}\notag \\
&=\frac{1}{J_1^3}\cdot f_{1,1+N,1}(q^{1+\tfrac{1}{2}(m+\ell)},q^{1-\tfrac{1}{2}(m-\ell)},q),
\label{equation:integer-string}
\end{align}
which can be evaluated  with Theorem \ref{theorem:masterFnp}. The last equality in (\ref{equation:integer-string}) follows from the substitutions $k\rightarrow -k$, $j\rightarrow j+1$ and the identity $f_{a,b,c}(x,y,q)=-yf_{a,b,c}(q^bx,q^cy,q)+j(x;q^a).$  

   Let us set $N=1$.  Here $\ell\in \{0,1\}$.  Using specialization (\ref{equation:n1p1}),
\begin{align*}
C_{m,\ell}^1(q)&=\frac{1}{J_1^3}\cdot f_{1,2,1}(q^{1+\tfrac{1}{2}(m+\ell)},q^{1-\tfrac{1}{2}(m-\ell)},q)\\
&=\frac{1}{J_1^3}\cdot \Big [ j(q^{1-\tfrac{1}{2}(m-\ell)};q)m(q^{(3m-\ell)/2},q^{3},-1)+j(q^{1+\tfrac{1}{2}(m+\ell)};q)m(q^{-(3m+\ell)/2},q^{3},-1)\\
&\ \ \ \ -\frac{1}{\overline{J}_{0,3}}\cdot \frac{q^{1-\frac{1}{2}(m-\ell)}J_3^3j(-q^m;q)j(q^{4+\ell};q^3)}{j(q^{2-\frac{1}{2}(3m-\ell)},q^{2+\frac{1}{2}(3m+\ell)};q^3)}\Big ].
\end{align*}
Both Appell-Lerch sums are defined, and their theta function coefficients are zero.  Hence the Appell-Lerch sum expression is zero, and only the theta function quotient remains.  By considering the cases $\ell={0,1}$, it is easy to show the well-known form of the level-1 string function,  \cite[Section $4.6$, Ex. 3]{KP}:
\begin{equation*}
C_{m,\ell}^1(q)=\frac{q^{\frac14(m^2-\ell^2)}}{(q)_{\infty}}.
\end{equation*}
\end{example}

In Theorem \ref{theorem:masterFnp}, we set $z_1=z_0=-1$ in the Appell-Lerch expression (\ref{equation:mdef-2}).  For other examples where $p=1,2,3,$ or $4$, we can set $z_1=z_0^{-1}=y^n/x^n$ in order to reduce the number of theta quotients.    

\begin{theorem} \label{theorem:genfn1} Let $n$ be a positive integer.  For generic $x,y\in \mathbb{C}^{*}$
\begin{align*}
f_{n,n+1,n}(x,y,q)=g_{n,n+1,n}(x,y,q,y^n/x^n,x^n/y^n).
\end{align*}
\end{theorem}

\begin{corollary} \label{corollary:f232-tenth}Theorem \ref{theorem:genfn1} with $n=2$ yields the Appell-Lerch sum representations of Section \ref{section:known} for the tenth order mock theta functions $\phi(q)$, $\psi(q)$, $X(q)$, $\chi(q)$.  
\end{corollary}
\begin{corollary} \label{corollary:f343-seventh}Theorem \ref{theorem:genfn1} with $n=3$ yields the Appell-Lerch sum representations of Section \ref{section:known} for the seventh order mock theta functions ${\mathcal{F}}_0(q)$, ${\mathcal{F}}_1(q)$, ${\mathcal{F}}_2(q)$.
\end{corollary}

\begin{theorem} \label{theorem:genfn2} Let $n$ be a positive odd integer.  For generic $x,y\in\mathbb{C}^*$
\begin{align*}
f_{n,n+2,n}(x,y,q)=g_{n,n+2,n}(x,y,q,y^n/x^n,x^n/y^n)-\Theta_{n,2}(x,y,q),
\end{align*}
where
\begin{align*}
\Theta_{n,2}(x,y,q):=\frac{y^{\frac{n+1}{2}}J_{2n,4n}J_{4(n+1),8(n+1)}j(y/x,q^{n+2}xy;q^{4(n+1)})j(q^{2n}/x^2y^2;q^{8(n+1)})}{q^{\frac{n^2-3}{2}}x^{\frac{n-3}{2}}j(y^n/x^n;q^{4n(n+1)})j(-q^{n+2}x^2,-q^{n+2}y^2;q^{4(n+1)})}.
\end{align*}
\end{theorem}


\begin{theorem} \label{theorem:genfn3} Let $n$ be a positive integer with $(n,3)=1$.  For generic $x,y\in\mathbb{C}^{*}$
\begin{align*}
f_{n,n+3,n}(x,y,q)=g_{n,n+3,n}(x,y,q,y^n/x^n,x^n/y^n)-\Theta_{n,3}(x,y,q),
\end{align*}
where
\begin{align*}
\Theta_{n,3}&(x,y,q):=\frac{q^{n\binom{n+1}{2}}(-x)(-y)^{n}J_{3n}J_{3(2n+3)}j(y/x;q^{3(2n+3)})j(q^{n^2+n}x,q^{n^2+n}y;q^{2n+3})}
 {J_{2n+3}^2j(y^{n}/x^{n};q^{3n(2n+3)})j(q^{3n^2+3n}x^3,q^{3n^2+3n}y^3;q^{3(2n+3)})}\\
&\cdot  \Big \{ j(q^{3n^2+5n+3}x^2y,q^{3n^2+5n+3}xy^2;q^{3(2n+3)})\\
&\ \ \ \ \  -{q^{2n^2+2n}}{xy}j(q^{3n^2+7n+6}x^2y,q^{3n^2+7n+6}xy^2;q^{3(2n+3)})\Big \}.
\end{align*}

\end{theorem}


\begin{theorem}  \label{theorem:genfn4} Let $n$ be a positive odd integer. For generic $x,y\in\mathbb{C}^{*}$
\begin{align*}
f_{n,n+4,n}(x,y,q)=g_{n,n+4,n}(x,y,q,y^n/x^n,x^n/y^n)-\Theta_{n,4}(x,y,q),
\end{align*}
where
\begin{align*}
\Theta_{n,4}(x,y,q):=&\frac{q^{-(n^2+n-3)}x^{-(n-3)/2}y^{(n+1)/2}j(y/x;q^{4(2n+4)})}{j(y^n/x^n;q^{4n(2n+4)})j(-q^{2n+8}x^4,-q^{2n+8}y^4;q^{4(2n+4)})}\Big \{ J_{4n,16n}  S_1-qJ_{8n,16n} S_2\Big \},
\end{align*}
with
\begin{align*}
S_1:&=\frac{j(q^{6n+16}x^2y^2,-q^{2(2n+4)}y/x;q^{4(2n+4)})j(q^{n+4}xy;q^{2(2n+4)})}{J_{2(2n+4)}^3J_{8(2n+4)}} \\
& \cdot  \Big \{j(-q^{2n+8}x^2y^2,q^{2(2n+4)}y^2/x^2;q^{4(2n+4)})J_{4(2n+4)}^2\\
&\ \ \ \ +\frac{q^{n+4}x^2 j(-q^{6n+16}x^2y^2;q^{4(2n+4)})j(q^{2(2n+4)}y/x,-y/x;q^{4(2n+4)})^2}{J_{4(2n+4)}}\Big \},\\
S_2:&=\frac{j(q^{2n+8}x^2y^2,-y/x;q^{4(2n+4)})j(q^{3n+8}xy;q^{2(2n+4)})}{J_{2(2n+4)}^2} \\
& \cdot  \Big \{\frac{q^{n+1}j(-q^{2n+8}x^2y^2,q^{2(2n+4)}y^2/x^2;q^{4(2n+4)})J_{8(2n+4)}}{yJ_{4(2n+4)} }\\ 
&\ \ \ \ +\frac{qxj(-q^{6n+16}x^2y^2;q^{4(2n+4)})j(q^{4(2n+4)}y^2/x^2;q^{8(2n+4)})^2}{J_{8(2n+4)} }\Big \}.
\end{align*}
\end{theorem}


\begin{corollary} \label{corollary:f373-fifth} Theorem \ref{theorem:genfn4} with $n=3$  yields the Appell-Lerch sum representations of Section \ref{section:known} for the fifth order mock theta functions $f_0(q), f_1(q), F_0(q), F_1(q)$.
\end{corollary}
Corollary \ref{corollary:f343-seventh} gives a new proof of Hickerson's identities for the seventh order functions \cite{H2}.  Corollary  \ref{corollary:f373-fifth} gives a new proof of the mock theta conjectures, which were first proved in \cite{H1}.

 We note that Theorem \ref{theorem:genfn1} with $n=1$  proves the  Appell-Lerch sum representations of Section \ref{section:known} for the sixth order mock theta functions $\phi(q)$, $\psi(q)$, $\rho(q)$, $\sigma(q)$; Theorem \ref{theorem:genfn2} with $n=1$ proves representations for the three second orders  as well as those for the eight eighth  orders; Theorem \ref{theorem:genfn3} with $n=1$ yields the representations for the fifth order mock theta functions $\psi_0(q)$, $\psi_1(q),$ $\phi_0(q),$ $\phi_1(q)$; Theorem \ref{theorem:genfn4} with $n=1$ can proves the representations for the sixth orders  $\lambda(q)$, $\mu(q)$, $\phi\_(q)$, ${\psi}\_(q)$. 

This paper is organized as follows.  In Section \ref{section:prelim}, we recall from the literature several theta function identities and other useful facts.  In Section \ref{section:prop-mxqz}, we recall known Appell-Lerch sum properties and prove new ones.  In Section \ref{section:genlambert}, we write generalized Lambert series, which have been used to express mock theta functions, in terms of Appell-Lerch sums.  We then use this information in Section \ref{section:known} to write the classical mock theta functions, as well as other functions found in the lost notebook, in terms of the Appell-Lerch sums.  In Section \ref{section:maintheorems-proofs}, we prove Theorems  \ref{theorem:masterFnp} and \ref{theorem:main-acdivb}.   We prove the four subtheorems in Section \ref{section:thm-proofs}.   In Section \ref{section:P1-KP-corollaries}, we prove Corollary \ref{corollary:cor-P1-KP}, and in Section \ref{section:cor-proofs}, we prove the corollaries to the four subtheorems.


\section{Preliminaries}\label{section:prelim}
We will frequently use the following identities without mention.  They easily follow from the definitions.
\begin{subequations}
\begin{gather}
\overline{J}_{0,1}=2\overline{J}_{1,4}=\frac{2J_2^2}{J_1},  \overline{J}_{1,2}=\frac{J_2^5}{J_1^2J_4^2},   J_{1,2}=\frac{J_1^2}{J_2},   \overline{J}_{1,3}=\frac{J_2J_3^2}{J_1J_6}, \notag\\
J_{1,4}=\frac{J_1J_4}{J_2},   J_{1,6}=\frac{J_1J_6^2}{J_2J_3},   \overline{J}_{1,6}=\frac{J_2^2J_3J_{12}}{J_1J_4J_6}.\notag
\end{gather}
\end{subequations}
Also following from the definitions are the following general identities:
\begin{subequations}
\begin{gather}
j(q^n x;q)=(-1)^nq^{-\binom{n}{2}}x^{-n}j(x;q), \ \ n\in\mathbb{Z},\label{equation:1.8}\\
j(x;q)=j(q/x;q)=-xj(x^{-1};q)\label{equation:1.7},\\
j(-x;q)={J_{1,2}j(x^2;q^2)}/{j(x;q)} \ \ {\text{if $x$ is not an integral power of $q$,}}\label{equation:1.9}\\
j(x;q)={J_1}j(x,qx,\dots,q^{n-1}x;q^n)/{J_n^n} \ \ {\text{if $n\ge 1$,}}\label{equation:1.10}\\
j(x;-q)={j(x;q^2)j(-qx;q^2)}/{J_{1,4}},\label{equation:1.11}\\
j(z;q)=\sum_{k=0}^{m-1}(-1)^k q^{\binom{k}{2}}z^k
j\big ((-1)^{m+1}q^{\binom{m}{2}+mk}z^m;q^{m^2}\big ),\label{equation:jsplit}\\
j(x^n;q^n)={J_n}j(x,\zeta_nx,\dots,\zeta_n^{n-1}x;q^n)/{J_1^n} \ \ {\text{if $n\ge 1$.}}\label{equation:1.12}
\end{gather}
\end{subequations}
\noindent  Here, $\zeta_n$ an $n$-th primitive root of unity. We recall the classical partial fraction  expansion for the reciprocal of Jacobi's theta product
\begin{equation}
\sum_n\frac{(-1)^nq^{\binom{n+1}{2}}}{1-q^nz}=\frac{J_1^3}{j(z;q)},\label{equation:Reciprocal}
\end{equation}
\noindent where $z$ is not an integral power of $q$.  A convenient form of the Riemann relation for theta functions is,
\begin{proposition}\label{proposition:CHcorollary} For generic $a,b,c,d\in \mathbb{C}^*$
\begin{align*}
j(ac,a/c,bd,b/d;q)=j(ad,a/d,bc,b/c;q)+b/c \cdot j(ab,a/b,cd,c/d;q).
\end{align*}
\end{proposition}

We collect several useful results about theta functions in terms of a proposition.  Among other places, these can be found in  \cite{H1}, \cite{H2}, and \cite{AH}.
\begin{proposition}   For generic $x,y,z\in \mathbb{C}^*$ 
 \begin{subequations}
\begin{gather}
j(qx^3;q^3)+xj(q^2x^3;q^3)=j(-x;q)j(qx^2;q^2)/J_2={J_1j(x^2;q)}/{j(x;q)},\label{equation:H1Thm1.0}\\
j(x;q)j(y;q)=j(-xy;q^2)j(-qx^{-1}y;q^2)-xj(-qxy;q^2)j(-x^{-1}y;q^2),\label{equation:H1Thm1.1}\\
j(-x;q)j(y;q)-j(x;q)j(-y;q)=2xj(x^{-1}y;q^2)j(qxy;q^2),\label{equation:H1Thm1.2A}\\
j(-x;q)j(y;q)+j(x;q)j(-y;q)=2j(xy;q^2)j(qx^{-1}y;q^2),\label{equation:H1Thm1.2B}\\
\frac{J_1^3j(xz;q)j(x^n;q^n)}{J_n^3j(x;q)j(z;q)}=\sum_{k=0}^{n-1}\frac{x^kj(q^kx^nz;q^n)}{j(q^kz;q^n)},
\label{equation:ThmAH1.1}\\
j(x;q)j(y;q^n)=\sum_{k=0}^n(-1)^kq^{\binom{k}{2}}x^kj\big ((-1)^nq^{\binom{n}{2}+kn}x^ny;q^{n(n+1)}\big )j\big (-q^{1-k}x^{-1}y;q^{n+1} \big ).\label{equation:Thm1.3AH6}
\end{gather}
\end{subequations}
\end{proposition}

\noindent Identity (\ref{equation:H1Thm1.0}) is the quintuple product identity. 

The next proposition follows immediately from \cite[Lemma $2$]{ASD} see also \cite[Theorem $1.7$]{H1}.
\begin{proposition}\label{proposition:H1Thm1.7} Let $C$ be a nonzero complex number, and let $n$ be a nonnegative integer.  Suppose that $F(z)$ is analytic for $z\ne 0$ and satisfies $F(qz)=Cz^{-n}F(z)$.  Then either $F(z)$ has exactly $n$ zeros in the annulus $|q|<|z|\le 1$ or $F(z)=0$ for all $z$.
\end{proposition}

The next two propositions involve computing residues.  Because the proofs are straightforward, they have been omitted.  The first proposition is \cite[Theorem 1.3]{H1}.
\begin{proposition} \label{proposition:H1Thm1.3} Define $G(z):={1}/{j(\beta z^b;q^m)}.$  $G(z)$ is meromorphic for $z\ne 0$ with simple poles at points $z_0$ such that $z_0^b=q^{km}/\beta$.  The residue at such $z_0$ is ${(-1)^{k+1}q^{m\binom{k}{2}}z_0}/{bJ_m^3}$.
\end{proposition}

\begin{proposition}\label{proposition:mnp2Lerch-residue}
Define
\begin{equation}
G(x):=\sum_{m}\frac{(-1)^mq^{\binom{m}{2}}z^m}{1-q^{m-1}\beta x^{-p}}.
\end{equation}
$G(x)$ is meromorphic with simple poles at points $x_0$ such that $x_0^p=\beta q^k$.  The residue at such an $x_0$ is ${(-1)^{k+1}q^{\binom{k+1}{2}}z^{k+1} x_0}/{p}.$
\end{proposition}

We present three identities, which appear to be new.

\begin{proposition}\label{proposition:prop-f141} For $\omega$ a primitive cube root of unity 
\begin{equation}
\frac{1}{J_3}\cdot j(\omega^2y;q)j(qy;q^3)j(y;q^3)=\omega yj(y;q^3)j(q^2y^2;q^3)+j(qy;q^3)j(y^2;q^3).
\end{equation}
\end{proposition}

\begin{proof}[Proof of Proposition \ref{proposition:prop-f141}]  We use (\ref{equation:jsplit}) to write
\begin{equation}
j(\omega^2y;q)=j(q^3y^3;q^9)-\omega^2yj(q^6y^3;q^9)+\omega qy^2 j(q^9y^3;q^9).
\end{equation}
\noindent In the middle term use $-\omega^2=1+\omega,$ to obtain
\begin{equation}
j(\omega^2y;q)=\big [j(q^3y^3;q^9)+yj(q^6y^3;q^9)\big ]
+\omega y\big [j(q^6y^3;q^9)+qyj(q^9y^3;q^9)\big ].
\end{equation}
\noindent Each bracketed term can be rewritten using the quintuple product identity (\ref{equation:H1Thm1.0}), resulting in
\begin{equation*}
j(\omega^2y;q)=\frac{J_3j(y^2;q^3)}{j(y;q^3)}+\omega y \frac{J_3j(q^2y^2;q^3)}{j(qy; q^3)}.\qedhere
\end{equation*}
\end{proof}

\begin{proposition}\label{proposition:prop-f373eqs}  For $x\in\mathbb{C}^*$ 
\begin{align}
\overline{J}_{5,20}j(-q^{6}x;q^{30})j(-q^3x^3;q^{30})-&q^2x^2\overline{J}_{30,120}j(-qx;q^{5})j(-q^{27}x^2,q^{30})\label{equation:f0f1}\\
&=j(q^2x^2;q^{20})j(q^6x;q^{15})j(q^{12}x^2;q^{60}),\notag\\
J_{10,20}j(q^{21}x;q^{30})j(-q^3x^3;q^{30})+&qxJ_{15,30}j(q^2x^2;q^{20})j(-q^{27}x^2;q^{30})\label{equation:F0F1}\\
&=j(-qx;q^{5})j(q^6x;q^{15})j(q^{12}x^2;q^{60}).\notag
\end{align}
\end{proposition}
\begin{proof}[Proof of Proposition \ref{proposition:prop-f373eqs}]  We prove the first identity and omit the proof of the second, which is similar.  Denote by $f(x)$ the difference between the left and right hand sides of (\ref{equation:f0f1}).  For fixed $q$, all three terms are analytic for $x\ne 0$.  It is straightforward to verify that $f(x)$ satisfies the functional equation $f(q^{30}x)=q^{-105}x^{-10}f(x).$  By Proposition \ref{proposition:H1Thm1.7}, $f(x)$ has either exactly ten zeros in the annulus $|q|^{30}<|x|<1$ or it equals zero for all $x$.  But there are at least eleven such values of $x$ for which at least one of the terms in $f(x)$ vanishes:  $\pm q^{3/2}$, $-q^4$, $q^9$, $-q^{14}$, $\pm q^{33/2}$, $q^{19}$, $\pm q^{24}$, $q^{29}$.  Verifying $f(x)$ vanishes for each of these just involves proving that the remaining two theta products sum to zero, which is easy.
\end{proof}

We finish with Hartog's Theorem:

\begin{theorem}\cite[p. 7]{GH} A holomorphic function on the complement of a point in an open set $U\subset \mathbb{C}^n$ ($n>1$) extends to a holomorphic function in all of $U$.
\end{theorem}


\section{Properties of the Appell-Lerch sum $m(x,q,z)$}\label{section:prop-mxqz}

Changing $r$ to $r+1$ in $(\ref{equation:mxqz-def})$ gives another useful form for $m(x,q,z)$:
\begin{equation}
m(x,q,z)=\frac{-z}{j(z;q)}\sum_r\frac{(-1)^rq^{\binom{r+1}{2}}z^r}{1-q^{r}xz}.\label{equation:mxqz-altdef}
\end{equation}

The Appell-Lerch sum $m(x,q,z)$ satisfies several functional equations and identities, which we collect in the form of a proposition.  The proofs are straightforward and will be omitted.  A list of Appell-Lerch sum properties with proofs can be found in \cite{Zw}.

\begin{proposition}  For generic $x,z\in \mathbb{C}^*$
{\allowdisplaybreaks \begin{subequations}
\begin{gather}
m(x,q,z)=m(x,q,qz),\label{equation:mxqz-fnq-z}\\
m(x,q,z)=x^{-1}m(x^{-1},q,z^{-1}),\label{equation:mxqz-flip}\\
m(qx,q,z)=1-xm(x,q,z),\label{equation:mxqz-fnq-x}\\
m(x,q,z)=1-q^{-1}xm(q^{-1}x,q,z)\label{equation:mxqz-altdef0},\\
m(x,q,z)=x^{-1}-x^{-1}m(qx,q,z). \label{equation:mxqz-altdef1}
\end{gather}
\end{subequations}}
\end{proposition}
Some simple evaluations of the Appell-Lerch sum follow.

\begin{corollary} \label{corollary:mxqz-eval} We have
\begin{align}
m(q,q^2,-1)&=1/2,\label{equation:mxqz-eval-a}\\
m(-1,q^2,q)&=0. \label{equation:mxqz-eval-b}
\end{align}
\end{corollary}

\begin{proof} [Proof of Corollary \ref{corollary:mxqz-eval}]  The first identity is a straightforward application of identities (\ref{equation:mxqz-flip}) and (\ref{equation:mxqz-altdef1}); the second is a straightforward consequence of identities (\ref{equation:mxqz-flip}) and (\ref{equation:mxqz-fnq-z}).
\end{proof}

We now introduce a heuristic point of view which will guide our further study of $m(x,q,z)$.  This heuristic leads us to new Appell-Lerch sum properties such as Theorem \ref{theorem:msplit-general-n} and also guides us to the Appell-Lerch sum expressions of Theorems \ref{theorem:masterFnp} and \ref{theorem:main-acdivb}.  If we iterate (\ref{equation:mxqz-altdef0}), we obtain
\begin{align}
m(x,q,z)&=1-q^{-1}xm(q^{-1}x,q,z)\notag\\
&=1-q^{-1}x+q^{-3}x^2m(q^{-2}x,q,z)\notag\\
&=1-q^{-1}x+q^{-3}x^2-q^{-6}x^3m(q^{-3}x,q,z)\notag\\
& \ \ \ \vdots\notag\\
&\sim 1-q^{-1}x+q^{-3}x^2-q^{-6}x^3+q^{-10}x^4-\dots \ ;\notag
\end{align}
that is,
\begin{equation}
m(x,q,z)\sim \sum_{r\ge 0}(-1)^rq^{-\binom{r+1}{2}}x^r.\label{equation:mxqz-heuristic}
\end{equation}
Of course, we cannot use an equal sign here, since the infinite series on the right diverges for $|q|<1$.  However, it is often useful to think of $m(x,q,z)$ as a partial theta series with $q$ replaced by $q^{-1}$.  

Roughly speaking, we may think of ``$\sim$'' as congruence `mod theta'.  For example, since the series (\ref{equation:mxqz-heuristic}) does not depend on $z$, we may write
\begin{equation}
m(x,q,z_0)\sim m(x,q,z_1).
\end{equation}
In fact, the difference between these two quantities is a theta function, as we see in the following well-known theorem.

\begin{theorem} \label{theorem:changing-z-theorem}  For generic $x,z_0,z_1\in \mathbb{C}^*$
\begin{equation}
m(x,q,z_1)-m(x,q,z_0)=\frac{z_0J_1^3j(z_1/z_0;q)j(xz_0z_1;q)}{j(z_0;q)j(z_1;q)j(xz_0;q)j(xz_1;q)}.
\end{equation}
\end{theorem}

\begin{corollary}\label{corollary:mxqz-flip-xz}  For generic $x,z\in \mathbb{C}^*$ 
\begin{equation}
m(x,q,z)=m(x,q,x^{-1}z^{-1}).
\end{equation}
\end{corollary}
  Let us break the sum (\ref{equation:mxqz-heuristic}) into two parts, depending on the parity of $r$.  We obtain
\begin{align}
m(x,q,z)&\sim \sum_{r\ge 0}(-1)^rq^{-\binom{r+1}{2}}x^r\label{equation:H5.0}\\
&\sim \sum_{r\ge 0}q^{-\binom{2r+1}{2}}x^{2r}-\sum_{r\ge 0}q^{-\binom{2r+2}{2}}x^{2r+1}\notag\\
&\sim \sum_{r\ge 0}(-1)^rq^{-4\binom{r+1}{2}}(-qx^2)^r-q^{-1}x\sum_{r\ge 0}(-1)^rq^{-4\binom{r+1}{2}}(-q^{-1}x^2)^r\notag\\
&\sim m(-qx^2,q^4,z_0)-q^{-1}xm(-q^{-1}x^2,q^4,z_1).\notag
\end{align}
More generally, if we break the sum into $n$ parts depending on the value of $r$ mod $n$, we find
\begin{equation}
m(x,q,z)\sim \sum_{r=0}^{n-1}(-1)^rq^{-\binom{r+1}{2}}x^rm\big ( -q^{\binom{n}{2}-nr}(-x)^n,q^{n^2},z_r\big ).\label{equation:msplit-heuristc}
\end{equation}
So we expect the difference between the two sides to be a theta function.  We see that in the next theorem, which to the best of our knowledge is new.

\begin{theorem} \label{theorem:msplit-general-n} For generic $x,z,z'\in \mathbb{C}^*$ 
\begin{align*}
m(&x,q,z) = \sum_{r=0}^{n-1} q^{{-\binom{r+1}{2}}} (-x)^r m\big(-q^{{\binom{n}{2}-nr}} (-x)^n, q^{n^2}, z' \big)\notag\\
& + \frac{z' J_n^3}{j(xz;q) j(z';q^{n^2})}  \sum_{r=0}^{n-1}
\frac{q^{{\binom{r}{2}}} (-xz)^r
j\big(-q^{{\binom{n}{2}+r}} (-x)^n z z';q^n\big)
j(q^{nr} z^n/z';q^{n^2})}
{ j\big(-q^{{\binom{n}{2}}} (-x)^n z', q^r z;q^n\big )}.
\end{align*}
\end{theorem}
\noindent Identity (\ref{equation:1.8}) easily yields the $n$ even and $n$ odd specializations:
\begin{corollary} \label{theorem:msplit-general-nparity} Let $n$ be a positive odd integer.   For generic $x,z,z'\in \mathbb{C}^*$ 
\begin{align}
m(x,&q,z)=\sum_{r=0}^{n-1}q^{-\binom{r+1}{2}}(-x)^rm\Big (q^{\binom{n}{2}-nr}x^n,q^{n^2},z'\Big )\\
&+\frac{z'J_n^3}{j(xz;q)j(z';q^{n^2})}\sum_{r=0}^{n-1}
\frac{q^{r(r-n)/2}(-x)^rz^{r-(n-1)/2}j(q^rx^nzz';q^n)j(q^{nr}z^n/z';q^{n^2})}{j(x^nz',q^rz;q^n)}.\notag
\end{align}
Let $n$ be a positive even integer.   For generic $x,z,z'\in \mathbb{C}^*$ 
\begin{align}
&m(x,q,z)=\sum_{r=0}^{n-1}q^{-\binom{r+1}{2}}(-x)^rm\Big (-q^{\binom{n}{2}-nr}x^n,q^{n^2},z'\Big )\\
&+\frac{z'J_n^3}{j(xz;q)j(z';q^{n^2})}\sum_{r=0}^{n-1}
\frac{q^{r(r-n+1)/2}(-x)^rz^{r+1-n/2}j(-q^{r+n/2}x^nzz';q^n)j(q^{nr}z^n/z';q^{n^2})}{j(-q^{n/2}x^nz',q^rz;q^n)}.\notag
\end{align}
\end{corollary}

\noindent For special values of $x$, $q$, $z$ and $z'$, the sum in Theorem \ref{theorem:msplit-general-n} reduces to a single quotient of theta functions.  Here are two useful examples: 
\begin{corollary}  \label{corollary:msplit-n=2} For generic $x,z\in \mathbb{C}^*$
\begin{equation}
m(x,q,z)=m\big (-qx^2,q^4,z^4\big )-\frac{x}{q}m\big (-\frac{x^2}{q},q^4,z^4\big)-\frac{J_2J_4j(-xz^2;q)j(-xz^3;q)}{xj(xz;q)j(z^4;q^4)j(-qx^2z^4;q^2)}.\label{equation:msplit-n=2}
\end{equation}
\end{corollary}

\begin{corollary} \label{corollary:msplit-n=3} For generic $x\in \mathbb{C}^*$
\begin{equation}
m(x,q,-1)=m\big (q^3x^3,q^9,-1\big )-\frac{x}{q}m\big (x^3,q^9,-1\big )+\frac{x^2}{q^{3}}m\big (\frac{x^3}{q^{3}},q^9,-1\big )
+\frac{xJ_1J_3^2J_6J_9j(qx^2;q^2)}{2qJ_2^2J_{18}^2j(-x^3;q^3)}.\label{equation:msplit-n=3}
\end{equation}
\end{corollary}
Before we prove Theorem \ref{theorem:msplit-general-n} as well as Corollaries \ref{corollary:msplit-n=2} and \ref{corollary:msplit-n=3}, we establish an intermediate result, which also appears to be new.

\begin{theorem} \label{theorem:rootsof1}
Let $n$ and $k$ be integers with $0 \leq k < n$. Let $\omega$ be a primitive $n$-th root of unity. Then
\begin{align}
\sum_{t=0}^{n-1}& \omega^{-kt} m(\omega^t x,q,z) =  n q^{-\binom{k+1}{2}} (-x)^k m\big(-q^{{\binom{n}{2}-nk}} (-x)^n, q^{n^2}, z' \big)\label{equation:rootsof1}\\
&-\frac{n x^k z^{k+1} J_{n^2}^3} {j(z;q)j(z';q^{n^2})}\cdot\sum_{t=0}^{n-1} \frac{q^{{\binom{t+1}{2}+kt}} (-z)^t j\big(-q^{{\binom{n+1}{2}+nk+nt}} (-z)^n/z';q^{n^2}\big)
j(q^{nt} x^n z^n z';q^{n^2})}
{j\big(-q^{{\binom{n}{2}-nk}} (-x)^n z', q^{nt} x^n z^n;q^{n^2}\big)}.\notag
\end{align}
\end{theorem}
\noindent To prove this, we need the partial fraction decomposition of $x^k/(1-x^n)$:
\begin{lemma}\label{lemma:rootsof1}
Let $n$ and $k$ be integers with $0 \leq k < n$. Let $\omega$ be a primitive $n$-th root of unity, and
suppose that $x^n \neq 1$. Then
\begin{equation}\label{parfrac}
\frac{x^k}{1-x^n} = \frac{1}{n}\sum_{t=0}^{n-1} \frac{\omega^{-kt}}{1-\omega^t x}.
\end{equation}
\end{lemma}
\begin{proof}[Proof of Lemma \ref{lemma:rootsof1}]
Since $1-x^n = \prod_{t=0}^{n-1} (1-\omega^t x)$ and $k<n$, we have
\begin{equation}
\frac{x^k}{1-x^n} = \sum_{t=0}^{n-1} \frac{c_t}{1-\omega^t x}
\end{equation}
for some constants $c_t$. Replacing $x$ by $\omega x$ gives
\begin{equation}
\frac{\omega^k x^k}{1-x^n} = \sum_{t=0}^{n-1} \frac{c_t}{1-\omega^{t+1} x}
= \sum_{t=0}^{n-1} \frac{c_{t-1}}{1-\omega^t x},
\end{equation}
where $c_{-1}: = c_{n-1}$. Combining these equations gives $c_t = \omega^{-k} c_{t-1}$, so
$c_t = \omega^{-kt} c_0$ and
\begin{equation}
\frac{x^k}{1-x^n} = c_0 \sum_{t=0}^{n-1} \frac{\omega^{-kt}}{1-\omega^t x}.
\end{equation}
Multiplying by $1-x$ and taking the limit as $x\rightarrow 1$ shows that $c_0 = 1/n$.
\end{proof}

\begin{proof}[Proof of Theorem \ref{theorem:rootsof1}]
By Definition \ref{definition:mdef},
\begin{align}
j(z,q) \sum_{t=0}^{n-1} \omega^{-kt} m(\omega^t x,q,z)
&=\sum_{t=0}^{n-1} \omega^{-kt} \sum_r \frac{q^{{\binom{r}{2}}} (-z)^r}{1-q^{r-1} z \omega^t x}\notag\\
&=\sum_r q^{{\binom{r}{2}}} (-z)^r \sum_{t=0}^{n-1} \frac{\omega^{-kt}}{1-\omega^t q^{r-1} x z}\notag\\
&=\sum_r q^{{\binom{r}{2}}} (-z)^r \frac{n (q^{r-1} x z)^k}{1-(q^{r-1} x z)^n} & \text{(by Lemma \ref{lemma:rootsof1})}\notag\\
&=n x^k z^k \sum_r \frac{q^{{\binom{r}{2}+kr-k}} (-z)^r} {1-q^{nr-n} x^n z^n}.
\end{align}
In the last sum we break up the terms according to the value of $r \text{ mod } n$.  Replacing $r$
by $nr+t+1$ with $0\leq t < n$, we obtain
\begin{align}
j(z,q) \sum_{t=0}^{n-1} \omega^{-kt} &m(\omega^t x,q,z)
=n x^k z^k \sum_{t=0}^{n-1} \sum_r \frac{q^{{\binom{nr+t+1}{2}+k(nr+t)}} (-z)^{nr+t+1}} {1-q^{n(nr+t)} x^n z^n} \notag\\
&=-n x^k z^{k+1} \sum_{t=0}^{n-1} q^{{\binom{t+1}{2}+kt}} \sum_r \frac{q^{{n^2\binom{r}{2}+r\binom{n+1}{2}+knr+nrt}} (-z)^{nr+t}} {1-q^{n^2 r+nt} x^n z^n} \notag\\
&=-n x^k z^{k+1} \sum_{t=0}^{n-1} q^{{\binom{t+1}{2}+kt}} (-z)^t \sum_r \frac{(q^{n^2})^{{\binom{r}{2}}}
\big(q^{{\binom{n+1}{2}+kn+nt}} (-z)^n\big)^r} {1-(q^{{n^2}})^{r-1} q^{n^2+nt} x^n z^n} \notag\\
&=-n x^k z^{k+1} \sum_{t=0}^{n-1} q^{{\binom{t+1}{2}+kt}} (-z)^t
j\big(-q^{{\binom{n+1}{2}+kn+nt}} (-z)^n;q^{n^2} \big) \notag\\
&\hspace{1in}\cdot m\big(-q^{{\binom{n}{2}-kn}} (-x)^n, q^{n^2}, -q^{{\binom{n+1}{2}+kn+nt}} (-z)^n\big).
\end{align}
By Theorem \ref{theorem:changing-z-theorem},
\begin{align}
&m\big(-q^{{\binom{n}{2}-kn}} (-x)^n, q^{n^2}, -q^{{\binom{n+1}{2}+kn+nt}} (-z)^n\big)\notag\\
&= m \big(-q^{{\binom{n}{2}-kn}} (-x)^n, q^{n^2}, z'\big)\\
&\ \ + \frac{z' J_{n^2}^3
\ j \big(-q^{{\binom{n+1}{2}+kn+nt}} (-z)^n/z';q^{n^2} \big)
\ j(q^{n^2+nt} x^n z^n z';q^{n^2})}
{j(z';q^{n^2})
\ j\big(-q^{{\binom{n+1}{2}+kn+nt}} (-z)^n;q^{n^2}\big)
\ j\big(-q^{{\binom{n}{2}-kn}}(-x)^n z';q^{n^2}\big)
\ j(q^{n^2+nt} x^n z^n;q^{n^2})}.\notag
\end{align}
By identities (\ref{equation:jsplit}) and (\ref{equation:1.8}),
\begin{align}
\sum_{t=0}^{n-1} q^{{\binom{t+1}{2}+kt}} (-z)^t j\big(-q^{{\binom{n+1}{2}+kn+nt}} (-z)^n; q^{n^2}\big)
&= j(q^{k+1} z;q)= q^{{-\binom{k+1}{2}}} (-z)^{-k-1} j(z;q).
\end{align}
Combining the last 3 equations, we obtain
\begin{align}
&j(z;q) \sum_{t=0}^{n-1} \omega^{-kt} m(\omega^t x,q,z)\notag\\
&=-n x^k z^{k+1} q^{{-\binom{k+1}{2}}}(-z)^{-k-1} j(z;q) m\big(-q^{{\binom{n}{2}-kn}} (-x)^n, q^{n^2}, z'\big) \notag\\
&\ \ -n x^k z^{k+1}
\sum_{t=0}^{n-1} q^{{\binom{t+1}{2}+kt}} (-z)^t
\frac{z' J_{n^2}^3
\ j\big(-q^{{\binom{n+1}{2}+kn+nt}} (-z)^n/z';q^{n^2}\big)
\ j(q^{n^2+nt} x^n z^n z';q^{n^2})}
{j(z';q^{n^2})
\ j\big(-q^{{\binom{n}{2}-kn}}(-x)^n z';q^{n^2}\big)
\ j(q^{n^2+nt} x^n z^n;q^{n^2})}\notag\\
&=n q^{{-\binom{k+1}{2}}} (-x)^k j(z,q) m\big(-q^{{\binom{n}{2}-kn}} (-x)^n, q^{n^2}, z'\big) \\
&\ \ - \frac{n x^k z^{k+1} J_{n^2}^3}
{j(z';q^{n^2})\ j\big(-q^{{\binom{n}{2}-kn}}(-x)^n z';q^{n^2}\big)}\notag\\
&\hspace{1in}\cdot \sum_{t=0}^{n-1}
\frac{q^{{\binom{t+1}{2}+kt}} (-z)^t
\ j\big(-q^{{\binom{n+1}{2}+kn+nt}} (-z)^n/z';q^{n^2}\big)
\ j(q^{nt} x^n z^n z';q^{n^2})}
{j(q^{nt} x^n z^n;q^{n^2})},\notag
\end{align}
where the last equality follows from 2 applications of (\ref{equation:1.8}). Dividing by $j(z,q)$ completes the proof.
\end{proof}
\begin{proof}[Proof of Theorem \ref{theorem:msplit-general-n}]
We add equation (\ref{equation:rootsof1}) for $k$ from $0$ to $n-1$ and divide by $n$. The left side is simple:
\begin{align} 
\frac{1}{n}\sum_{k=0}^{n-1} \sum_{t=0}^{n-1} \omega^{-kt} m(\omega^t x,q,z) =
\frac{1}{n}\sum_{t=0}^{n-1} m(\omega^t x,q,z)  \sum_{k=0}^{n-1}\omega^{-kt}.
\end{align}
The innermost sum equals $0$ unless $k=0$, in which case it equals $n$, so
\begin{align} 
\frac{1}{n}\sum_{k=0}^{n-1} \sum_{t=0}^{n-1} \omega^{-kt} m(\omega^t x,q,z) = m(x,q,z).
\end{align}
Hence
{\allowdisplaybreaks\begin{align} \label{equation:readytochangez}
m&(x,q,z) =
\sum_{k=0}^{n-1} q^{{-\binom{k+1}{2}}} (-x)^k m\big(-q^{{\binom{n}{2}-nk}} (-x)^n, q^{n^2}, z' \big)\notag\\*
& - \sum_{k=0}^{n-1} \frac{x^k z^{k+1} J_{n^2}^3}
{j(z;q)
j(z';q^{n^2})
j\big(-q^{{\binom{n}{2}-nk}} (-x)^n z'; q^{n^2}\big)}\notag\\*
&\hspace{1in}\cdot \sum_{t=0}^{n-1} \frac{q^{{\binom{t+1}{2}+kt}} (-z)^t j\big(-q^{{\binom{n+1}{2}+nk+nt}} (-z)^n/z';q^{n^2}\big)
j(q^{nt} x^n z^n z';q^{n^2})}
{j(q^{nt} x^n z^n;q^{n^2})}\notag\\
=&
\sum_{k=0}^{n-1} q^{{-\binom{k+1}{2}}} (-x)^k m\big(-q^{{\binom{n}{2}-nk}} (-x)^n, q^{n^2}, z' \big)\notag\\
& + \frac{J_{n^2}^3}{j(z;q) j(z';q^{n^2})}
\sum_{t=0}^{n-1}
\frac{q^{{\binom{t+1}{2}}} (-z)^{t+1} j(q^{nt} x^n z^n z';q^{n^2})} {j(q^{nt} x^n z^n;q^{n^2})} \notag\\
&\hspace{2in}\cdot\sum_{k=0}^{n-1}
\frac{q^{kt} x^k z^k j\big(-q^{{\binom{n+1}{2}+nk+nt}} (-z)^n/z';q^{n^2}\big)}
{j\big(-q^{{\binom{n}{2}-nk}} (-x)^n z'; q^{n^2}\big)} \notag\\
=&
\sum_{k=0}^{n-1} q^{{-\binom{k+1}{2}}} (-x)^k m\big(-q^{{\binom{n}{2}-nk}} (-x)^n, q^{n^2}, z' \big)\notag\\*
& + \frac{J_{n^2}^3}{j(z;q) j(z';q^{n^2})}
\sum_{t=0}^{n-1}
\frac{q^{{\binom{t+1}{2}}} (-z)^{t+1} j(q^{nt} x^n z^n z';q^{n^2})} {j(q^{nt} x^n z^n;q^{n^2})} \notag\\*
&\hspace{2in}\cdot\sum_{k=0}^{n-1}
\frac{(q^t x z)^k j\big(-q^{{\binom{n+1}{2}+nk+nt}} (-z)^n/z';q^{n^2}\big)}
{j\big(-q^{{\binom{n+1}{2}+nk}} (-x)^{-n}/z'; q^{n^2}\big)} \notag\\
=&
\sum_{k=0}^{n-1} q^{{-\binom{k+1}{2}}} (-x)^k m\big(-q^{{\binom{n}{2}-nk}} (-x)^n, q^{n^2}, z' \big)\notag\\*
& + \frac{J_{n^2}^3}{j(z;q) j(z';q^{n^2})}
\sum_{t=0}^{n-1}
\frac{q^{{\binom{t+1}{2}}} (-z)^{t+1} j(q^{nt} x^n z^n z';q^{n^2})} {j(q^{nt} x^n z^n;q^{n^2})} \notag\\*
&\hspace{2in}\cdot\frac{J_n^3\;j\big(q^{{\binom{n+1}{2}+t}} (-x)^{1-n} z/z'; q^n\big) j(q^{nt} x^n z^n;q^{n^2})}
{J_{n^2}^3\;j(q^t x z;q^n) j\big(-q^{{\binom{n+1}{2}}} (-x)^{-n}/z'; q^n\big)} \notag\\
=&
\sum_{k=0}^{n-1} q^{{-\binom{k+1}{2}}} (-x)^k m\big(-q^{{\binom{n}{2}-nk}} (-x)^n, q^{n^2}, z' \big)\notag\\*
& + \frac{J_n^3}{j(z;q) j\big(-q^{{\binom{n+1}{2}}} (-x)^{-n}/z'; q^n\big) j(z';q^{n^2})} \notag\\*
&\hspace{1in} \cdot \sum_{t=0}^{n-1}
\frac{q^{{\binom{t+1}{2}}} (-z)^{t+1}
j\big(q^{{\binom{n+1}{2}+t}} (-x)^{1-n} z/z'; q^n\big)
j(q^{nt} x^n z^n z';q^{n^2})}
{j(q^t x z;q^n)},
\end{align}}
where the fourth equality follows from (\ref{equation:ThmAH1.1}).  Now we substitute $q x^{-1} z^{-1}$ for $z$ in (\ref{equation:readytochangez}) and recall Corollary \ref{corollary:mxqz-flip-xz}:
{\allowdisplaybreaks\begin{align}
m&(x,q,z) =\ m(x,q,q x^{-1} z^{-1}) \notag\\
&= \sum_{k=0}^{n-1} q^{{-\binom{k+1}{2}}} (-x)^k m\big(-q^{{\binom{n}{2}-nk}} (-x)^n, q^{n^2}, z' \big)\notag\\*
& \ \ + \frac{J_n^3}{j(q x^{-1} z^{-1};q) j(z';q^{n^2})} \notag\\*
&\hspace{.5in} \cdot\sum_{t=0}^{n-1}
\frac{q^{{\binom{t+1}{2}}} (-q x^{-1} z^{-1})^{t+1}
j\big(-q^{{\binom{n+1}{2}+t+1}} (-x)^{-n} z^{-1}/z'; q^n\big)
j(q^{nt+n} z^{-n} z';q^{n^2})}
{ j\big(-q^{{\binom{n+1}{2}}} (-x)^{-n}/z', q^{t+1} z^{-1};q^n\big )} \notag\\
&=\sum_{k=0}^{n-1} q^{{-\binom{k+1}{2}}} (-x)^k m\big(-q^{{\binom{n}{2}-nk}} (-x)^n, q^{n^2}, z' \big)\notag\\*
&\  \ + \frac{z' J_n^3}{j(xz,q)j(z',q^{n^2})} \\*
&\hspace{.5in} \cdot\sum_{t=0}^{n-1}
\frac{q^{{\binom{n-1-t}{2}}} (-xz)^{n-1-t}
j\big(-q^{\binom{n}{2}+n-1-t} (-x)^n z z';q^n\big)
j(q^{n(n-1-t)} z^n/z';q^{n^2})}
{ j\big(-q^{\binom{n}{2}} (-x)^n z', q^{n-1-t} z;q^n\big)} ,\notag
\end{align}}%
by several applications of (\ref{equation:1.8}) and (\ref{equation:1.7}). Letting $t=n-1-r$ in the final sum gives the result.
\end{proof}


\begin{proof} [Proof of Corollary \ref{corollary:msplit-n=2}]
We specialize Theorem \ref{theorem:msplit-general-n} to the case $n=2$, $z'=z^4.$  This yields
\begin{align*}
m(x,q,z)&=m(-qx^2,q^4,z^4)-q^{-1}xm(-q^{-1}x^2,q^4,z^4)\\
&+\frac{z^4J_2^3}{j(xz;q)j(z^4;q^4)}
\Big\{\frac{j(-qx^2z^5;q^2)j(z^{-2};q^4)}{j(-qx^2z^4,z;q^2)}
-xz\frac{j(-q^2x^2z^5;q^2)j(q^2z^{-2};q^4)}{j(-qx^2z^4,qz;q^2)}\Big\}.
\end{align*}
We restrict ourselves to the sum of theta quotients.  Using (\ref{equation:1.7}) and (\ref{equation:1.8}) and simplifying, we have
\begin{align*}
-&\frac{J_2^3}{xj(xz;q)j(-qx^2z^4;q^2)j(z^4;q^4)}
\Big\{xz^2\frac{j(-qx^2z^5;q^2)j(z^{2};q^4)}{j(z;q^2)}
+\frac{j(-x^2z^5;q^2)j(q^2z^{2};q^4)}{j(qz;q^2)}\Big\}\\
&=-\frac{J_2J_4}{xj(xz;q)j(-qx^2z^4;q^2)j(z^4;q^4)}
\Big\{xz^2{j(-qx^2z^5,-z;q^2)}
+{j(-x^2z^5,-qz;q^2)}\Big\},
\end{align*}
where the second line follows from (\ref{equation:1.9}).  Applying (\ref{equation:H1Thm1.1}) produces the desired result.
\end{proof}

\begin{proof} [Proof of Corollary \ref{corollary:msplit-n=3}] Specializing Theorem \ref{theorem:msplit-general-n} to the case $n=3,z=z'=-1$, and simplifying yields
{\allowdisplaybreaks \begin{align*}
m&(x,q,-1)=m(q^3x^3,q^9,-1)-q^{-1}xm(x^3,q^9,-1)+q^{-3}x^2m(q^{-3}x^3,q^9,-1)\\
&\ \ +\frac{J_3^3}{j(-x;q)j(-x^3;q^3)j(-1;q^9)}\Big\{
\frac{q^{-1}xj(qx^3;q^3)j(q^3;q^9)}{j(-q;q^3)}
+\frac{q^{-1}x^2j(q^2x^3;q^3)j(q^6;q^9)}{j(-q^2;q^3)}\Big\}\\
&=m(q^3x^3,q^9,-1)-q^{-1}xm(x^3,q^9,-1)+q^{-3}x^2m(q^{-3}x^3,q^9,-1)\\
&\ \ +\frac{xJ_1J_3^2J_6J_9}{2qJ_2J_{18}^2j(-x;q)j(-x^3;q^3)}\Big\{ j(x^3q;q^3)+xj(q^2x^3;q^3)\Big\}.
\end{align*}}
We restrict ourselves to the sum of theta quotients and find
\begin{align*}
\frac{xJ_1J_3^2J_6J_9(j(x^3q;q^3)+xj(q^2x^3;q^3))}{2qJ_2J_{18}^2j(-x;q)j(-x^3;q^3)}&=\frac{xq^{-1}J_1J_3^2J_6J_9}{2J_2J_{18}^2j(-x;q)j(-x^3;q^3)} \frac{J_1j(x^2;q)}{j(x;q)}& {\text{(by (\ref{equation:H1Thm1.0})}})\\
&=\frac{xJ_3^2J_6J_9j(x^2;q)}{2qJ_{18}^2j(-x^3;q^3)j(x^2;q^2)}\qedhere
&{\text{(by (\ref{equation:1.9}))}}\\
&=\frac{xJ_1J_3^2J_6J_9j(qx^2;q^2)}{2qJ_2^2J_{18}^2j(-x^3;q^3)}.&{\text{(by (\ref{equation:1.10}))}}
\end{align*}
\end{proof}


\section{Other generalized Lambert series}\label{section:genlambert}

In this section we recall the functions $g(x,q),$ $h(x,q)$, and $k(x,q)$ which have been used to express mock theta functions.  We then derive representations for these functions in terms of Appell-Lerch sums.

\begin{definition} \label{definition:g-def} Let $x$ be neither $0$ nor an integral power of $q$. Then 
\begin{equation}
g(x,q):=x^{-1}\Big ( -1 +\sum_{n\ge0}\frac{q^{n^2}}{(x)_{n+1}(q/x)_{n}} \Big ).
\end{equation}
\end{definition}
\noindent It is not difficult to show that $g(x,q)$ can be expressed more simply,
\begin{equation}
g(x,q)=\sum_{n\ge0}\frac{q^{n(n+1)}}{(x)_{n+1}(q/x)_{n+1}}.\label{equation:g-alt-def}
\end{equation}

We begin by expressing $g(x,q)$ in terms of Appell-Lerch sums.
\begin{proposition}  \label{proposition:g-to-mxqz} For generic $x\in\mathbb{C}^*$
\begin{equation}
g(x,q)=-x^{-1}m(q^2x^{-3},q^3,x^2)-x^{-2}m(qx^{-3},q^3,x^2).
\end{equation}
\end{proposition}
\begin{proof} It was shown in \cite[Theorem $2.2$]{H1} that
\begin{align}
g(x,q)=\frac{1}{j(x^3z;q^3)}& \Big [\sum_r\frac{(-1)^rq^{3r(r+1)/2}x^{3r+1}z^{r+1}}{1-q^{3r+1}z}\notag\\
&+\sum_r \frac{(-1)^rq^{3r(r+3)/2+1}x^{-3r-1}z^{-r-1}}{1-q^{3r+1}z^{-1}}
+\frac{J_1^2j(xz;q)j(z;q^3)}{j(x;q)j(z;q)}\Big ].
\end{align}
Changing $r$ to $-1-r$ in the second sum and using (\ref{equation:mxqz-altdef}), we obtain
\begin{align}
g(x,q)=-x^{-2}m(qx^{-3},q^3,x^3z)-x^{-1}m(q^2x^{-3},q^3,x^3z)+\frac{J_1^2j(xz;q)j(z;q^3)}{j(x;q)j(z;q)j(x^3z;q^3)}.
\end{align}
Setting $z=x^{-1}$ gives the proposition, since then $j(xz,q)=0.$
\end{proof}
\begin{definition} Let  $x$ be neither $0$ nor an integral power of $q$. Then
\begin{equation}
h(x,q):=\frac{1}{j(q;q^2)}\sum_n \frac{(-1)^n q^{n(n+1)}}{1-q^n x}.\label{equation:h-def}
\end{equation}
\end{definition}
\begin{proposition} \label{proposition:h-to-mxqz} For generic $x\in \mathbb{C}^*$
\begin{equation}
h(x,q)=-x^{-1}m(x^{-2}q,q^2,x).
\end{equation}
\end{proposition}
\begin{proof}  Recalling \cite[Theorem $2.10$]{C1}, we have that
\begin{align*}
j(x^2z;&q^2)h(x,q)-\frac{J_1^3j(xz;q)j(z;q^2)}{j(q;q^2)j(x;q)j(z;q)}\\
&=\frac12 \sum_{n}\frac{(-1)^nq^{n(n+1)}x^{2n+1}z^{n+1}}{1-q^{2n+1}z}
+\frac12 \sum_{n}\frac{(-1)^nq^{n(n+3)+1}x^{-2n-1}z^{-n-1}}{1-q^{2n+1}z^{-1}}\\
&=-\frac{1}{2}x^{-1}j(x^2z;q^2)m(x^{-2}q,q^2,x^2z)-\frac{1}{2}q^{-1}xj(q^2x^{-2}z^{-1};q^2)m(x^2q^{-1},q^2,q^2x^{-2}z^{-1})\\
&=-x^{-1}j(x^2z;q^2)m(x^{-2}q,q^2,x^2z),
\end{align*}
where the second equality follows from (\ref{equation:mxqz-altdef}) and the last equality follows from (\ref{equation:mxqz-fnq-z}) and (\ref{equation:mxqz-flip}).  Setting $z=x^{-1}$ gives the proposition.
\end{proof}

\begin{definition} Let $x^2$ be neither zero nor an integral power of $q^2$. Then
\begin{equation}
k(x,q):=\frac{1}{xj(-q;q^4)}\sum_n \frac{q^{n(2n+1)}}{1-q^{2n} x^2}.
\end{equation}
\end{definition}

\begin{proposition}\label{proposition:CH-k(x,q)} For generic $x\in \mathbb{C}^*$
\begin{align}
xk(x,q)&=m(-qx^4,q^4,-x^{-2}q^{-1})+q^{-1}x^2m(-q^{-1}x^4,q^4,-x^{-2}q^{-1})\\
&=m(-x^2,q,x^{-2})+\frac{J_1^4}{2J_2^2j(x^2;q)}.
\end{align}
\end{proposition}
\begin{proof}
Working as for Proposition \ref{proposition:h-to-mxqz} and setting $z=qx^{-2}$ in \cite[Theorem $2.5.4$]{C2} yields the first equality:
\begin{align*}
xk(x,q)&=1-x^{-2}m(-x^{-4}q,q^4,-x^2q) -m(-x^{-4}q^3,q^4,-x^2q)\\
&=-x^{-2}m(-x^{-4}q,q^4,-x^2q) -q^{-1}x^{-4}m(-x^{-4}q^{-1},q^4,-x^2q) & ({\text{by }} (\ref{equation:mxqz-fnq-x}))\\
&=m(-qx^4,q^4,-x^{-2}q^{-1})+q^{-1}x^2m(-q^{-1}x^4,q^4,-x^{-2}q^{-1}). & ({\text{by }} (\ref{equation:mxqz-flip}))
\end{align*}
Using the $n=2$ case of Theorem \ref{theorem:msplit-general-n}, we have
\begin{align*}
m(-x^2,q,x^{-2})&=m(-qx^4,q^4,-x^{-2}q^{-1})+q^{-1}x^2m(-q^{-1}x^4,q^4,-x^{-2}q^{-1})\\
&\ -\frac{x^{-2}q^{-1}J_2^3}{j(-1;q)j(x^2;q^2)j(-x^{-2}q^{-1};q^4)}
\cdot\Big \{ \frac{j(q;q^2)j(-q^3x^{-2};q^4)}{j(qx^{-2};q^2)}\Big \}.
\end{align*}
The second equality of the proposition then follows from elementary theta function properties.
\end{proof}


\section{Known functions in terms of the Appell-Lerch sum $m(x,q,z)$}\label{section:known}
In this section we express known mock theta functions in terms of the Appell-Lerch sum $m(x,q,z)$.  For each of the mock theta functions we start by giving its $q$-series definition(s).  Then, if possible, we express the function in terms of either $g(x,q)$, $h(x,q)$, or $k(x,q),$ which were discussed in Section \ref{section:genlambert}.  Next we express the function entirely as a linear combination of $m(x,q,z)$ functions.  In some cases these linear combinations involve two or more terms of the form $m(x,q,z),$ with the same $x$ and $q$ but different values of $z$.  In such cases we also express the function without such duplication, by adding on a suitable theta function.  For example, the identities in (\ref{equation:3rd-chi(q)}) expresses the `3rd order' function $\chi(q)$ as a combination of two functions of the form $m(-q,q^3,z)$, and then as one such function plus a theta function.   Theorem \ref{theorem:changing-z-theorem} and the Riemann relation for theta functions are used frequently.  Throughout, `$i$' denotes $\sqrt{-1}$ and `$\omega$' denotes a primitive cube root of $1.$  As in \cite{AH}, the summation symbol $\sum^*$ for the `6th order' function $\mu(q)$ denotes the average of the sequence of even partial sums and the sequence of odd partial sums.

\noindent {\bf `2nd order' functions}
\begin{align}
A(q)=&\sum_{n\ge 0}\frac{q^{n+1}(-q^2;q^2)_n}{(q;q^2)_{n+1}}=\sum_{n\ge 0}\frac{q^{(n+1)^2}(-q;q^2)_n}{(q;q^2)_{n+1}^2}=-m(q,q^4,q^2)\label{equation:2nd-A(q)}\\
B(q)=&\sum_{n\ge 0}\frac{q^{n}(-q;q^2)_n}{(q;q^2)_{n+1}}=\sum_{n\ge 0}\frac{q^{n^2+n}(-q^2;q^2)_n}{(q;q^2)_{n+1}^2}=-q^{-1}m(1,q^4,q^3)\label{equation:2nd-B(q)}\\
\mu(q)=&\sum_{n\ge 0}\frac{(-1)^nq^{n^2}(q;q^2)_n}{(-q^2;q^2)_{n}^2}=2m(-q,q^4,-1)+2m(-q,q^4,q)=4m(-q,q^4,-1)-\frac{J_{2,4}^4}{J_1^3}\label{equation:2nd-mu(q)}
\end{align}
\noindent {\bf `3rd order' functions}
{\allowdisplaybreaks \begin{align}
f(q)&=\sum_{n\ge 0}\frac{q^{n^2}}{(-q)_n^2}=2-2g(-1,q)=2m(-q,q^3,q)+2m(-q,q^3,q^2)\label{equation:3rd-f(q)}\\
&=4m(-q,q^3,q)+\frac{J_{3,6}^2}{J_1}\notag\\
\phi(q)&=\sum_{n\ge 0}\frac{q^{n^2}}{(-q^2;q^2)_n}=(1-i)(1+ig(i,q))\label{equation:3rd-phi(q)}\\
&=(1+i)m(iq,q^3,-1)+(1-i)m(-iq,q^3,-1)\notag\\
&=m(q^5,q^{12},q^4)+m(q^5,q^{12},q^8)+q^{-1}m(q,q^{12},q^4)+q^{-1}m(q,q^{12},q^8)\notag\\
&=2m(q,-q^3,-1)+\frac{2qJ_{12}^3}{J_4 J_{3,12}} \notag \\
\psi(q)&=\sum_{n\ge 1}\frac{q^{n^2}}{(q;q^2)_n}=qg(q,q^4)=-q^{-1}m(q,q^{12},q^2)-m(q^5,q^{12},q^2)\label{equation:3rd-psi(q)}\\
&=-m(q,-q^3,-q)+\frac{qJ_{12}^3}{J_4 J_{3,12}}\notag\\
\chi(q)&=\sum_{n\ge 0}\frac{q^{n^2}(-q)_n}{(-q^3;q^3)_n}=(1+\omega)(1-\omega g(-\omega,q))=2m(-q,q^3,q^2)-m(-q,q^3,q)\label{equation:3rd-chi(q)}\\
&=m(-q,q^3,q)+\frac{J_{3,6}^2}{J_1}\notag\\
 \omega(q)&=\sum_{n\ge 0}\frac{q^{2n(n+1)}}{(q;q^2)_{n+1}^2}=g(q,q^2)=-q^{-1}m(q,q^6,q^2)-q^{-1}m(q,q^6,q^4)\label{equation:3rd-omega(q)}\\
&=-2q^{-1}m(q,q^6,q^2)+\frac{J_6^3}{J_2 J_{3,6}}\notag\\
\nu(q)&=\sum_{n\ge 0} \frac{q^{n(n+1)}} {(-q;q^2)_{n+1}}=g(i\sqrt{q},q)=iq^{-1/2}\Big (m(i\sqrt{q},q^3,-q)-m(-i\sqrt{q},q^3,-q^2)\Big )\label{equation:3rd-nu(q)}\\
&=q^{-1}m(q^2,q^{12},-q^3)+q^{-1}m(q^2,q^{12},-q^9)=2q^{-1}m(q^2,q^{12},-q^3)+\frac{J_1 J_{3,12}}{J_2}\notag\\
\rho(q)&= \sum_{n\ge 0}\frac{q^{2n(n+1)}(q;q^2)_{n+1}}{(q^3;q^6)_{n+1}}=g(\omega q,q^2)\label{equation:3rd-rho(q)}\\
&=-\omega q^{-1}m(q,q^6,\omega q^4)-\omega^2 q^{-1}m(q,q^6,\omega^2q^2)= q^{-1}m(q,q^6,-q)\notag
\end{align}}
\noindent {\bf `5th order' functions}
{\allowdisplaybreaks \begin{align}
f_0(q)&=\sum_{n\ge 0}\frac{q^{n^2}}{(-q)_n}=\frac{J_{5,10}J_{2,5}}{J_1}-2q^2g(q^2,q^{10})\label{equation:5th-f0(q)}\\
&=m(q^{14},q^{30},q^{14})+m(q^{14},q^{30},q^{29})+q^{-2}m(q^{4},q^{30},q^4)+q^{-2}m(q^{4},q^{30},q^{19}) \notag\\
&=2m(q^{14},q^{30},q^4)+2q^{-2}m(q^{4},q^{30},q^4)+\frac{J_{5,10}J_{2,5}}{J_1}\notag\\
\phi_0(q)&=\sum_{n\ge 0}q^{n^2}(-q;q^2)_n=qg(-q,-q^{5})+\frac{J_{10}j(-q^2;-q^5)}{J_{2,10}}\label{equation:5th-phi0(q)}\\
&=m(-q^{7},-q^{15},q^{9})-q^{-1}m(q^{2},-q^{15},q^{9})\notag \\
\psi_0(q)&=\sum_{n\ge 0}q^{\binom{n+2}{2}}(-q)_n=q^2g(q^2,q^{10})+\frac{qJ_{5}J_{1,10}}{J_{2,5}}=-m(q^{14},q^{30},q^{3})-q^{-2}m(q^{4},q^{30},q^{3})\label{equation:5th-psi0(q)}\\
F_0(q)&=\sum_{n\ge 0}\frac{q^{2n^2}}{(q;q^2)_n}=1+qg(q,q^{5})-\frac{qJ_{10}\overline{J}_{5,20}}{J_{4,10}}\label{equation:5th-BigF0(q)}\\
&=-\frac12 q^{-1}m(q^{2},q^{15},q^{2})-\frac12 q^{-1}m(q^{2},q^{15},-q^{2}) +\frac12 m(q^{8},q^{15},q^8)+\frac12 m(q^{8},q^{15},-q^{8}) \notag\\
&=-q^{-1}m(q^{2},q^{15},q)+m(q^{8},q^{15},q^4)-\frac{qJ_{10}\overline{J}_{5,20}}{J_{4,10}}\notag\\
\chi_0(q)&=\sum_{n\ge 0}\frac{q^n}{(q^{n+1})_n}=1+\sum_{n\ge 0}\frac{q^{2n+1}}{(q^{n+1})_{n+1}}=2+3qg(q,q^5)-\frac{J_5^2J_{2,5}}{J_{1,5}^2}\label{equation:5th-chi0(q)}\\
&=2-2m(q^7,q^{15},q^{12})-m(q^7,q^{15},q^9)-2q^{-1}m(q^2,q^{15},q^{12})-q^{-1}m(q^2,q^{15},q^9)\notag\\
&= 2-3m(q^7,q^{15},q^9)-3q^{-1}m(q^2,q^{15},q^4)+\frac{2J_5^2J_{2,5}}{J_{1,5}^2}\notag \\
f_1(q)&=\sum_{n\ge 0}\frac{q^{n(n+1)}}{(-q)_n}=\frac{J_{5,10}J_{1,5}}{J_1}-2q^3g(q^4,q^{10})\label{equation:5th-f1(q)}\\
&=q^{-1}m(q^{8},q^{30},q^{8})+q^{-1}m(q^{8},q^{30},q^{23})+q^{-3}m(q^{2},q^{30},q^2)+q^{-3}m(q^{2},q^{30},q^{17}) \notag\\
&=2q^{-1}m(q^{8},q^{30},q^8)+2q^{-3}m(q^{2},q^{30},q^{-8})+\frac{J_{5,10}J_{1,5}}{J_1}\notag\\
\phi_1(q)&=\sum_{n\ge 0}q^{(n+1)^2}(-q;q^2)_n=q^2g(q^2,-q^{5})+\frac{qJ_{10}j(q;-q^5)}{J_{4,10}}\label{equation:5th-phi1(q)}\\
&=q^{-1}m(-q,-q^{15},q^{-3})-m(q^{4},-q^{15},q^{3}) \notag\\
\psi_1(q)&=\sum_{n\ge 0}q^{\binom{n+1}{2}}(-q)_n=q^3g(q^4,q^{10})+\frac{J_{5}J_{3,10}}{J_{1,5}}=-q^{-1}m(q^{8},q^{30},q^{-9})-q^{-3}m(q^{2},q^{30},q^{9})\label{equation:5th-psi1(q)}\\
F_1(q)&=\sum_{n\ge 0}\frac{q^{2n(n+1)}}{(q;q^2)_{n+1}}=qg(q^2,q^{5})+\frac{J_{10}\overline{J}_{5,20}}{J_{2,10}}\label{equation:5th-BigF1(q)}\\
&=-\frac12 q^{-2}m(q,q^{15},q)-\frac12 q^{-2}m(q,q^{15},-q)\notag\\
&\ \ \ \ \ -\frac12 q^{-1}m(q^{4},q^{15},q^4)-\frac12 q^{-1}m(q^{4},q^{15},-q^{4}) \notag\\
&=-q^{-2}m(q,q^{15},q^{-4})-q^{-1}m(q^{4},q^{15},q^4)+\frac{J_{10}\overline{J}_{5,20}}{J_{2,10}}\notag\\
\chi_1(q)&=\sum_{n\ge 0}\frac{q^n}{(q^{n+1})_{n+1}}=1+\sum_{n\ge 0}\frac{q^{2n+1}(1+q^n)}{(q^{n+1})_{n+1}}=3qg(q^2,q^5)+\frac{J_5^2J_{1,5}}{J_{2,5}^2}\label{equation:5th-chi1(q)}\\
&=-2q^{-1}m(q^4,q^{15},q^{-6})-q^{-1}m(q^4,q^{15},q^3)-2q^{-2}m(q,q^{15},q^{6})-q^{-2}m(q,q^{15},q^{-3})\notag\\
&= -3q^{-1}m(q^4,q^{15},q^3)-3q^{-2}m(q,q^{15},q^2)-\frac{2J_5^2J_{1,5}}{J_{2,5}^2} \notag \\
\Phi(q)&=-1+\sum_{n\ge 0}\frac{q^{5n^2}}{(q;q^5)_{n+1}(q^4;q^5)_n}=qg(q,q^5)=-q^{-1}m(q^2,q^{15},q^2)-m(q^7,q^{15},q^2)\label{equation:5th-AuxPhi(q)}\\
\Psi(q)&=-1+\sum_{n\ge 0}\frac{q^{5n^2}}{(q^2;q^5)_{n+1}(q^3;q^5)_n}=q^2g(q^2,q^5)=-q^{-1}m(q,q^{15},q^{-4})-m(q^4,q^{15},q^4)\label{equation:5th-AuxPsi(q)}
\end{align}}
\noindent {\bf `6th order' functions}
{\allowdisplaybreaks \begin{align}
\phi(q)&=\sum_{n\ge 0}\frac{(-1)^nq^{n^2}(q;q^2)_n}{(-q)_{2n}}=2m(q,q^3,-1)\label{equation:6th-phi(q)}\\
\psi(q)&=\sum_{n\ge 0}\frac{(-1)^nq^{(n+1)^2}(q;q^2)_n}{(-q)_{2n+1}}=m(1,q^3,-q)\label{equation:6th-psi(q)}\\
\rho(q)&=\sum_{n\ge 0}\frac{q^{\binom{n+1}{2}}(-q)_n}{(q;q^2)_{n+1}}=-q^{-1}m(1,q^6,q)\label{equation:6th-rho(q)}\\
\sigma(q)&=\sum_{n\ge 0}\frac{q^{\binom{n+2}{2}}(-q)_n}{(q;q^2)_{n+1}}=-m(q^2,q^6,q)\label{equation:6th-sigma(q)}\\
\lambda(q)&=\sum_{n\ge 0}\frac{(-1)^nq^n(q;q^2)_n}{(-q)_n}
=q^{-1}m(1,q^6,-q^2)+q^{-1}m(1,q^6,-q)\label{equation:6th-lambda(q)}\\
&=2q^{-1}m(1,q^6,-q^2)+\frac{J_{1,2}\overline{J}_{3,12}}{\overline{J}_{1,4}}\notag\\
\mu(q)&={\sum_{n\ge 0}}^*\frac{(-1)^n(q;q^2)_n}{(-q)_n}=\frac12 +\frac12 \sum_{n\ge 0}\frac{(-1)^nq^{n+1}(1+q^n)(q;q^2)_n}{(-q;q)_{n+1}}\label{equation:6th-mu(q)}\\
&=m(q^2,q^6,-1)+ m(q^2,q^6,-q^3)=2m(q^2,q^6,-1)-\frac{J_{1,2}\overline{J}_{1,3}}{2\overline{J}_{1,4}}\notag\\
\gamma(q)&=\sum_{n\ge 0}\frac{q^{n^2}(q)_n}{(q^3;q^3)_n}=(1-\omega)(1+\omega g(\omega,q))\label{equation:6th-gamma(q)}\\
&=2 m(q,q^3,-1)+ m(q,q^3,-q)=3m(q,q^3,-q)+\frac{J_{1,2}^2}{{\overline{J}}_{1,3}}\notag\\
{\phi}\_(q)&=\sum_{n\ge 1}\frac{q^n(-q;q)_{2n-1}}{(q;q^2)_n}=-\frac{3}{4}m(q,q^3,q)-\frac{1}{4}m(q,q^3,-q)=-m(q,q^3,q)-q\frac{\overline{J}_{3,12}^3}{J_1\overline{J}_{1,4}}\label{equation:phibar}\\
{\psi}\_(q)&=\sum_{n\ge 1}\frac{q^n(-q;q)_{2n-2}}{(q;q^2)_n}=-\frac{3}{4}m(1,q^3,q)+\frac{1}{4}m(1,q^3,-q)=-\frac{1}{2}m(1,q^3,q)+q\frac{{J}_{6}^3}{2J_1 J_2}\label{equation:psibar}
\end{align}}
\noindent {\bf `7th order' functions}
{\allowdisplaybreaks \begin{align}
{\mathcal{F}}_0(q)&=\sum_{n\ge 0}\frac{q^{n^2}}{(q^{n+1};q)_{n}}=2+2qg(q,q^{7})-\frac{J_{3,7}^2}{J_1}\label{equation:7th-F0(q)}\\
&=m(q^{10},q^{21},q^9)+m(q^{10},q^{21},q^{-9})-q^{-1}m(q^{4},q^{21},q^9)-q^{-1}m(q^{4},q^{21},q^{-9}) \notag\\
&=2m(q^{10},q^{21},q^9)-2q^{-1}m(q^{4},q^{21},q^{-9})+\frac{J_{3,7}^2}{J_1}\notag\\
{\mathcal{F}}_1(q)&=\sum_{n\ge 1}\frac{q^{n^2}}{(q^{n};q)_{n}}=2q^2g(q^2,q^{7})+\frac{qJ_{1,7}^2}{J_1}\label{equation:7th-F1(q)}\\
&=-m(q^{8},q^{21},q^3)-m(q^{8},q^{21},q^{-3})-q^{-2}m(q,q^{21},q^3)-q^{-2}m(q,q^{21},q^{-3}) \notag\\
&=-2m(q^{8},q^{21},q^3)-2q^{-2}m(q,q^{21},q^{3})-\frac{qJ_{1,7}^2}{J_1}\notag\\
{\mathcal{F}}_2(q)&=\sum_{n\ge 0}\frac{q^{n(n+1)}}{(q^{n+1};q)_{n+1}}=2q^2g(q^3,q^{7})+\frac{J_{2,7}^2}{J_1}\label{equation:7th-F2(q)}\\
&=-q^{-1}m(q^{5},q^{21},q^6)-q^{-1}m(q^{5},q^{21},q^{-6})-q^{-2}m(q^2,q^{21},q^6)-q^{-2}m(q^2,q^{21},q^{-6}) \notag\\
&=-2q^{-1}m(q^{5},q^{21},q^6)-2q^{-2}m(q^2,q^{21},q^{-6})+\frac{J_{2,7}^2}{J_1}\notag
\end{align}}
\noindent {\bf `8th order' functions}
{\allowdisplaybreaks \begin{align}
S_0(q)&=\sum_{n\ge 0}\frac{q^{n^2}(-q;q^2)_n}{(-q^2;q^2)_{n}}=m(-q^3,q^8,-q^2)+m(-q^3,q^8,-q^6)\label{equation:8th-S0(q)}\\
&=2m(-q^3,q^8,-1) +\frac{q{\overline{J}}_{1,8}J_{2,8}^2}{J_{3,8}^2}\notag\\
S_1(q)&=\sum_{n\ge 0}\frac{q^{n(n+2)}(-q;q^2)_n}{(-q^2;q^2)_{n}}=-q^{-1}m(-q,q^8,-q^2)-q^{-1}m(-q,q^8,-q^6)\label{equation:8th-S1(q)}\\
&=-2q^{-1}m(-q,q^8,-1) +\frac{{\overline{J}}_{3,8}J_{2,8}^2}{qJ_{1,8}^2}\notag\\
T_0(q)&=\sum_{n\ge 0}\frac{q^{(n+1)(n+2)}(-q^2;q^2)_n}{(-q;q^2)_{n+1}}=-m(-q^3,q^8,q^2)\label{equation:8th-T0(q)}\\
T_1(q)&=\sum_{n\ge 0}\frac{q^{n(n+1)}(-q^2;q^2)_n}{(-q;q^2)_{n+1}}=q^{-1}m(-q,q^8,q^6)\label{equation:8th-T1(q)}\\
U_0(q)&=\sum_{n\ge 0}\frac{q^{n^2}(-q;q^2)_n}{(-q^4;q^4)_{n}}=2m(-q,q^4,-1)\label{equation:8th-U0(q)}\\
U_1(q)&=\sum_{n\ge 0}\frac{q^{(n+1)^2}(-q;q^2)_n}{(-q^2;q^4)_{n+1}}=-m(-q,q^4,-q^2)\label{equation:8th-U1(q)}\\
V_0(q)&=-1+2\sum_{n\ge 0}\frac{q^{n^2}(-q;q^2)_n}{(q;q^2)_{n}}=-1+2\sum_{n\ge 0}\frac{q^{2n^2}(-q^2;q^4)_n}{(q;q^2)_{2n+1}}\label{equation:3rd-chi(q)5}\\
&=-q^{-1}m(1,q^8,q)-q^{-1}m(1,q^8,q^3)=-2q^{-1}m(1,q^8,q) -\frac{{\overline{J}}_{1,4}^2}{J_{2,8}}\notag\\
V_1(q)&=\sum_{n\ge 0}\frac{q^{(n+1)^2}(-q;q^2)_n}{(q;q^2)_{n+1}}=\sum_{n\ge 0}\frac{q^{2n^2+2n+1}(-q^4;q^4)_n}{(q;q^2)_{2n+2}}=\sum_{n\ge 0}\frac{q^{n+1}(-q)_{2n}}{(-q^2;q^4)_{n+1}} \label{equation:8th-V1(q)}\\
&=-m(q^2,q^8,q)\notag
\end{align}}
\noindent {\bf `10th order' functions}
{\allowdisplaybreaks \begin{align}
{\phi}(q)&=\sum_{n\ge 0}\frac{q^{\binom{n+1}{2}}}{(q;q^2)_{n+1}}=2qh(q^2,q^{5})+\frac{J_{5}J_{10}J_{4,10}}{J_{2,5}J_{2,10}}\label{equation:10th-phi(q)}\\
&=-q^{-1}m(q,q^{10},q)-q^{-1}m(q,q^{10},q^{2})
=-2q^{-1}m(q,q^{10},q^2)+\frac{J_{5}J_{10}J_{4,10}}{J_{2,5}J_{2,10}}\notag\\
{\psi}(q)&=\sum_{n\ge 0}\frac{q^{\binom{n+2}{2}}}{(q;q^2)_{n+1}}=2qh(q,q^{5})-\frac{qJ_5J_{10}J_{2,10}}{J_{1,5}J_{4,10}}\label{equation:10th-psi(q)}\\
&=-m(q^3,q^{10},q)-m(q^3,q^{10},q^{3})
=-2m(q^3,q^{10},q)-\frac{qJ_5J_{10}J_{2,10}}{J_{1,5}J_{4,10}}\notag\\
{X}(q)&=\sum_{n\ge 0}\frac{(-1)^nq^{n^2}}{(-q;q)_{2n}}
=2qk(q,q^5)-\frac{J_5J_{10}J_{2,5}}{J_{2,10}J_{1,5}}\label{equation:10th-BigX(q)}\\
&=m(-q^2,q^{5},q)+m(-q^2,q^{5},q^{4})
=2m(-q^2,q^5,q^4)-\frac{{{J}}_{3,10}{{J}}_{5,10} }{J_{1,5}}\notag\\
{\chi}(q)&=\sum_{n\ge 0}\frac{(-1)^nq^{(n+1)^2}}{(-q;q)_{2n+1}}
=2-2q^2k(q^2,q^5)+q\frac{J_5J_{10}J_{1,5}}{J_{4,10}J_{2,5}}\label{equation:10th-chi(q)}\\
&=m(-q,q^{5},q^2)+m(-q,q^{5},q^{3})
=2m(-q,q^5,q^2)+\frac{q{{J}}_{1,10}{{J}}_{5,10} }{J_{2,5}}\notag
\end{align}}
\subsection{Proofs of Identities}

{\bf `2nd order' functions}  To prove (\ref{equation:2nd-A(q)}), we recall \cite[(4)]{M2}:
\begin{align*}
A(q)&=\frac{1}{2}\Big (V_1(q^{1/2})+V_1(-q^{1/2})\Big )
=-\frac{1}{2}\big (m(q,q^4,q^{1/2})+m(q,q^4,-q^{1/2})\Big ) & ({\text{by }} (\ref{equation:8th-V1(q)}))\\
&=-m(q,q^4,q^2),
\end{align*}
where the last equality follows from using Theorem \ref{theorem:changing-z-theorem} twice.  To prove (\ref{equation:2nd-B(q)}), we first recall \cite[(4.3)]{AM}:
\begin{align*}
B(q)&=\frac{(-q^2;q^2)_{\infty}}{(q^2;q^2)_{\infty}}\sum_{n}\frac{(-1)^nq^{2n^2+2n}}{1-q^{2n+1}}
=\frac{J_4}{J_2^2}\sum_{n}\frac{(-1)^nq^{2n^2+2n}}{1-q^{4n+2}}(1+q^{2n+1})\\
&=\frac{J_4}{J_2^2}\Big [ \frac{J_4^3}{J_{2,4}} +qj(q^6;q^4)m(1,q^4,q^6)\Big ]=\frac{J_4^4}{J_2^2J_{2,4}} -q^{-1}m(1,q^4,q^2),& ({\text{by }} (\ref{equation:Reciprocal}))
\end{align*}
where the last equality follows from (\ref{equation:1.8}) and (\ref{equation:mxqz-fnq-z}).  The result then follows from Theorem \ref{theorem:changing-z-theorem}. To establish (\ref{equation:2nd-mu(q)}), we prove the penultimate equality.  The last equality then follows from Theorem \ref{theorem:changing-z-theorem} and elementary theta function properties.  Here we insert (\ref{equation:8th-U0(q)}) and (\ref{equation:8th-U1(q)}) into \cite[(2)]{M2}.

\noindent  {\bf `3rd order' functions}

Here, the first equality is just the Eulerian form  \cite[p. 62]{W3}.   For all but (\ref{equation:3rd-psi(q)}), the second equality is just Definition \ref{definition:g-def} or identity (\ref{equation:g-alt-def}).  For (\ref{equation:3rd-psi(q)}), the second equality follows from Definition \ref{definition:g-def} and the form found in \cite[p. 65]{W3}.  We proceed with the other equalities on a case-by-case basis.  

We prove the last equality in (\ref{equation:3rd-f(q)}).   The penultimate equality then follows from Theorem {\ref{theorem:changing-z-theorem}}.  Rewriting the Lambert series found in \cite[p. 64]{W3} and using (\ref{equation:Reciprocal})
\begin{align*}
J_1f(q)&=2j(q^2;q^3)m(-q,q^3,q^2)-\frac{2J_3^3}{\overline{J}_{0,3}}+j(q^4;q^3)m(-q^{-1},q^3,q^4)\\
&=4J_1m(-q,q^3,q^2)-J_{3,6}^2,
\end{align*}
where the last equality follows from applying (\ref{equation:1.8}), (\ref{equation:mxqz-fnq-z}) and (\ref{equation:mxqz-flip}) to the last term.    The result follows from Theorem {\ref{theorem:changing-z-theorem}}.  For  (\ref{equation:3rd-chi(q)}) we use the Lambert series in \cite[p. 64]{W3} and argue in a similar fashion.  

For (\ref{equation:3rd-omega(q)}), we use Proposition \ref{proposition:g-to-mxqz} and (\ref{equation:mxqz-fnq-z}) to obtain the next to last equality and Theorem {\ref{theorem:changing-z-theorem}} to obtain the last equality.  For  (\ref{equation:3rd-rho(q)}), the penultimate equality follows from Proposition \ref{proposition:g-to-mxqz}, (\ref{equation:mxqz-flip}), and (\ref{equation:mxqz-fnq-z}).  For the final equality, we use the formula in \cite[p. 66]{W3} and argue as in (\ref{equation:3rd-f(q)}) where we use Theorem {\ref{theorem:changing-z-theorem}} at the very end to remove the theta quotient.

We prove (\ref{equation:3rd-phi(q)}).  The third equality is just an application of Proposition \ref{proposition:g-to-mxqz} and (\ref{equation:mxqz-fnq-x}).  To obtain the last two expressions we combine the results of (\ref{equation:3rd-f(q)}) with an identity from \cite[p. 63]{W3}:
\begin{align*}
2\phi(-q)&=f(q)+\frac{J_{1,2}^2}{J_1}=2m(-q,q^3,q)+2m(-q,q^3,q^2)+\frac{J_{1,2}^2}{J_1}\\
&=2\Big [ m(-q^5,q^{12},q^4)+q^{-2}m(-q^{-1},q^{12},q^4)+0\Big ]\\
&\ \  +2\Big [ m(-q^5,q^{12},q^8)+q^{-2}m(-q^{-1},q^{12},q^8)-\frac{J_{6}J_{12}J_1^2}{\overline{J}_{0,3}J_{4,12}\overline{J}_{1,6}}\Big ]+\frac{J_{1,2}^2}{J_1}.&({\text{by Cor \ref{corollary:msplit-n=2}}})
\end{align*}
The penultimate equality of (\ref{equation:3rd-phi(q)}) then follows from (\ref{equation:mxqz-fnq-z}), (\ref{equation:mxqz-flip}), and elementary theta function properties.  For the final equality, we again use results of (\ref{equation:3rd-f(q)}) with the same identity from \cite[p. 63]{W3}:
\begin{align*}
2\phi(-q)=f(q)+\frac{J_{1,2}^2}{J_1}&=2m(-q,q^3,q)+2m(-q,q^3,q^2)+\frac{J_{1,2}^2}{J_1}\\
&=4m(-q,q^3,-1)-\frac{J_{3,6}^2}{J_1}+\frac{J_{1,2}^2}{J_1},
\end{align*}
where the last equality follows from Theorem \ref{theorem:changing-z-theorem}.  Focusing on the sum of theta quotients,
\begin{align*}
J_{1,2}^2-J_{3,6}^2=\frac{J_1^4}{J_2^2}-\frac{J_3^4}{J_6^2}=&\frac{J_1^2}{J_2^2}\Big [ J_1^2-\frac{J_3^4J_2^2}{J_1^2J_6^2}\Big ]=\frac{J_1^2}{J_2^2}\Big [ J_{1,3}^2-\overline{J}_{1,3}^2\Big ]
=\frac{J_1^2}{J_2^2}\cdot \frac{j(i;q^3)^2\overline{J}_{1,3}\overline{J}_{0,3}}{j(-iq;q^3)j(-iq^{-1};q^3)},
\end{align*}
where the last step follows from Proposition \ref{proposition:CHcorollary} with $q=q^3,\ a=i,\ b=q,\ c=-q,\ d=1.$  The result then follows from elementary theta function properties.

We prove (\ref{equation:3rd-psi(q)}).  The third equality is just an application of Proposition \ref{proposition:g-to-mxqz}.  To obtain the last expression we combine the results of (\ref{equation:3rd-f(q)}) with an identity from \cite[p. 63]{W3}: 
\begin{align*}
4\psi(-q)=-f(q)+\frac{J_{1,2}^2}{J_1}=-4m(-q,q^3,q)-\frac{J_{3,6}^2}{J_1}+\frac{J_{1,2}^2}{J_1}.
\end{align*}
The argument found at the end of the proof of (\ref{equation:3rd-phi(q)}) then gives the desired result.

We prove (\ref{equation:3rd-nu(q)}).  The second equality follows from (\ref{equation:g-alt-def}) and the third follows from Proposition \ref{proposition:g-to-mxqz} and (\ref{equation:mxqz-flip}).  To obtain the last two equalities, we combine the results of (\ref{equation:3rd-omega(q)}) with an identity from \cite[p. 63]{W3}.   Using Theorem \ref{theorem:changing-z-theorem} and simplfying,
\begin{align*}
\nu(q)&=-q\omega(q^2)+\frac{J_4^3}{J_2^2}
=q^{-1}m(q^2,q^{12},q^4)+q^{-1}m(q^2,q^{12},q^8)+\frac{J_4^3}{J_2^2}\\
&=q^{-1} m(q^2,q^{12},-q^3)+q^{-1} m(q^2,q^{12},-q^9)+\frac{J_4^3}{J_2^2}
-\frac{J_{12}^3\overline{J}_{1,12}}
{J_{4,12}\overline{J}_{3,12}}
\Big [ \frac{q^2\overline{J}_{3,12}}{J_{6,12}\overline{J}_{5,12}}+\frac{\overline{J}_{7,12}}{J_{2,12}\overline{J}_{11,12}}\Big ].
\end{align*}
The penultimate equality of (\ref{equation:3rd-nu(q)}) follows from showing that the sum of theta quotients is zero.   Here we rewrite the term in brackets using Proposition {\ref{proposition:CHcorollary}} with $q=q^{12},a=q^6,b=q^4,c=q^2,d=-q$.  The final equality of (\ref{equation:3rd-nu(q)}) then follows from Theorem \ref{theorem:changing-z-theorem}.

\noindent {\bf `5th order' functions}

For these identities, there are four types of proofs.  Functions $\Phi(q)$, $\Psi(q)$ are in the first; $\phi_0(q)$, $\phi_1(q)$, $\psi_0(q)$, $\psi_1(q)$ are in the second; $f_0(q)$, $f_1(q)$, $F_0(q)$, $F_1(q)$ are in the third; and $\chi_0(q)$, $\chi_1(q)$ are in the last.

For the first type,  proving (\ref{equation:5th-AuxPhi(q)}) and (\ref{equation:5th-AuxPsi(q)}) is just an application of Proposition \ref{proposition:g-to-mxqz}.  For (\ref{equation:5th-AuxPsi(q)}), we follow Proposition \ref{proposition:g-to-mxqz} with  (\ref{equation:mxqz-flip}).

For the second type, we prove (\ref{equation:5th-phi0(q)}).  The proofs of the other identities are similar.    The penultimate equality is just the respective mock theta conjecture.  For the last equality,  Proposition \ref{proposition:g-to-mxqz} gives
\begin{align*}
\phi_0(q)&=m(-q^7,-q^{15},q^2)-q^{-1}m(q^2,-q^{15},q^2)+\frac{J_{10}j(-q^2;-q^5)}{J_{2,10}}\\
&=m(-q^7,-q^{15},q^9)-q^{-1}m(q^2,-q^{15},q^9)+\frac{J_{10}j(-q^2;-q^5)}{J_{2,10}}\\
&\ \ \ \ +\frac{qj(-q^{15};-q^{45})^3j(q^7;-q^{15})}{j(q^2;-q^{15})j(q^9;-q^{15})}
\Big \{ \frac{j(q^{13};-q^{15})}{j(q^4,q^{11};-q^{15})}-q^{-1}\frac{j(q^3;-q^{15})}{j(q,-q^9;-q^{15})}\Big \},
\end{align*}
where the last line follows from Theorem \ref{theorem:changing-z-theorem}.  Showing that the sum of the three quotients is zero follows from Proposition \ref{proposition:CHcorollary} with the substitutions $q=-q^{15}, a=q^{10},b=q^7,c=q^6,d=q^4.$  For (\ref{equation:5th-phi1(q)}), (\ref{equation:5th-psi0(q)}), and (\ref{equation:5th-psi1(q)}), we argue in a similar fashion.  

For the third type, we prove (\ref{equation:5th-f0(q)}) as an example.  The proofs of the others are similar.  The second equality is just the respective mock theta conjecture.   Using Proposition \ref{proposition:g-to-mxqz}, we have
\begin{align*}
f_0&(q)=2m(q^{14},q^{30},q^4)+2q^{-2}m(q^4,q^{30},q^4)+\frac{J_{5,10}J_{2,5}}{J_1}\\
&=m(q^{14},q^{30},q^{14})+m(q^{14},q^{30},q^{29})+q^{-2}m(q^{4},q^{30},q^4)+q^{-2}m(q^4,q^{30},q^{19})+\frac{J_{5,10}J_{2,5}}{J_1}\\
&\ \ \ \ +q^2\frac{J_{30}^3J_{10,30}}{J_{4,30}J_{18,30}J_{16,30}}
-\frac{J_{30}^3J_{5,30}}{J_{4,30}J_{18,30}J_{1,30}}
-q^2\frac{J_{30}^3J_{15,30}J_{3,30}}{J_{4,30}J_{8,30}J_{19,30}J_{7,30}},
\end{align*}
where the last line follows from Theorem \ref{theorem:changing-z-theorem}.   It remains to show that the sum of quotients is zero, i.e.
\begin{align}
\frac{J_{5,10}J_{2,5}}{J_{1}}=\frac{J_{30}^3}{J_{4,30}}
\Big [ -q^2\frac{J_{10,30}}{J_{18,30}J_{16,30}}+\frac{J_{5,30}}{J_{18,30}J_{1,30}} +q^2\frac{J_{15,30}J_{3,30}}{J_{8,30}J_{19,30}J_{7,30}}.   \Big ]\label{equation:bunny}
\end{align}
Combining the first and third summands with Proposition \ref{proposition:CHcorollary} and the substitutions $q=q^{30}, a=q^{9},b=q^{17},c=q^6,d=q$, shows that (\ref{equation:bunny}) is equivalent to
\begin{align*}
\frac{J_{5,10}J_{2,5}}{J_1}=\frac{J_{30}^3J_{5,30}}{J_{4,30}J_{18,30}}\Big [ \frac{1}{J_{1,30}}-q^5\frac{J_{4,30}}{J_{16,30}J_{19,30}}\Big ],
\end{align*}
which follows from (\ref{equation:ThmAH1.1}) and the substitutions $n=2, q=q^{15},z=q,x=-q^5$.   For (\ref{equation:5th-f1(q)}), we argue as in (\ref{equation:5th-f0(q)}).     For (\ref{equation:5th-BigF0(q)})  we argue as in (\ref{equation:5th-f0(q)}) to reduce the identities to showing
\begin{align}
\frac{2J_{10}j(-q^5;q^{20})}{J_{4,10}}=\frac{J_{15}^3}{J_{2,15}}\Big [\frac{{\overline{J}}_{0,15}{\overline{J}}_{6,15}}{J_{4,15}{\overline{J}}_{2,15}{\overline{J}}_{4,15}}- \frac{J_5}{J_{6,15}J_{8,15}}\Big ]+\frac{J_{15}^3{\overline{J}}_{5,15}}{J_{2,15}J_{6,15}{\overline{J}}_{8,15}}.\label{equation:bunny2}
\end{align}
We combine the first and second summands using Proposition  \ref{proposition:CHcorollary} with the substitutions $q=q^{15}, a=q^{3},b=-q^{7},c=-q^3,d=-q$ to show that (\ref{equation:bunny2}) is equivalent to
\begin{align*}
\frac{2J_{10}j(-q^5;q^{20})}{J_{4,10}}=\frac{J_{15}^3\overline{J}_{5,15}}{J_{6,15}}\Big [ \frac{1}{\overline{J}_{2,15}J_{8,15}}+\frac{1}{J_{2,15}{\overline{J}}_{8,15}}\Big ].
\end{align*}
This follows from (\ref{equation:H1Thm1.2B}) with $q=q^{15}, x=q^8,y=q^2$.  For (\ref{equation:5th-BigF1(q)}), we argue as for (\ref{equation:5th-BigF0(q)}) and use (\ref{equation:H1Thm1.2A}).

For the fourth type, we prove (\ref{equation:5th-chi0(q)}).  The proof of (\ref{equation:5th-chi1(q)}) is similar and will be omitted.  The third equality is just the respective mock theta conjecture.  We prove the last line of (\ref{equation:5th-chi0(q)}).  We use Proposition  \ref{proposition:g-to-mxqz} and Theorem  \ref{theorem:changing-z-theorem} to obtain
\begin{align*}
&\chi_0(q)=2+3qg(q,q^5)-\frac{J_5^2J_{2,5}}{J_{1,5}^2}=2-3m(q^7,q^{15},q^2)-3q^{-1}m(q^2,q^{15},q^2)-\frac{J_5^2J_{2,5}}{J_{1,5}^2}\\
&=2-3m(q^7,q^{15},q^9)-3q^{-1}m(q^2,q^{15},q^4)+\frac{3J_{15}^3J_{7,15}J_{3,15}}{J_{2,15}J_{9,15}^2J_{1,15}}+\frac{3qJ_{15}^3J_{2,15}J_{8,15}}{J_{2,15}J_{4,15}^2J_{6,15}}-\frac{J_5^2J_{2,5}}{J_{1,5}^2}.
\end{align*}
The result then follows from Proposition \ref{proposition:CHcorollary} with $q=q^{15}, \ a=q^4, \ b=q, \ c=q^2, \ d=1.$  We show that the second to last line of (\ref{equation:5th-chi0(q)}) follows from the last line of (\ref{equation:5th-chi0(q)}).  We begin with
\begin{align*}
\chi_0(q)&=2-3m(q^7,q^{15},q^9)-3q^{-1}m(q^2,q^{15},q^4)+\frac{2J_5^2J_{2,5}}{J_{1,5}^2}\\
&=2-2m(q^7,q^{15},q^{12})-m(q^7,q^{15},q^9)-2q^{-1}m(q^2,q^{15},q^{12})-q^{-1}m(q^2,q^{15},q^9)\\
&\ \ \ \ -\frac{2qJ_{15}^3J_{2,15}}{J_{9,15}J_{1,15}J_{4,15}}-\frac{2J_{15}^3J_{8,15}}{J_{4,15}J_{6,15}J_{1,15}}+\frac{2J_{5}^2J_{2,5}}{J_{1,5}^2},
\end{align*}
by Theorem  \ref{theorem:changing-z-theorem}.  The result then follows from (\ref{equation:H1Thm1.0}) with $q=q^5, \ x=q.$

\noindent {\bf `6th order' functions}

To prove (\ref{equation:6th-phi(q)}),  we first recall identity \cite[(3.25)]{AH} and use Theorem  \ref{theorem:changing-z-theorem}:
\begin{align*}
{\overline{J}}_{1,3}\phi(q)&=2\sum_{r}\frac{q^{r(3r+1)/2}}{1+q^{3r}}=2j(-q^2;q^3)m(q,q^3,-q^2)=2j(-q^2;q^3)m(q,q^3,-1).
\end{align*}
To prove (\ref{equation:6th-psi(q)}) (resp. (\ref{equation:6th-rho(q)}), (\ref{equation:6th-sigma(q)})) we use identity $(3.23)$ (resp. $(4.11)$, $(4.12)$) of \cite{AH}.

We prove the last equality in each of (\ref{equation:6th-lambda(q)}) -- (\ref{equation:psibar}).   The penultimate equalities then follow from a straightforward application of Theorem \ref{theorem:changing-z-theorem} and elementary theta function properties.  For (\ref{equation:6th-lambda(q)}) (resp. (\ref{equation:6th-mu(q)})),
we recall identity $(4.19)$ (resp. $(4.20)$) of \cite{AH} and make the substitution $z=-q^2$ (resp. $z=-1$).   For (\ref{equation:6th-gamma(q)}), we recall identity \cite[(4.27)]{AH} and make the substitution $z=-q$ to obtain
\begin{align*}
{\overline{J}}_{1,3}\gamma(q)&=-\frac{J_1^2{\overline{J}}_{1,3}}{J_3{\overline{J}_{0,1}}}\Big [ \omega j(-\omega ^2q;q)+\omega^2j(-\omega q;q) \Big ]+3{\overline{J}}_{1,3}m(q,q^3,-q)\\
&=\frac{2J_1^2J_{1,2}J_6{\overline{J}}_{1,3}}{J_3^2\overline{J}_{0,1}}+3{\overline{J}}_{1,3}m(q,q^3,-q),
\end{align*}
where the last line follows from the fact that $j(-\omega, q)=(1+\omega){J_{1,2}J_6}/{J_3}$.  The result follows.  For  (\ref{equation:phibar}), we rewrite identity \cite[(2.14)]{BC}.  For (\ref{equation:psibar}), we recall identity \cite[(2.13)]{BC} and argue in a similar fashion.

\noindent {\bf `7th order' functions}

  We prove $(\ref{equation:7th-F0(q)})$.  The proofs for  (\ref{equation:7th-F1(q)}) and  (\ref{equation:7th-F2(q)}) are similar.  From \cite[(0.8)]{H2}, we have
\begin{align*}
{\mathcal{F}}_0(q)&=2+2qg(q,q^7)-\frac{J_{3,7}^2}{J_1}
=2-2m(q^{11},q^{21},q^2)-2q^{-1}m(q^4,q^{21},q^2) -\frac{J_{3,7}^2}{J_1} \\
&=2q^{-10}m(q^{-10},q^{21},q^2)-2q^{-1}m(q^4,q^{21},q^2)-\frac{J_{3,7}^2}{J_1}\\
&=2m(q^{10},q^{21},q^{-2})-2q^{-1}m(q^4,q^{21},q^2)-\frac{J_{3,7}^2}{J_1},
\end{align*}
where the last three equalities follow from Proposition \ref{proposition:g-to-mxqz} and identities (\ref{equation:mxqz-fnq-x}), (\ref{equation:mxqz-flip}).  Using Theorem \ref{theorem:changing-z-theorem} and elementary theta function properties, we obtain
\begin{align*}
{\mathcal{F}}_0(q)&=m(q^{10},q^{21},q^9)+m(q^{10},q^{21},q^{-9})-q^{-1}m(q^{4},q^{21},q^9)-q^{-1}m(q^{4},q^{21},q^{-9})\\
&\ \ \ \ +\frac{J_{21}^3J_{11,21}J_{17,21}}{J_{2,21}J_{8,21}J_{9,21}J_{19,21}}
+q\frac{J_{21}^3J_{11,21}J_{3,21}}{J_{2,21}J_{6,21}J_{9,21}J_{5,21}}
-\frac{J_{3,7}^2}{J_1}\\
&=2m(q^{10},q^{21},q^9)-2q^{-1}m(q^{4},q^{21},q^{-9})\\
&\ \ \ \ \ +2\frac{J_{21}^3J_{11,21}J_{17,21}}{J_{2,21}J_{8,21}J_{9,21}J_{19,21}}
+2q\frac{J_{21}^3J_{11,21}J_{3,21}}{J_{2,21}J_{6,21}J_{9,21}J_{5,21}}
-\frac{J_{3,7}^2}{J_1}.
\end{align*}
To establish the two remaining identities in (\ref{equation:7th-F0(q)}), we use Proposition \ref{proposition:CHcorollary} with the specialization $q=q^{21},a=q^3,b=q^5,c=q,d=q^2.$

\noindent {\bf `8th order' functions}

Using the generalized Lambert series (1.3) -- (1.6) of \cite{GM}, we first prove the expressions in (\ref{equation:8th-U0(q)}) -- (\ref{equation:8th-V1(q)}).  We then combine these expressions in (\ref{equation:8th-U0(q)}) -- (\ref{equation:8th-V1(q)}) with the eighth order identities (1.7) and (1.8) of \cite{GM} in order to prove the expressions in (\ref{equation:8th-S0(q)}) -- (\ref{equation:8th-T1(q)}).

For (\ref{equation:8th-U0(q)}) (resp. (\ref{equation:8th-U1(q)}), (\ref{equation:8th-V1(q)})) we use identities (1.3) (resp. (1.4), (1.6)) of \cite{GM} and argue in a manner similar to that for identity (\ref{equation:2nd-B(q)}).  For (\ref{equation:3rd-chi(q)5}) we prove the penultimate equality, the last equality then follows from Theorem \ref{theorem:changing-z-theorem}.  We begin by recalling (1.5) of \cite{GM} and argue as for (\ref{equation:2nd-B(q)}) to obtain
\begin{align*}
V_0(q)=&-1+2\frac{(-q^2;q^4)_{\infty}}{(q^4;q^4)_{\infty}}\Big [ j(q^6;q^8) m(q^4,q^8,q^6)+qj(q^{10};q^8)m(1,q^8,q^{10})\Big ]\\
=& -1 + 2m(q^4,q^8,q^6)-2q^{-1}m(1,q^8,q^{2}),
\end{align*}
where the last line follows from (\ref{equation:mxqz-fnq-z}).  Using Theorem \ref{theorem:changing-z-theorem} and Corollary \ref{corollary:mxqz-eval}, we have
\begin{align*}
m(q^4,q^8,q^6)&=m(q^4,q^8,-1)+\frac{J_8^3\overline{J}_{2,8}^2}{J_{2,8}^2\overline{J}_{0,8}\overline{J}_{4,8}}
=\frac{1}{2}+\frac{J_8^3\overline{J}_{2,8}^2}{J_{2,8}^2\overline{J}_{0,8}\overline{J}_{4,8}}.
\end{align*}
Applying Theorem \ref{theorem:changing-z-theorem} twice more,
\begin{align*}
V_0(q)=-q^{-1}m(1,q^8,q)-q^{-1}m(1,q^8,q^3)+\frac{qJ_{8}^3J_{1,8}J_{3,8}}{J_{2,8}^2J_{3,8}^2}-
\frac{J_{8}^3J_{1,8}J_{3,8}}{J_{2,8}^2J_{1,8}^2}+2\frac{J_8^3\overline{J}_{2,8}^2}{J_{2,8}^2\overline{J}_{0,8}\overline{J}_{4,8}}.
\end{align*}
It remains to show that the sum of three theta quotients is zero; however, this follows from Proposition \ref{proposition:CHcorollary} with $q=q^8$, $a=q^2$, $b=q^4$, $c=q$, $d=1$ .

With the above information, we now proceed to the expressions for (\ref{equation:8th-S0(q)}) -- (\ref{equation:8th-T1(q)}).   For (\ref{equation:8th-S0(q)}), we begin by recalling \cite[(1.7)]{GM}, which states
\begin{equation*}
U_0(q)=S_0(q^2)+qS_1(q^2),
\end{equation*}
hence,
\begin{equation*}
2S_0(q^2)=U_0(q)+U_0(-q).
\end{equation*}
Recalling (\ref{equation:8th-U0(q)}) and then using Corollary \ref{corollary:mxqz-flip-xz}, Theorem \ref{theorem:msplit-general-n}, and (\ref{equation:jsplit}) in reverse yields
\begin{align*}
U_0&(q)=2m(-q,q^4,-1)=2m(-q,q^4,q^3)\\
&=2m(-q^6,q^{16},-q^4)+2q^{-3}m(-q^{-2},q^{16},-q^4)\\
&\ \ \ \ -2\frac{q^4J_8^3}{j(-q^4;q^4)j(q^{10};q^8)j(-q^4;q^{16})}\Big \{
\frac{j(q^{13};q^8)j(-q^2;q^{16})}{j(q^3;q^8)}+q^4\frac{j(q^{17};q^8)j(-q^{10};q^{16})}{j(q^7;q^8)}\Big \}\\
&=2m(-q^6,q^{16},-q^4)+2q^{-3}m(-q^{-2},q^{16},-q^4)+2\frac{J_8^3J_{1,4}}{\overline{J}_{0,4}J_{2,8}\overline{J}_{4,16}}.
\end{align*}
Similarly,
\begin{equation*}
U_0(-q)=2m(-q^6,q^{16},-q^{12})-2q^{-3}m(-q^{-2},q^{16},-q^{12})-2\frac{J_8^3\overline{J}_{1,4}}{\overline{J}_{0,4}J_{2,8}\overline{J}_{4,16}}.
\end{equation*}
Thus by expanding $J_{1,4}$ and $\overline{J}_{1,4}$ with (\ref{equation:jsplit}) where $m=2$,
{\allowdisplaybreaks \begin{align*}
S_0(q^2)&=\frac12 U_0(q)+\frac12 U_0(-q)=m(-q^6,q^{16},-q^4)+q^{-3}m(-q^{-2},q^{16},-q^4)\\
&\ \ \ \ +m(-q^6,q^{16},-q^{12})-q^{-3}m(-q^{-2},q^{16},-q^{12})
-\frac{2qJ_8^3\overline{J}_{2,16}}{\overline{J}_{0,4}J_{2,8}\overline{J}_{4,16}}\\
&=m(-q^6,q^{16},-q^4)+m(-q^6,q^{16},-q^{12})+\frac{qJ_{16}^3J_{8,16}\overline{J}_{2,16}}{\overline{J}_{4,16}^2J_{2,16}J_{10,16}}-\frac{2qJ_8^3\overline{J}_{2,16}}{\overline{J}_{0,4}J_{2,8}\overline{J}_{4,16}},
\end{align*}}%
where the last line follows from Theorem \ref{theorem:changing-z-theorem}.  The penultimate equality of (\ref{equation:8th-S0(q)}) then follows from elementary theta function properties.  For the last equality of (\ref{equation:8th-S0(q)}), we have
\begin{align*}
S_0(q)&=m(-q^3,q^8,-q^2)+m(-q^3,q^8,-q^6)\\
&=2m(-q^3,q^8,-1)-\frac{J_8^3J_{2,8}}{\overline{J}_{2,8}\overline{J}_{0,8}J_{3,8}J_{1,8}J_{5,8}}\Big [ 
{\overline{J}_{5,8}}{J_{1,8}}-{\overline{J}_{1,8}}{J_{5,8}}\Big ] &({\text{by Thm \ref{theorem:changing-z-theorem}}})\\
&=2m(-q^3,q^8,-1)-\frac{J_8^3J_{2,8}}{\overline{J}_{2,8}\overline{J}_{0,8}J_{3,8}J_{1,8}J_{5,8}}\Big [ 
-2qJ_{4,16}J_{14,16}\Big ], &({\text{by (\ref{equation:H1Thm1.2A})}})
\end{align*}
and the result follows.  The proof for (\ref{equation:8th-S1(q)}) is similar.  For (\ref{equation:8th-T0(q)}), we use \cite[(1.8)]{GM}:
\begin{align*}
2T_0(q^2)&=U_1(q)+U_1(-q)=-m(-q,q^4,-q^2)-m(q,q^4,-q^2)& ({\text{by (\ref{equation:8th-U1(q)})}})\\
&=-m(-q,q^4,q)-m(q,q^4,-q)=-2m(-q^6,q^{16},q^4).& ({\text{by  Cors \ref{corollary:mxqz-flip-xz}, \ref{corollary:msplit-n=2}}})
\end{align*}
The argument for  (\ref{equation:8th-T1(q)}) is analogous.

\noindent {\bf `10th order' functions}

  For each set of identities, we prove the last equality.  The penultimate equalities then follow from Theorem \ref{theorem:changing-z-theorem} and elementary theta function properties.  To prove (\ref{equation:10th-phi(q)}) (resp. (\ref{equation:10th-psi(q)})), we recall information from \cite[pp. 525, 534]{C1} (resp. \cite[pp. 525, 533]{C1}) and use Proposition \ref{proposition:h-to-mxqz}.   To show (\ref{equation:10th-BigX(q)}), we recall information from \cite[pp. 183, 222]{C2} to write
\begin{align*}
X(q)&=2qk(q,q^5)-\frac{J_5J_{10}J_{2,5}}{J_{2,10}J_{1,5}}
= 2m(-q^2,q^5,q^{-2})+\frac{J_{5,10}^2}{J_{2,5}}-\frac{J_{5}J_{10}J_{2,5}}{J_{2,10}J_{1,5}}\\
&=2m(-q^2,q^5,q^4)-\frac{J_{5}J_{10}J_{2,5}}{J_{2,10}J_{1,5}},
\end{align*}
where the second equality follows from Proposition \ref{proposition:CH-k(x,q)}, and the last equality follows from Theorem \ref{theorem:changing-z-theorem}.  The result follows.  Using the information in \cite[pp. 183, 223]{C2} and arguing similarly proves (\ref{equation:10th-chi(q)}).


\section{Proofs of Theorems  \ref{theorem:masterFnp} and \ref{theorem:main-acdivb} }\label{section:maintheorems-proofs}
\subsection{Preliminaries}

We establish elementary properties of Hecke-type sums.   We state and prove functional equations which $f_{a,b,c}$, $g_{a,b,c}$, and $h_{a,b,c}$ satisfy. We then prove results concerning the poles and residues of $g_{a,b,c}$ and $h_{a,b,c}$.

\begin{proposition} \label{proposition:fabc-mod2} For $x,y\in\mathbb{C}^*$
\begin{align}
f_{a,b,c}(x,y,q)&=f_{a,b,c}(-x^2q^a,-y^2q^c,q^4)-xf_{a,b,c}(-x^2q^{3a},-y^2q^{c+2b},q^4)\label{equation:fabc-mod2}\\
&\ \ \ \ -yf_{a,b,c}(-x^2q^{a+2b},-y^2q^{3c},q^4)+xyq^bf_{a,b,c}(-x^2q^{3a+2b},-y^2q^{3c+2b},q^4).\notag
\end{align}
\end{proposition}
\begin{proof}  Break up the double sum in Definition \ref{definition:fabc-def} into four parts depending on the parity of $r$ and $s$.
\end{proof}
Propositions \ref{proposition:H7eq1.14} and \ref{proposition:f-functionaleqn} follow from identities $(1.14)$ and $(1.15)$ of \cite{H2} respectively, which read
\begin{align}
\sum_{sg(r)=sg(s)}sg(r)c_{r,s}=&-\sum_{sg(r)=sg(s)}sg(r)c_{-1-r,-1-s},\label{equation:1.14H2}\\
\sum_{sg(r)=sg(s)}sg(r)c_{r,s}=&\sum_{sg(r)=sg(s)}sg(r)c_{r+\ell,s+k}+\sum_{r=0}^{\ell -1}\sum_{s}c_{r,s}+\sum_{s=0}^{k -1}\sum_{r}c_{r,s}.\label{equation:1.15H2}
\end{align}

\begin{proposition}  \label{proposition:H7eq1.14}For $x,y\in\mathbb{C}^*$
\begin{equation}
f_{a,b,c}(x,y,q)=-\frac{q^{a+b+c}}{xy}f_{a,b,c}(q^{2a+b}/x,q^{2c+b}/y,q).\label{equation:H7eq1.14}
\end{equation}
\end{proposition}

\begin{proposition}  \label{proposition:f-functionaleqn}For $x,y\in\mathbb{C}^*$ and $\ell, k \in \mathbb{Z}$
\begin{align}
f_{a,b,c}(x,y,q)&=(-x)^{\ell}(-y)^kq^{a\binom{\ell}{2}+b\ell k+c\binom{k}{2}}f_{a,b,c}(q^{a\ell+bk}x,q^{b\ell+ck}y,q) \notag\\
&\ \ \ \ +\sum_{m=0}^{\ell-1}(-x)^mq^{a\binom{m}{2}}j(q^{mb}y;q^c)+\sum_{m=0}^{k-1}(-y)^mq^{c\binom{m}{2}}j(q^{mb}x;q^a),\label{equation:Gen1}
\end{align}
where when $b<a$, we follow the usual convention:
\begin{equation}
\sum_{r=a}^{b} c_r:=-\sum_{r=b+1}^{a-1} c_r.\label{equation:sumconvention}
\end{equation}
\end{proposition}

\begin{corollary} \label{corollary:fabc-funceqnspecial}We have two simple specializations:
\begin{align}
f_{a,b,c}(x,y,q) =&-yf_{a,b,c}(q^bx,q^cy,q)+j(x;q^a),\label{equation:fabc-fnq-1}\\
f_{a,b,c}(x,y,q) =&-xf_{a,b,c}(q^ax,q^by,q)+j(y;q^c).\label{equation:fabc-fnq-2}
\end{align}
\end{corollary}

The functions $f_{a,b,c}(x,y,q)$ and $g_{a,b,c}(x,y,q,-1,-1)$ satisfy the same functional equation.

\begin{proposition}  \label{proposition:Genfg-functional}The functions $f_{a,b,c}(x,y,q)$ and $g_{a,b,c}(x,y,q,-1,-1)$ satisfy
\begin{align*}
G(q^{b^2-ac}x,y,q)&=q^{c\binom{b+1}{2}-a\binom{c+1}{2}}\frac{(-x)^c}{(-y)^b}G(x,y,q)
+\sum_{r=0}^{c-1}(-x)^rq^{a\binom{m}{2}}q^{r(b^2-ac)}j(q^{rb}y;q^c)\\
&\ \ \ -q^{c\binom{b+1}{2}-a\binom{c+1}{2}}\frac{(-x)^c}{(-y)^b}\sum_{r=0}^{b-1}
(-y)^rq^{c\binom{r}{2}}j(q^{rb}x;q^a).
\end{align*} 
The functional equation with respect to $y$ is obtained by interchanging $x$ with $y$ and then $a$ with $c$.
\end{proposition}
\begin{proof}[Proof of Proposition \ref{proposition:Genfg-functional}]  For $f_{a,b,c}(x,y,q)$, we specialize $\ell=-c, k=b$ in Proposition \ref{proposition:f-functionaleqn}, rewrite the first sum of (\ref{equation:Gen1}) using the summation convention (\ref{equation:sumconvention}) of Proposition \ref{proposition:f-functionaleqn}, and then rearrange terms
\begin{align}
f_{a,b,c}(q^{b^2-ac}x,y,q)&=q^{c\binom{b+1}{2}-a\binom{c+1}{2}}\frac{(-x)^c}{(-y)^b}\Big [ f_{a,b,c}(x,y,q)\notag\\
&\ \ \ \ +\sum_{r=-c}^{-1}(-x)^rq^{a\binom{r}{2}}j(q^{rb}y;q^c)-\sum_{r=0}^{b-1}(-y)^rq^{c\binom{r}{2}}j(q^{rb}x;q^a)\Big ].\label{equation:squirrel}
\end{align}
We rewrite the first sum of the right-hand side of (\ref{equation:squirrel}). We replace $r$ with $r-c$ and simplify to obtain
\begin{align*}
\sum_{r=-c}^{-1}(-x)^rq^{a\binom{r}{2}}j(q^{rb}y;q^c)
=q^{a\binom{c+1}{2}-c\binom{b+1}{2}}\frac{(-y)^b}{(-x)^c}\sum_{r=0}^{c-1}(-x)^rq^{a\binom{r}{2}}q^{r(b^2-ac)}j(q^{rb}y;q^c),
\end{align*}
and the result follows.  We recall the definition of $g_{a,b,c}(x,y,q,-1,-1)$ from line (\ref{equation:mdef-2}).  We consider each of the two sums separately.  Applying (\ref{equation:mxqz-fnq-x}) to the second sum in (\ref{equation:mdef-2}) yields
\begin{align}
\sum_{t=0}^{c-1}&q^{t(b^2-ac)}(-x)^tq^{a\binom{t}{2}}j(q^{bt}y;q^c)+q^{c\binom{b+1}{2}-a\binom{c+1}{2}}\frac{(-x)^c}{(-y)^b}\cdot\Big [ \notag \\ 
& \sum_{t=0}^{c-1}(-x)^tq^{a\binom{t}{2}}j(q^{bt}y;q^c)m\Big (-q^{c\binom{b+1}{2}-a\binom{c+1}{2}-t(b^2-ac)}\frac{(-x)^c}{(-y)^b},q^{c(b^2-ac)},-1\Big )\Big ].\label{equation:gf-sumIA}
\end{align}
Applying (\ref{equation:1.8}) to the first sum in (\ref{equation:mdef-2}) and then replacing $t$ with $t-b$ produces
\begin{align*}
q^{c\binom{b+1}{2}-a\binom{c+1}{2}}&\frac{(-x)^c}{(-y)^b}\\
&\cdot \sum_{t=b}^{a+b-1}(-y)^tq^{c\binom{t}{2}}j(q^{bt}x;q^a)m\Big (-q^{a\binom{b+1}{2}-c\binom{a+1}{2}-t(b^2-ac)}\frac{(-y)^a}{(-x)^b},q^{a(b^2-ac)},-1\Big ).
\end{align*}
With the convention (\ref{equation:sumconvention}) in mind, we write $\sum_{t=b}^{a+b-1}=\sum_{t=b}^{a-1}+\sum_{t=a}^{a+b-1}$.  In the second sum, we replace $t$ with $t+a$ and then use (\ref{equation:1.8}) and (\ref{equation:mxqz-altdef1}).  Simplifying and adding to (\ref{equation:gf-sumIA}) produces the result.
\end{proof}

\begin{proposition}  \label{proposition:Genfh-functional}If $a$ and $c$ divide $b$, then both $f_{a,b,c}(x,y,q)$ and $h_{a,b,c}(x,y,q,-1,-1)$ satisfy
\begin{align*}
G(q^{b^2/c-a}x,y,q)=&q^{c\binom{b/c+1}{2}-a}\frac{(-x)}{(-y)^{b/c}}\Big [ G(x,y,q)
-\sum_{r=0}^{b/c-1}
(-y)^rq^{c\binom{r}{2}}j(q^{rb}x;q^a)\Big ]
+j(y;q^c).
\end{align*} 
The functional equation with respect to $y$ is obtained by interchanging $x$ with $y$ and then $a$ with $c$.
\end{proposition}

\begin{proof}[Proof of Proposition \ref{proposition:Genfh-functional}]
Showing this for $f_{a,b,c}(x,y,q)$ follows from Proposition \ref{proposition:f-functionaleqn} with the specializations $\ell=-1, k=b/c$, where we use the summation convention (\ref{equation:sumconvention}) to rewrite the sum from $0$ to $-2$.  Recalling the definition of $h_{a,b,c}(x,y,q,-1,-1)$ from Theorem \ref{theorem:main-acdivb}, we have
\begin{align}
h_{a,b,c}&(q^{b^2/c-a}x,y,q,-1,-1)=j(q^{b^2/c-a}x;q^a)m\Big ( -\frac{q^{a\binom{b/a+1}{2}-c}}{q^{b/c(b^2/a-c)}}\cdot \frac{(-y)}{(-x)^{b/a}},q^{b^2/a-c},-1\Big )\notag\\
&\ \ \ \ +j(y;q^c)m\Big ( -{q^{c\binom{b/c+1}{2}-a}}{q^{(b^2/c-a)}}\frac{(-x)}{(-y)^{b/c}},q^{b^2/c-a},-1\Big ).\label{equation:h-sum}
\end{align}
Using (\ref{equation:mxqz-fnq-x}), the second term of the right-hand side of (\ref{equation:h-sum}) becomes
\begin{align}
j(y;q^c)+{q^{c\binom{b/c+1}{2}-a}}\frac{(-x)}{(-y)^{b/c}}j(y;q^c)m\Big ( -{q^{c\binom{b/c+1}{2}-a}}\frac{(-x)}{(-y)^{b/c}},q^{b^2/c-a},-1\Big ).\label{equation:h-sumII}
\end{align}
We now focus on the first term of the right-hand side of (\ref{equation:h-sum}).  Using (\ref{equation:1.8}), we have
\begin{align*}
j(q^{b^2/c-a}x;q^a)=j(q^{a(b^2/ac-1)}x;q^a)=(-x)^{-(b^2/ac-1)}q^{-a\binom{b^2/ac-1}{2}}j(x;q^a).
\end{align*}
Iterating equation (\ref{equation:mxqz-altdef1}) $b/c$ times yields
\begin{align*}
m&\Big ( -\frac{q^{a\binom{b/a+1}{2}-c}}{q^{b/c(b^2/a-c)}} \frac{(-y)}{(-x)^{b/a}},q^{b^2/a-c},-1\Big )=-\sum_{k=0}^{b/c-1}\Big (\frac{q^{a\binom{b/a+1}{2}-c}}{q^{b/c(b^2/a-c)}}\frac{(-y)}{(-x)^{b/a}}\Big )^{-(k+1)}q^{-(b^2/a-c)\binom{k+1}{2}}\\
&+\Big (\frac{q^{a\binom{b/a+1}{2}-c}}{q^{b/c(b^2/a-c)}}\frac{(-y)}{(-x)^{b/a}}\Big )^{-b/c}q^{-(b^2/a-c)\binom{b/c}{2}}m\Big ( -{q^{a\binom{b/a+1}{2}-c}}\frac{(-y)}{(-x)^{b/a}},q^{b^2/a-c},-1\Big ).
\end{align*}
Simplifying shows that
\begin{align*}
j&(q^{b^2/c-a}x;q^a)m\Big ( -\frac{q^{a\binom{b/a+1}{2}-c}}{q^{b/c(b^2/a-c)}}\cdot \frac{(-y)}{(-x)^{b/a}},q^{b^2/a-c},-1\Big )\\
&={q^{c\binom{b/c+1}{2}-a}}\frac{(-x)}{(-y)^{b/c}}j(x;q^a)m\Big ( -{q^{a\binom{b/a+1}{2}-c}}\frac{(-y)}{(-x)^{b/a}},q^{b^2/a-c},-1\Big )\\
&\ \ \ \ -(-x)^{-(b^2/ac-1)}q^{-a\binom{b^2/ac-1}{2}}\sum_{k=0}^{b/c-1}\Big (\frac{q^{a\binom{b/a+1}{2}-c}}{q^{b/c(b^2/a-c)}}\cdot \frac{(-y)}{(-x)^{b/a}}\Big )^{-(k+1)}q^{-(b^2/a-c)\binom{k+1}{2}}j(x;q^a).
\end{align*}
In the second line, use (\ref{equation:1.8}) to write
\begin{equation*}
j(x;q^a)=(-x)^{(b/c-k-1)b/a}q^{a\binom{(b/c-k-1)b/a}{2}}j(q^{(b/c-k-1)b}x;q^a).
\end{equation*}
Replace $k$ with $b/c-1-k$, simplify, and the result follows.
\end{proof}
We collect results on the poles and residues of $g_{a,b,c}$ and $h_{a,b,c}$.
\begin{proposition}\label{proposition:g-residues}  Fix a generic $y\in \mathbb{C}^*$ and let  $a,$ $b,$ and $c$ be positive integers with $b^2>ac$.  The function 
$g_{a,b,c}(x,y,q,-1,-1)$ is meromorphic for $x\ne0$ and has poles at points $x_0$, where $x_0$ satisfies at least one of the following two conditions:
\begin{align*}
\textup{I}.& \ \ q^{a\binom{b+1}{2}-c\binom{a+1}{2}-t_1(b^2-ac)}(-y)^a(-x_0)^{-b}=q^{ka(b^2-ac)},\\
\textup{II}.& \ \ q^{c\binom{b+1}{2}-a\binom{c+1}{2}-t_2(b^2-ac)}(-x_0)^c(-y)^{-b}=q^{kc(b^2-ac)},
\end{align*}
where $t_1,t_2,k\in \mathbb{Z}$, $0\le t_1<a-1$ and $0\le t_2<c-1$.  If $x_0$ satisfies \textup{I} or \textup{II} exclusively, then it is a simple pole with respective residue
\begin{align*}
\textup{I}.&\ \ x_0(-y)^{t_1}q^{c\binom{t_1}{2}+a(b^2-ac)\binom{k}{2}}j(q^{bt_1}x_0,q^a)/\big ({b\cdot \overline{J}_{0,a(b^2-ac)}}\big ),\\
\textup{II}.&\ \ {(-x_0)^{t_2+1}q^{a\binom{t_2}{2}}j(q^{bt_2}y,q^c)}/{\big (c\cdot \overline{J}_{0,c(b^2-ac)}}\big ),
\end{align*}
where the residues for type \textup{II} have only been computed at $k=0.$  Given the residue at $k=0$ for poles of type \textup{II}, one can use the functional equation of Proposition \ref{proposition:Genfg-functional} to compute the residue for general $k\in\mathbb{Z}.$
\end{proposition}
\begin{proposition}\label{proposition:h-residues} Fix a generic $y\in \mathbb{C}^*$ and let $a,$ $b,$ and $c$ be positive integers with $b^2>ac$ and $b$ divisible by $a,c$.  The function $h_{a,b,c}(x,y,q,-1,-1)$ is meromorphic for $x\ne0$ and has poles at points $x_0$, where $x_0$ satisfies at least one of the following two conditions:
\begin{align*}
\textup{I}.& \ \ q^{c\binom{b/c+1}{2}-a}(-x_0)(-y)^{-b/c}=q^{k(b^2/c-a)},\\
\textup{II}.& \ \ q^{a\binom{b/a+1}{2}-c}(-y)(-x_0)^{-b/a}=q^{k(b^2/a-c)},
\end{align*}
where $k\in \mathbb{Z}$.  If $x_0$ satisfies \textup{I} or \textup{II} exclusively, then it is a simple pole with respective residue
\begin{align*}
\textup{I}.&\ \ {-x_0j(y,q^c)}/\overline{J}_{0,b^2/c-a},\\
\textup{II}.&\ \ x_0q^{(b^2/a-c)\binom{k}{2}}j(x_0,q^a)/\big (b/a\cdot \overline{J}_{0,b^2/a-c}\big ),
\end{align*}
where the residues for type \textup{I} have only been computed at $k=0.$ Given the residue at $k=0$ for poles of type \textup{I}, one can use the functional equation of Proposition \ref{proposition:Genfg-functional} to compute the residue for general $k\in\mathbb{Z}.$
\end{proposition}
\begin{proof}[Proofs of Propositions \ref{proposition:g-residues} and \ref{proposition:h-residues}]  Both proofs are similar, so we will only prove Proposition \ref{proposition:g-residues}.  The poles follow from the definition of $g_{a,b,c}(x,y,q,-1,-1)$.  For poles of type \textup{I}, the residue follows from using Proposition \ref{proposition:mnp2Lerch-residue}.   For poles of type \textup{II}, we use (\ref{equation:mxqz-flip}) and Proposition \ref{proposition:mnp2Lerch-residue}, and then use the fact that $q^{c\binom{b+1}{2}-a\binom{c+1}{2}-t_2(b^2-ac)}(-x_0)^c(-y)^{-b}=1$ when $k=0.$
\end{proof}


\subsection{Proof of Theorem \ref{theorem:masterFnp}}\label{subsection:thm-main}
We first prove technical results.  The first lemma is straightforward.
\begin{lemma} \label{lemma:Fnp-thetafunctional}For generic $x,y\in\mathbb{C}^*$
\begin{equation*}
\theta_{n,p}(q^{p(2n+p)}x,y,q)=q^{n(np+\binom{p+1}{2})}{(-x)^n}{(-y)^{-(n+p)}}\theta_{n,p}(x,y,q).
\end{equation*}
The functional equation with respect to $y$ is obtained by interchanging $x$ with $y$.
\end{lemma}
\begin{lemma}  \label{lemma:thetamasterFnp-residue}Fix a generic $y\in \mathbb{C}^*$.  The function $\theta_{n,p}(x,y,q)$ is meromorphic for $x\ne0$ and has poles at points $x_0$, where $x_0$ satisfies at least one of the following conditions:
\begin{align}
\textup{I}.& \ \ \ q^{sp(2n+p)+p(n+p)/2}{(-x_0)^{n+p}}{(-y)^{-n}}=q^{\ell p^2(2n+p)},\label{equation:Fnp-poletypeI}\\
\textup{II}.&  \ \ \ q^{rp(2n+p)+p(n+p)/2}{(-y)^{n+p}}{(-x_0)^{-n}}=q^{\ell p^2(2n+p)}.\label{equation:Fnp-poletypeII}
\end{align}
If $x_0$ satisfies either \textup{I} or \textup{II} exclusively, then it is a simple pole with respective residue
\begin{align*}
\textup{I}.&\ \ \ \frac{-x_0\cdot y^{s-\ell p +(n+1)/2}}{ (n+p)\cdot x_0^{s-\ell p +(n+1)/2}\cdot q^{p(s-\ell p +(n+1)/2)^2}}\cdot j(q^{p(s-\ell p +(n+1)/2)}x_0;q^n),\\
\textup{II}.&\ \ \ \frac{x_0^{r-\ell p+(n+1)/2+1}}{n\cdot y^{r-\ell p+(n+1)/2}\cdot q^{p(r-\ell p+(n+1)/2)^2}}
\cdot j(q^{p(r-\ell p+(n+1)/2)}y;q^{n}).
\end{align*}
\end{lemma}

\begin{proof}[Proof of Lemma \ref{lemma:thetamasterFnp-residue}]

We prove the residue for poles of type \textup{I}.  Using Proposition \ref{proposition:H1Thm1.3},
\begin{align}
&\lim_{x\rightarrow x_0}(x-x_0)\theta_{n,p}(x,y,q)\label{equation:thetaresidue-FnpII}\\
&=\sum_{r^*=0}^{p-1}q^{n\binom{r-(n-1)/2}{2}+(n+p)\big (r-(n-1)/2\big )\big (s+(n+1)/2\big )+n\binom{s+(n+1)/2}{2}}{(-x_0)^{r-(n-1)/2}}{ (-y)^{s+(n+1)/2}} \notag\\
&\ \ \ \ \cdot \frac{j\big (-q^{np(s-r)}x_0^n/y^n;q^{np^2}\big )j\big (q^{p(2n+p)(r+s)+p(n+p)}x_0^py^p;q^{p^2(2n+p)}\big )}{j\big ((-1)^pq^{p(2n+p)r+p(n+p)/2}y^{n+p}/x_0^n;q^{p^2(2n+p)}\big )}\cdot \frac{(-1)^{\ell+1}q^{p^2(2n+p)\binom{\ell}{2}}x_0}{ (n+p)}.\notag
\end{align}
We rewrite two theta functions of (\ref{equation:thetaresidue-FnpII}).  Using (\ref{equation:Fnp-poletypeI}) to substitute for $x_0^n/y^n$ and then using (\ref{equation:1.8}),
\begin{align*}
j(-q^{np(s-r)}x_0^n/y^n;q^{np^2})&=j(-q^{np^2-np(s-r)}y^n/x_0^n;q^{np^2})\\
&=j(-q^{np^2(1-2\ell)+np(r+s)+sp^2+p(n+p)/2-\ell p^3}(-x_0)^p;q^{np^2})\\
&=-q^{-np^2\binom{1-2\ell}{2}}\big ( (-1)^{p+1}q^{np(r+s)+sp^2+p(n+p)/2-\ell p^3}x_0^p\big )^{2\ell-1} \\
&\ \ \ \ \ \cdot j(-q^{np(r+s)+sp^2+p(n+p)/2-\ell p^3}(-x_0)^p;q^{np^2})\\
&=-q^{-np^2\binom{1-2\ell}{2}}\big ( (-1)^{p+1}q^{np(r+s)+sp^2+p(n+p)/2-\ell p^3}x_0^p\big )^{2\ell-1}\\
& \ \ \ \ \cdot  j(-q^{n\binom{p}{2}+np(r+s+1-p)}(-q^{p(s-\ell p+(n+1)/2)}x_0)^p;q^{np^2}).
\end{align*}
Using (\ref{equation:Fnp-poletypeI}) to substitute for $x_0^p$ and then using (\ref{equation:1.8}) yields
\begin{align*}
j(q^{p(2n+p)(r+s)+p(n+p)}x_0^py^p;q^{p^2(2n+p)})&=j((-1)^p q^{p(2n+p)r+p(n+p)/2+\ell p^2(2n+p)}y^{n+p}/x_0^n;q^{p^2(2n+p)})\\
&=(-1)^{\ell}q^{-p^2(2n+p)\binom{\ell}{2}}\big ( (-1)^pq^{p(2n+p)r+p(n+p)/2}y^{n+p}/x_0^n\big )^{-\ell}\\
&\ \ \ \ \ \ \ \ \ \ \cdot j((-1)^p q^{p(2n+p)r+p(n+p)/2}y^{n+p}/x_0^n;q^{p^2(2n+p)}).
\end{align*}
In the above line, there is a $y^n/x_0^n$ outside of the theta function.  We use (\ref{equation:Fnp-poletypeI}) to substitute for this $y^n/x_0^n$.  Assembling the pieces and collecting terms, we obtain
\begin{align*}
\frac{-x_0\cdot y^{s-\ell p +(n+1)/2}}{ (n+p)x_0^{s-\ell p +(n+1)/2}q^{p(s-\ell p +(n+1)/2)^2}}&\sum_{r^*=0}^{p-1}\Big \{ q^{n\binom{s+r+1-p}{2}}\\
\cdot (-q^{p(s-\ell p +(n+1)/2)}x_0)^{s+r+1-p}& 
j\big (-q^{n\binom{p}{2}+np(r+s+1-p)}(-q^{p(s-\ell p+(n+1)/2)}x_0)^p;q^{np^2}\big )\Big \}.
\end{align*}
The result follows from (\ref{equation:jsplit}).

We prove the residue for poles of type \textup{II}.  Using identity (\ref{equation:1.7}) and Proposition \ref{proposition:H1Thm1.3},
\begin{align}
&\lim_{x\rightarrow x_0}(x-x_0)\theta_{n,p}(x,y,q)
=\sum_{s^*=0}^{p-1}q^{n\binom{r-(n-1)/2}{2}+(n+p)\big (r-(n-1)/2\big )\big (s+(n+1)/2\big )+n\binom{s+(n+1)/2}{2}} \label{equation:thetaresidue-FnpI}\\
& \ \ \
 \cdot \frac{(-x_0)^{r-(n-1)/2}(-y)^{s+(n+1)/2}J_{p^2(2n+p)}^3
 j(q^{p(2n+p)(r+s)+p(n+p)}x_0^py^p;q^{p^2(2n+p)})}{j((-1)^pq^{p(2n+p)s+p(n+p)/2}x_0^{n+p}/y^n;q^{p^2(2n+p)})}\notag\\
 & \ \ \ \cdot j(-q^{np(s-r)}x_0^n/y^n;q^{np^2})\frac{(-1)^{p+1}q^{-p(2n+p)r-p(n+p)/2}x_0^n} {y^{n+p}}\cdot \frac{(-1)^{\ell+1}q^{p^2(2n+p)\binom{-\ell}{2}}x_0}{nJ_{p^2(2n+p)}^3}.\notag
\end{align}
We rewrite two of the theta functions of (\ref{equation:thetaresidue-FnpI}).  Using (\ref{equation:Fnp-poletypeII}) to substitute for $x_0^n/y^n$ and using (\ref{equation:1.8}),
\begin{align*}
j(-q^{np(s-r)}x_0^n/y^n;q^{np^2})&=j(-q^{np(s+r)+rp^2+p(n+p)/2-2\ell np^2-\ell p^3}(-y)^p;q^{np^2})\\
&=(-1)^{2\ell}q^{-np^2\binom{-2\ell}{2}}(-q^{np(s+r)+rp^2+p(n+p)/2-\ell p^3}(-y)^p)^{2\ell}\\
&\ \ \ \ \ \ \ \ \ \  \cdot j(-q^{np(s+r)+rp^2+p(n+p)/2-\ell p^3}(-y)^p;q^{np^2})\\
&=(-1)^{2\ell}q^{-np^2\binom{-2\ell}{2}}(-q^{np(s+r)+rp^2+p(n+p)/2-\ell p^3}(-y)^p)^{2\ell}\\
&\ \ \ \ \ \ \ \ \ \  \cdot j(-q^{n\binom{p}{2}+np(s+r+1-p)}(-q^{rp-\ell p^2+p(n+1)/2}y)^p;q^{np^2}).
\end{align*}
Using (\ref{equation:Fnp-poletypeII}) to substitute for $y^p$ and then using (\ref{equation:1.8}),
\begin{align*}
j&(q^{p(2n+p)(r+s)+p(n+p)}x_0^py^p;q^{p^2(2n+p)})\\
&=j(q^{p(2n+p)(r+s)+p(n+p)}x_0^p(-1)^pq^{-rp(2n+p)-p(n+p)/2+\ell p^2(2n+p)}x_0^{n}/y^n;q^{p^2(2n+p)})\\
&=j((-1)^pq^{p(2n+p)s+p(n+p)/2+\ell p^2(2n+p)}x_0^{n+p}/y^n;q^{p^2(2n+p)})\\
&=(-1)^{\ell}q^{-p^2(2n+p)\binom{\ell}{2}}((-1)^pq^{p(2n+p)s+p(n+p)/2}x_0^{n+p}/y^n)^{-\ell}\\
&\ \ \ \ \cdot j((-1)^pq^{p(2n+p)s+p(n+p)/2}x_0^{n+p}/y^n;q^{p^2(2n+p)}).
\end{align*}
In the above line, there is an $x_0^n/y^n$ outside of the theta function.  We use (\ref{equation:Fnp-poletypeII}) to substitute for this $x_0^n/y^n$.  Assembling the pieces, collecting terms, and using $t=-\ell p+r+(n+1)/2$,  yields for the residue 
\begin{align*}
\frac{x_0^{-\ell p+r+(n+1)/2+1}}{n\cdot y^{-\ell p+r+(n+1)/2}\cdot q^{p(-\ell p+r+(n+1)/2)^2}}&
\sum_{s^*=0}^{p-1} \Big \{q^{n\binom{s+r+1-p}{2}}\cdot (-q^{p(-\ell p+r+(n+1)/2)}y)^{s+r+1-p}\\
 \cdot j(-q^{n\binom{p}{2}+np(s+r+1-p)}&(-q^{p(-\ell p+r+(n+1)/2)}y)^p;q^{np^2})\Big \}.
\end{align*}
The result follows from (\ref{equation:jsplit}).
\end{proof}

\begin{proposition}  \label{proposition:MT-gen} For $x,y\in\mathbb{C}^*$, $g_{n,n+p,n}(x,y,q,-1,-1)+\frac{1}{\overline{J}_{0,np(2n+p)}}\cdot \theta_{n,p}(x,y,q)$ is analytic.
\end{proposition}

\begin{proof}[Proof of Proposition \ref{proposition:MT-gen}]
We want to show that
\begin{equation*}
w(x,y):=g_{n,n+p,n}(x,y,q,-1,-1)+\frac{1}{\overline{J}_{0,np(2n+p)}}\cdot \theta_{n,p}(x,y,q)
\end{equation*}
extends to a function which is analytic for $x,y\in \mathbb{C}^*$.  By the definition of $\theta_{n,p}(x,y,q)$ in Theorem \ref{theorem:masterFnp} and the defining equation (\ref{equation:mdef-2}) of $g_{n,n+p,n}(x,y,q,-1,-1)$, we know that $w(x,y)$ is meromorphic for $x,y\in \mathbb{C}^*$ with the two sources of singularities
\begin{align*}
\textup{I}'.& \ \ q^{n\binom{n+p+1}{2}-n\binom{n+1}{2}-t_1p(2n+p)}(-y)^n(-x)^{-n-p}=q^{knp(2n+p)},\\
\textup{II}'.& \ \ q^{n\binom{n+p+1}{2}-n\binom{n+1}{2}-t_2p(2n+p)}(-x)^n(-y)^{-n-p}=q^{knp(2n+p)},
\end{align*}
where $0\le t_1 \le n-1$, $0\le t_2 \le n-1$ and $k \in \mathbb{Z}$.  The families of singularities \textup{I}$'$ and \textup{II}$'$ are curves in $\mathbb{C}^2$, whose intersections with each other are points.  By considering residues, we will show analyticity of $w(x,y)$ off the points of intersection.  Given the functional equations of Proposition \ref{proposition:Genfg-functional} and Lemma \ref{lemma:Fnp-thetafunctional}, this collection of intersection points reduces to a finite set.  Using Hartog's Theorem, it follows that $w(x,y)$ is analytic at this finite set of intersection points.  Thus $w(x,y)$ is analytic for $x,y\in \mathbb{C}^*.$  

It remains to show that $w(x,y)$ is analytic off the points of intersection, which is equivalent to showing that $w$ has local power series expansions in the variables $x$ and $y$.  The arguments for the variables $x$ and $y$ are the same, so without loss of generality we only demonstrate the local power series expansions in $x$.  Because we are off the points of intersection, if we take an $x_0$ which is a singularity, then $x_0$ satisfies \textup{I}$'$ or \textup{II}$'$ exclusively and is a simple pole. We consider poles of type \textup{I}$'$ and show that the residues sum to zero.  One sees that a pole of type \textup{I} in Proposition \ref{proposition:g-residues} is a pole of type \textup{I} in Lemma \ref{lemma:thetamasterFnp-residue} precisely when $t_1+kn=s-\ell p+(n+1)/2$.  If we take the appropriate residue from Lemma \ref{lemma:thetamasterFnp-residue}, replace $s-\ell p+(n+1)/2$ with $t_1+kn$ and then use (\ref{equation:1.8}), we have 
\begin{align*}
-&x_0\cdot (y/x_0)^{t_1+kn} q^{-p(t_1+kn)^2}\cdot j(q^{p(t_1+kn)}x_0;q^n)/(n+p)\\
&=-x_0\cdot (y/x_0)^{t_1+kn} q^{-p(t_1+kn)^2}\cdot q^{-n\binom{pk}{2}}(-q^{pt_1}x_0)^{-pk}q^{n\binom{t_1}{2}}(-q^{pt_1}x_0)^{t_1}\frac{j(q^{(n+p)t_1}x_0;q^n)}{n+p}\\
&=-x_0\cdot (-y)^{t_1} q^{n\binom{t_1}{2}}q^{np(2n+p)\binom{k}{2}}j(q^{(n+p)t_1}x_0;q^n)/(n+p),
\end{align*}
where the last equality follows from using (\ref{equation:Fnp-poletypeI}) to substitute for $(-y)^n(-x_0)^{-n-p}$ and then simplifying with  $t_1=-kn-\ell p+(n+1)/2$.  We consider poles of type \textup{II}$'$ and show that the residues sum to zero.  Given the functional equations of Proposition \ref{proposition:Genfg-functional} and Lemma \ref{lemma:Fnp-thetafunctional}, it suffices to consider the case $k=0$ for poles of type \textup{II} in Proposition \ref{proposition:g-residues}. One sees that a pole of type \textup{II} in Proposition \ref{proposition:g-residues} is a pole of type \textup{II} in Lemma \ref{lemma:thetamasterFnp-residue} precisely when $t_2=r-\ell p+(n+1)/2$.  We take the appropriate residue from Lemma \ref{lemma:thetamasterFnp-residue}, replace $r-\ell p+(n+1)/2$ with $t_2$, and argue as above.
\end{proof}

\begin{proof}[Proof of Theorem \ref{theorem:masterFnp}] We make the following definition
\begin{equation}
D_{n,p}(x,y,q):=f_{n,n+p,n}(x,y,q)-\Big (g_{n,n+p,n}(x,y,q,-1,-1)+\frac{1}{\overline{J}_{0,np(2n+p)}}\cdot\theta_{n,p}(x,y,q)\Big ).
\end{equation}
From Proposition \ref{proposition:Genfg-functional} and Lemma \ref{lemma:Fnp-thetafunctional}, we have that
\begin{align}
D_{n,p}(q^{p(2n+p)}x,y,q)&=q^{n\binom{n+p+1}{2}-n\binom{n+1}{2}}{(-x)}^n{(-y)^{-(n+p)}}D_{n,p}(x,y,q).\label{equation:Dabc-functional}
\end{align}
We fix $y\ne 0$ and show that $D_{n,p}(x,y,q)=0.$  We have that $f_{n,n+p,n}(x,y,q)$ is analytic for nonzero $x$.  This implies that $D_{n,p}(x,y,q)$ is also, so we can write it as a Laurent series in $x$ valid for all $x\ne 0$
\begin{equation}
D_{n,p}(x,y,q)=\sum_{m}C_mx^m,\label{equation:block}
\end{equation}
where $C_m$ depends on $q,y,n$ and $p$.  Inserting (\ref{equation:block}) into (\ref{equation:Dabc-functional}) yields
\begin{equation}
C_m=(-1)^{n}q^{-p(2n+p)m+n\binom{n+p+1}{2}-n\binom{n+1}{2}}(-y)^{-(n+p)}C_{m-n}.
\end{equation}
We write $m=kn+r$ where $k\in\mathbb{Z}$ and $0\le r \le n-1$.  Induction on $k$ yields
\begin{equation}
C_{kn+r}=(-1)^{kn}q^{-p(2n+p)(kr+n\binom{k+1}{2})+k\big (n\binom{n+p+1}{2}-n\binom{n+1}{2}\big )}(-y)^{-k(n+p)}C_{r}.
\end{equation}
Hence
\begin{equation}
D_{n,p}(x,y,q)=\sum_{r=0}^{n-1}C_rx^r\sum_{k}(-1)^{kn}q^{-p(2n+p)(kr+n\binom{k+1}{2})+k\big (n\binom{n+p+1}{2}-n\binom{n+1}{2}\big )}(-y)^{-k(n+p)}x^{kn}.
\end{equation}
Because $D$ is analytic and because of the $-np(2n+p)k^2/2$ term in the exponent of $q$, we must have that $C_0=C_1=\dots =C_{n-1}=0.$
\end{proof}

\subsection{Proof of Theorem \ref{theorem:main-acdivb}}
We prove technical results analogous to those needed for the proof of Theorem \ref{theorem:masterFnp}.  The following is straightforward.
\begin{lemma} For generic $x,y\in\mathbb{C}^*$
\begin{equation*}
\theta_{a,b,c}(q^{b^2/c-a}x,y,q)=q^{c\binom{b/c+1}{2}-a}{(-x)}{(-y)^{-b/c}}\theta_{a,b,c}(x,y,q).
\end{equation*}
The functional equation with respect to $y$ is obtained by interchanging $x$ with $y$ and then $a$ with $c$.
\end{lemma}

\begin{lemma}  \label{lemma:thetaacdivb-residue}Fix a generic $y\in \mathbb{C}^*$.  The function $\theta_{a,b,c}(x,y,q)$ is meromorphic for $x\ne0$ and has poles at points $x_0$, where $x_0$ satisfies at least one of the following conditions:
\begin{align}
\textup{I}.& \ \ \ q^{(b^2/c-a)(e+1)-c\binom{b/c}{2}}(-x_0)(-y)^{-b/c}=q^{\ell b(b^2/(ac)-1)},\label{equation:acdivb-poletypeI}\\
\textup{II}.&  \ \ \  q^{(b^2/a-c)(d+1)-a\binom{b/a}{2}}(-x_0)^{-b/a}(-y)=q^{\ell b(b^2/(ac)-1)}.\label{equation:acdivb-poletypeII}
\end{align}
If $x_0$ satisfies either \textup{I} or \textup{II} exclusively, then it is a simple pole with respective residue
\begin{align*}
\textup{I}.& \ \ \ -x_0j(y,q^c)\cdot \overline{J}_{0,b^2/a-c},\\
\textup{II}.&\ \ \ {x_0q^{(b^2/a-c)\binom{\ell b/c-d}{2}}}\cdot j(x_0;q^a)\cdot \overline{J}_{0,b^2/c-a}/\big ({b/a}\big ),
\end{align*}
where in \textup{I} we have only computed the residue at $\ell=e=0$.
\end{lemma}

\begin{proof}[Proof of Lemma \ref{lemma:thetaacdivb-residue}] We prove the residue for poles of type \textup{I}.  
With $\ell=e=0$ and Proposition \ref{proposition:H1Thm1.3},
\begin{align}
\lim_{x\rightarrow x_0}&(x-x_0)\theta_{a,b,c}(x,y,q)\label{equation:thetaresidue-acdivbI}\\
&=-x_0\sum_{d=0}^{b/c-1}\sum_{f=0}^{b/a-1}
q^{(b^2/a-c)\binom{d+1}{2}+(b^2/c-a)\binom{f+1}{2}+a\binom{f}{2}}(-x_0)^{f}
 j\big (q^{(b^2/a-c)(d+1)+bf}y;q^{b^2/a}\big ) \notag \\
 &\ \ \ \ \ \cdot j\big (q^{b(b^2/(ac)-1)(f+1)-(b^2/a-c)(d+1)+b^2/c\binom{b/a}{2}}(-x_0)^{b/a}y^{-1};q^{(b^2/a)(b^2/(ac)-1)}\big ) \notag\\
& \ \ \ \ \ \cdot \frac{j\big (q^{ (b^2/c-a)+(b^2/a-c)(d+1)-c\binom{b/c}{2}-a\binom{b/a}{2}}(-x_0)^{1-b/a}(-y)^{1-b/c};q^{b(b^2/(ac)-1)}\big )}
{j\big (q^{(b^2/a-c)(d+1)-a\binom{b/a}{2}}(-x_0)^{-b/a}(-y);q^{b(b^2/(ac)-1)}\big )}.\notag
\end{align}
We rewrite three of the theta functions from (\ref{equation:thetaresidue-acdivbI}).  
Collecting terms and using (\ref{equation:1.8}),
\begin{align*}
j\big (&q^{(b^2/a-c)(d+1)+bf}y;q^{b^2/a}\big )=j\big ((q^c)^{(b^2/ac)d+(b^2/ac-1)+b/c\cdot f-d}y;(q^c)^{b^2/ac}\big )\\
&=(-1)^dq^{-b^2/a\binom{d}{2}} (q^{(b^2/a-c)+b\cdot f-dc}y)^{-d}j\big ((q^c)^{(b^2/ac-1)+b/c\cdot f-d}y;(q^c)^{b^2/ac}\big )
\end{align*}
Using (\ref{equation:acdivb-poletypeI}) with $\ell=e=0$ to substitute for $(-x_0)$ yields
\begin{align*}
j\big (&q^{ (b^2/c-a)+(b^2/a-c)(d+1)-c\binom{b/c}{2}-a\binom{b/a}{2}}(-x_0)^{1-b/a}(-y)^{1-b/c};q^{b(b^2/(ac)-1)}\big )\\
&=j\big (q^{(b^2/a-c)(d+1)-a\binom{b/a}{2}}(-x_0)^{-b/a}(-y);q^{b(b^2/(ac)-1)}\big ),
\end{align*}
and
\begin{align*}
&j\big (q^{b(b^2/(ac)-1)(f+1)-(b^2/a-c)(d+1)+b^2/c\binom{b/a}{2}}(-x_0)^{b/a}y^{-1};q^{(b^2/a)(b^2/(ac)-1)}\big )\\
& =j\big (-q^{b(b^2/(ac)-1)f-(b^2/a-c)(d+1)+b^2/c\binom{b/a}{2}+bc/a\binom{b/c}{2}}(-y)^{b^2/ac-1};q^{(b^2/a)(b^2/ac-1)}\big ).
\end{align*}
We again use (\ref{equation:acdivb-poletypeI}) with $\ell=e=0$ to substitute for $(-x_0)$ in the first line of (\ref{equation:thetaresidue-acdivbI}).  Assembling the pieces and then collecting terms shows that the residue (\ref{equation:thetaresidue-acdivbI}) is equal to
\begin{align*}
-x_0\sum_{d=0}^{b/c-1}\sum_{f=0}^{b/a-1}&(q^c)^{\binom{b/c \cdot f-d}{2}}(-y)^{b/c \cdot f-d} j\big ((q^c)^{(b^2/ac-1)+b/c\cdot f-d}y;(q^c)^{b^2/ac}\big ) \\
& \ \ \ \ \cdot j\big (-(q^c)^{\binom{b^2/ac-1}{2}+(b^2/ac-1)(b/c \cdot f-d)}(-y)^{b^2/ac-1};(q^c)^{(b^2/ac)(b^2/ac-1)}\big )\\
&=-x_0j(y,q^c)\cdot \overline{J}_{0,b^2/a-c},
\end{align*}
where the equality follows from (\ref{equation:Thm1.3AH6}).  

We prove the residue for poles of type \textup{II}.   Using (\ref{equation:1.7}), condition (\ref{equation:acdivb-poletypeII}), and then Proposition \ref{proposition:H1Thm1.3},
{\allowdisplaybreaks
\begin{align}
\lim_{x\rightarrow x_0}&(x-x_0)\theta_{a,b,c}(x,y,q)
=\sum_{e=0}^{b/a-1}\sum_{f=0}^{b/a-1}
q^{(b^2/a-c)\binom{d+1}{2}+(b^2/c-a)\binom{e+f+1}{2}+a\binom{f}{2}}(-x_0)^{f}\label{equation:thetaresidue-acdivbII}\\
&\cdot
  j\big (q^{b(b^2/(ac)-1)(e+f+1)-(b^2/a-c)(d+1)+b^2/c\binom{b/a}{2}}(-x_0)^{b/a}y^{-1};q^{(b^2/a)(b^2/(ac)-1)}\big ) \notag\\
& \cdot j\big (q^{(b^2/a-c)(d+1)+bf}y;q^{b^2/a}\big ) {(-1)\cdot q^{-\ell b(b^2/ac-1)}\cdot (-1)^{\ell+1}q^{b(b^2/ac-1)\binom{-\ell }{2}}x_0}\cdot{(b/a)^{-1}}\notag\\
&\cdot\frac{j\big (q^{ (b^2/c-a)(e+1)+(b^2/a-c)(d+1)-c\binom{b/c}{2}-a\binom{b/a}{2}}(-x_0)^{1-b/a}(-y)^{1-b/c};q^{b(b^2/(ac)-1)}\big )}
{j\big (q^{(b^2/c-a)(e+1)-c\binom{b/c}{2}}(-x_0)(-y)^{-b/c};q^{b(b^2/(ac)-1)}\big )}.\notag
\end{align}}
We rewrite three of the theta functions from (\ref{equation:thetaresidue-acdivbII}).  Using (\ref{equation:acdivb-poletypeII}) to substitute for $(-x_0)^{-b/a}(-y)$,
{\allowdisplaybreaks
\begin{align}
j&\big (q^{ (b^2/c-a)(e+1)+(b^2/a-c)(d+1)-c\binom{b/c}{2}-a\binom{b/a}{2}}(-x_0)^{1-b/a}(-y)^{1-b/c};q^{b(b^2/(ac)-1)}\big )\notag\\
&=j\big (q^{ (b^2/c-a)(e+1)-c\binom{b/c}{2}+\ell b(b^2/ac-1)}(-x_0)(-y)^{-b/c};q^{b(b^2/(ac)-1)}\big )\notag\\
&=(-1)^{\ell}q^{-b(b^2/ac-1)\binom{\ell}{2}}(q^{ (b^2/c-a)(e+1)-c\binom{b/c}{2}}(-x_0)(-y)^{-b/c})^{-\ell}\label{equation:cat}\\
& \ \ \ \ \cdot j\big (q^{ (b^2/c-a)(e+1)-c\binom{b/c}{2}}(-x_0)(-y)^{-b/c};q^{b(b^2/(ac)-1)}\big ),\notag
\end{align}}
where the last equality follows from (\ref{equation:1.8}).  Similarly,
\begin{align*}
& j\big (q^{b(b^2/(ac)-1)(e+f+1)-(b^2/a-c)(d+1)+b^2/c\binom{b/a}{2}}(-x_0)^{b/a}y^{-1};q^{(b^2/a)(b^2/(ac)-1)}\big )\\
&=j\big (-q^{b(b^2/(ac)-1)(e+f+1)+(b^2/c-a)\binom{b/a}{2}-\ell b(b^2/ac-1)};q^{(b^2/a)(b^2/(ac)-1)}\big ).
\end{align*}
Using (\ref{equation:acdivb-poletypeII}) to substitute for $y$ yields
\begin{align*}
j\big (q^{(b^2/a-c)(d+1)+bf}y;q^{b^2/a}\big )=j\big (-q^{bf+\ell b(b^2/ac-1)+a\binom{b/a}{2}}(-x_0)^{b/a};q^{b^2/a}\big ).
\end{align*}
In extreme right of (\ref{equation:cat}), there is a $(-y)$ outside of the theta function.  We replace this $(-y)$ with the value given by (\ref{equation:acdivb-poletypeII}).  Assembling the pieces and collecting terms shows that the residue (\ref{equation:thetaresidue-acdivbII}) is
\begin{align*}
\frac{x_0q^{(b^2/a-c)\binom{\ell b/c-d}{2}}}{b/a} &\cdot \sum_{f=0}^{b/a-1}
(-x_0)^{f+\ell (b^2/ac-1)}q^{a\binom{f+\ell (b^2/ac-1)}{2}}\\
&\cdot  j\big (-(q^a)^{\binom{b/a}{2}+b/a\cdot \big (f+\ell (b^2/ac-1)\big )}(-x_0)^{b/a};(q^a)^{b^2/a^2}\big ) \\
& \cdot \sum_{e=0}^{b/a-1}  q^{(b^2/c-a)\binom{e+f+1-\ell}{2}}j\big (-(q^{(b^2/c-a)})^{\binom{b/a}{2}+b/a\cdot \big (e+f+1-\ell )\big )};(q^{(b^2/c-a)})^{(b^2/a^2)}\big )\\
&={x_0q^{(b^2/a-c)\binom{\ell b/c-d}{2}}}\cdot j(x_0;q^a)\cdot \overline{J}_{0,b^2/c-a}/\big ({b/a}\big ),
\end{align*}
where the last line follows from (\ref{equation:jsplit}).
\end{proof}

Arguing as in the proof of Theorem \ref{theorem:masterFnp}, one sees that 
Theorem \ref{theorem:main-acdivb} follows from the proposition:
\begin{proposition}  \label{proposition:MTacdivb-gen} For $x,y\in\mathbb{C}^*$, $h_{a,b,c}(x,y,q)-\frac{1}{\overline{J}_{0,b^2/a-c}\overline{J}_{0,b^2/c-a}}\cdot\theta_{a,b,c}(x,y,q)$ is analytic.
\end{proposition}

\begin{proof}[Proof of Proposition \ref{proposition:MTacdivb-gen}]One argues as in the proof of Proposition \ref{proposition:MT-gen}.  Here, a pole of type \textup{I} in Proposition \ref{proposition:h-residues} is a pole of type \textup{I} in Lemma \ref{lemma:thetaacdivb-residue} precisely when $k=\ell b/a-e$.  Because of the functional equations of Proposition \ref{proposition:Genfh-functional} and Lemma \ref{lemma:thetaacdivb-residue}, it suffices to consider the case $k=e=\ell=0$.  Here we see that the residues cancel.  We also see that a pole of type \textup{II} in Proposition \ref{proposition:h-residues} is a pole of type \textup{II} in Lemma \ref{lemma:thetaacdivb-residue} precisely when $k=\ell b/c-d$.  We see again that the residues cancel.
\end{proof}


\section{Proofs of the four subtheorems}\label{section:thm-proofs}

The four subtheorems can be obtained from Theorem \ref{theorem:masterFnp} by using Appell-Lerch sum properties such as Theorems \ref{theorem:changing-z-theorem} and \ref{theorem:msplit-general-n} as well as  theta function properties; however, we will prove these four subtheorems directly.  To abbreviate notation, we define
\begin{equation}
M_{n,p}(x,y,q):=g_{n,n+p,n}(x,y,q,y^n/x^n,x^n/y^n).
\end{equation}
Using identity (\ref{equation:mxqz-fnq-z}), it is easy to show that $M_{n,p}(x,y,q)$ satisfies the functional equation of Proposition \ref{proposition:Genfg-functional}.  The following proposition's proof is a straightforward use of (\ref{equation:1.8}).  Here we define $\Theta_{n,1}(x,y,q):=0$.
\begin{lemma} \label{lemma:GenTheta-functional} For generic $x,y\in \mathbb{C}^*$ and $p\in\{1,2,3,4\}$, 
\begin{align*}
\Theta_{n,p}(q^{p(2n+p)}x,y,q)=&q^{n(np+\binom{p+1}{2})}{(-x)^n}{(-y)^{-(n+p)}}\Theta_{n,p}(x,y,q).
\end{align*}
The functional equation with respect to $y$ is obtained by interchanging $x$ with $y$.
\end{lemma}
\begin{proposition}  \label{proposition:MTsub-gen}For $x,y\in \mathbb{C}^*$, the function $M_{n,p}(x,y,q)-\Theta_{n,p}(x,y,q)$ is analytic.
\end{proposition}
With this proposition in hand, one argues as in the proof of Theorem \ref{theorem:masterFnp} to prove the four subtheorems.  It remains to prove Proposition \ref{proposition:MTsub-gen}.  The proof of Proposition \ref{proposition:MTsub-gen} serves as a guide to the remaining subsections, which contain the referenced corollaries and lemmas as well as their proofs.
\begin{proof}[Proof of Proposition \ref{proposition:MTsub-gen}] We want to show that
\begin{equation*}
w(x,y):=M_{n,p}(x,y,q)-\Theta_{n,p}(x,y,q)
\end{equation*}
extends to a function which is analytic for $x,y\in \mathbb{C}^*$.  By the defining equation (\ref{equation:mdef-2}) and the definition of $\Theta_{n,p}(x,y,q)$ in the four subtheorems, we know that $w(x,y)$ is meromorphic for $x,y\in \mathbb{C}^*$ with the three sources of singularities
\begin{align}
\textup{I}'.&\ \  y^n/x^n=q^{knp(2n+p)}\label{equation:polesoftypeI},\\
\textup{II}'.&\ \ x^p=(-1)^{p+1}q^{knp(2n+p)+n(np+\binom{p+1}{2})-rp(2n+p)}\label{equation:polesoftypeII},\\
\textup{III}'.&\ \  y^p=(-1)^{p+1}q^{knp(2n+p)+n(np+\binom{p+1}{2})-rp(2n+p)}\label{equation:polesoftypeIII},
\end{align}
where  $r, k \in \mathbb{Z}$ and $0\le r \le n-1$.  The families of singularities \textup{I}$'$, \textup{II}$'$ and \textup{III}$'$ are lines in $\mathbb{C}^2$, whose intersections with each other are points.  By considering residues, we will show analyticity of $w(x,y)$ off the points of intersection.  Given the functional equations of Proposition \ref{proposition:Genfg-functional} and Lemma \ref{lemma:GenTheta-functional} satisfied by $M_{n,p}(x,y,q)$ and $\Theta_{n,p}(x,y,q)$, this collection of interesection points reduces to a finite set.  From Hartog's theorem, it follows that $w(x,y)$ is analytic at this finite set of intersection points.   Thus $w(x,y)$ is analytic for $x,y \in \mathbb{C}^*$.  To prove that $w(x,y)$ is analytic off the points of intersection is equivalent to showing that $w$ has local power series expansions in the variables $x$ and $y$.  Because we are off the points of intersection, any singularity will satisfy \textup{I}$'$, \textup{II}$'$, or \textup{III}$'$ exclusively thus making it a simple pole.  It suffices to show that the residues at each simple pole sum to zero.  The arguments for the variables $x$ and $y$ are the same, so without loss of generality we only demonstrate the local power series expansions in $x$.  For this, we will consider the residues for poles of type \textup{I}$'$
\begin{align*}
\lim_{x\rightarrow \zeta_n y}(x-\zeta_n y)M_{n,p}(x,y,q)=&R_{n,p}^1(y,q) {\text{ and }}
\lim_{x\rightarrow \zeta_n y}(x-\zeta_n y)\Theta_{n,p}(x,y,q)=T_{n,p}^1(y,q),
\end{align*}
which are Proposition \ref{proposition:Mnp-residue1} and Lemma \ref{lemma:Tnp1-def} respectively, and the residues for poles of type \textup{II}$'$
\begin{align*}
\lim_{x\rightarrow x_0}(x-x_0)M_{n,p}(x,y,q)=&R_{n,p}^2(y,q) {\text{ and }}\lim_{x\rightarrow x_0}(x-x_0)\Theta_{n,p}(x,y,q)=T_{n,p}^2(y,q),
\end{align*}
which are Proposition \ref{proposition:Mnp-residue2} and Lemma \ref{lemma:Tnp2-def} respectively.  Here $x_0$ satisfies (\ref{equation:polesoftypeII}).

We prove the $p=1$ case. By analysis of residues, we show $M_{n,1}(x,y,q)$ extends to a function analytic for all $x\ne 0.$  We consider poles of type \textup{I}$'$.   Because of (\ref{equation:1.10}), the poles of $R_{n,1}^1(y,q)$ (see Definition \ref{definition:defRnp1}) are seen to be removable;  hence, $R_{n,1}^1(y,q)$ extends to a function $f(y)$ analytic for all $y\ne 0$.  By Proposition \ref{proposition:Rnp-functionaleqn}, we have that
$f(q^{2n+1}y)=y^{-2}\zeta_{n}^{-1}q^{n}f(y)$.  By Proposition \ref{proposition:H1Thm1.7}, either $f(y)$ has exactly $2$ zeros in the annulus $|q^{2n+1}|<|y|\le 1$ or $f(y)=0$ for all $y\ne 0.$  But by Proposition \ref{proposition:Rnp-zeros}, there are at least four such zeros.  We consider poles of type \textup{II}$'$.  Because of the functional equation of Proposition \ref{proposition:Genfg-functional} satisfied by $M_{n,1}(x,y,q)$, we can assume $k=0$ in (\ref{equation:polesoftypeII}).  Using  Proposition \ref{proposition:mnp2Lerch-residue}, this pole is removable:
\begin{align*}
\lim_{x\rightarrow q^{n^2+n}}(x-q^{n^2+n})&j(x;q^n)m\Big (q^{n^2+n}\frac{y^n}{x^{n+1}},q^{n(2n+1)},\frac{x^n}{y^n}\Big )=\lim_{x\rightarrow q^{n^2+n}}\frac{j(x;q^n)x^{n+1}y^{-n}}{j(x^n/y^n;q^{n(2n+1)})}=0.
\end{align*}

We prove the $p=2$ (resp. $p=3,4$) case.   Throughout, we let $0\le \{\alpha\} <1$ denote the fractional part of $\alpha,$ and note that $(n,p)=1.$  We consider poles of type \textup{I}$'$.  Given the functional equations of Proposition \ref{proposition:Genfg-functional} and Lemma \ref{lemma:GenTheta-functional} satisfied by $M_{n,p}(x,y,q)$ and $\Theta_{n,p}(x,y,q)$ respectively, we only need to be concerned with $x=\zeta_n y$ in (\ref{equation:polesoftypeI}), where $\zeta_n$ is an $n$-th root of unity.  Poles where $\zeta_n=1$ are removable.  We define 
\begin{equation*}
g(y):=R_{n,p}^1(y,q)-T_{n,p}^1(y,q).
\end{equation*}
We first show $g(y)$ is analytic for all $y\ne 0$.  We see that $g(y)$ is meromorphic for $y\ne 0$ with at most simple poles at the points found in Corollary \ref{corollary:Rn21-residue} (resp. \ref{corollary:Rn31-residue}, \ref{corollary:Rn41-residue}) and Lemma \ref{lemma:Tn21-residue} (resp. \ref{lemma:Tn31-residue}, \ref{lemma:Tn41-residue}).  There is an easily derived correspondance between the poles.  We note that $k,\ell\in \mathbb{Z}$ and define $\ell^*:=\ell-\{(n-1)/2\}$.  If $y_0$ is a simple pole of $R_{n,p}^1(y,q)$, i.e.
\begin{equation*}
y_0^p=(-1)^{p+1}\zeta_n^{-p}q^{knp(2n+p)-n(np+\binom{p}{2})-rp(2n+p)} \ ({\text{or }}y_0^p=(-1)^{p+1}q^{knp(2n+p)-n(np+\binom{p}{2})-rp(2n+p)}),
\end{equation*}
then $y_0$ is a simple pole of $T_{n,p}^1(y,q)$, i.e.
\begin{equation*}
y_0^p=(-1)^{p+1}\zeta_n^{-p}q^{\ell^* p(2n+p)-p(n+p)/2} \ ({\text{or }}y_0^p=(-1)^{p+1}q^{\ell^* p(2n+p)-p(n+p)/2}),
\end{equation*}
precisely when $\ell=kn-r-(n-1)/2+\{(n-1)/2\}.$  Because $(n,p)=1$, the two families which do and do not involve $\zeta_n$ have no overlap.  If we take the residue of such a $y_0$ from Corollary \ref{corollary:Rn21-residue} (resp. \ref{corollary:Rn31-residue}, \ref{corollary:Rn41-residue}), and substitute this value for $\ell$ into the corresponding residue from Lemma \ref{lemma:Tn21-residue} (resp. \ref{lemma:Tn31-residue}, \ref{lemma:Tn41-residue}), we find that the sum of the two residues is zero.  Hence $g(y)$ is analytic for all $y\ne 0$.  From Lemma \ref{lemma:Tnp1-functional} and Proposition \ref{proposition:Rnp-functionaleqn}, we have
\begin{equation*}
g(q^{p(2n+p)}y)=y^{-2p}\zeta_n^{-p}q^{np-(p-1)p(2n+p)}g(y).
\end{equation*}
By Proposition \ref{proposition:H1Thm1.7}, either $g(y)$ has exactly $2p$ zeros in the annulus $|q^{p(2n+p)}|<|y|\le 1$ or $g(y)=0$ for all $y\ne 0.$  However, by Proposition \ref{proposition:Rnp-zeros} and Lemma \ref{lemma:Tnp1-zeros}, there are least $4p$ such zeros, e.g. $g(\pm \zeta_n ^{-1/2} q^{((2k+1)n+pk)/2)})=0$ for all $k\in \mathbb{Z}$, $0\le k\le 2p-1$.  

We consider poles of type \textup{II}$'$.  We define
\begin{equation*}
g(y):=R_{n,p}^2(y,q)-T_{n,p}^2(y,q).
\end{equation*}
Our goal is to show that $g(y)$ is analytic for $y\ne 0$.  Although $R_{n,p}^2(y,q)$ only has a single family of poles (see Proposition \ref{proposition:Rnp2-residue}), all of which are simple, it turns out that $T_{n,p}^2(y,q)$ has two families of poles (see Lemma \ref{lemma:Tn22-residue} (resp. \ref{lemma:Tn32-residue}, \ref{lemma:Tn42-residue})) which may overlap and lead to poles which are not simple.  To get around these intersection points from the two families of singularities of $T_{n,p}^2(y,q)$, we replace the fixed $x_0$ with the variable $x$ and define $f(x,y):=R_{n,p}^2(x,y,q)-T_{n,p}^2(x,y,q)$.  We see $f(x,y)$ is a meromorphic function for $x,y\in \mathbb{C}^*$ with the two sources of singularities 
\begin{align}
\textup{I}.&\ \  y^n=x^nq^{tnp(2n+p)},\label{equation:Rnp2-poleI}\\
\textup{II}.&\ \ y^p=(-1)^{p+1}q^{t^*p(2n+p)-p(n+p)/2},\label{equation:Rnp2-poleII}
\end{align}
where $t \in \mathbb{Z}$ and $t^*:=t-\{(n-1)/2\}$.  The singularities \textup{I} and \textup{II} are lines in $\mathbb{C}^2$, whose intersections are points.  If a singularity satisfies \textup{I} or \textup{II} exclusively, then it is a simple pole.  By considering residues, we will show analyticity of $f(x,y)$ off these intersection points.  Using the functional equations from Proposition \ref{proposition:Rnp2-functionaleqn} and Lemma \ref{lemma:Tnp2-functional} (and easily obtained analogous functional equations for $x$),  this collection of intersection points reduces to a finite set.   From Hartog's theorem, it follows that $f(x,y)$ is analytic at this finite set of points.  We focus on poles of type \textup{I}.  Here the poles are at the points $y_0$ found in Proposition \ref{proposition:Rnp2-residue} and Lemma \ref{lemma:Tn22-residue} (resp. \ref{lemma:Tn32-residue}, \ref{lemma:Tn42-residue}).  We again note that $k,\ell\in \mathbb{Z}$ and define $\ell^*:=\ell-\{(n-1)/2\}$.  We note that $x_0$ is a simple pole of $M_{n,p}(x,y,q)$, i.e.
\begin{equation*}
x_0^p=(-1)^{p+1}q^{knp(2n+p)+n(np+\binom{p+1}{2})-rp(2n+p)},
\end{equation*}
if and only if $x_0$ is a simple pole of $\Theta_{n,p}(x,y,q)$, i.e.
\begin{equation*}
x_0^p=(-1)^{p+1}q^{\ell^* p(2n+p)-p(n+p)/2},
\end{equation*}
where $\ell=kn-r+(n+1)/2+\{(n-1)/2\}.$  If we take the residue of such a $y_0$ from Corollary \ref{corollary:Rnp2-residue} and substitute this value for $\ell$ into the corresponding residue from Lemma \ref{lemma:Tn22-residue} (resp. \ref{lemma:Tn32-residue}, \ref{lemma:Tn42-residue}), we find that the sum of the two residues is zero.  Given the functional equations of Lemma \ref{lemma:Tnp2-functional} and Proposition \ref{proposition:Rnp2-functionaleqn}, we only need to do this for $t=0$ in (\ref{equation:Rnp2-poleI}).  From Lemma \ref{lemma:Tn22-residue} (resp. \ref{lemma:Tn32-residue}, \ref{lemma:Tn42-residue}), poles of type \textup{II} are removable.  Thus $f(x,y)$ is analytic for all $x,y\ne0$; as a consequence, $g(y)=f(x_0,y)$ is analytic for all $y\ne 0.$  From Lemma \ref{lemma:Tnp2-functional} and Proposition \ref{proposition:Rnp2-functionaleqn}, we have
\begin{equation*}
g(q^{p(2n+p)}y)=q^{n(np+\binom{p+1}{2})}{(-y)^n}{(-x_0)^{-(n+p)}}g(y).
\end{equation*}
Arguing as in the proof of Theorem {\ref{theorem:masterFnp}} shows that $g(y)=0$ for all $y\ne 0.$
\end{proof}

\subsection{Functional equations, zeros, and poles for $R_{n,p}^1(y,q)$}

\begin{definition} \label{definition:defRnp1} Let $n$ and $p$ be positive integers with $(n,p)=1$, $y\in \mathbb{C}^*$ generic, and $\zeta_n$ an $n$-th root of unity.  Then
{\allowdisplaybreaks \begin{align*}
R_{n,p}^1(y,q):=\frac{\zeta_n y}{n}\Big \{\sum_{r=0}^{n-1} &\frac{\zeta_n^{-r}j(q^{pr}\zeta_n y;q^n)}{q^{pr^2}j(-q^{n(np+\binom{p}{2})+rp(2n+p)}(-\zeta_ny)^p;q^{np(2n+p)})}\\
&-\sum_{r=0}^{n-1}\frac{\zeta_n^rj(q^{pr}y;q^n)}{q^{pr^2}j(-q^{n(np+\binom{p}{2})+rp(2n+p)}(-y)^p;q^{np(2n+p)})}\Big \}.
\end{align*}}
\end{definition}

\begin{proposition} \label{proposition:Mnp-residue1} Let $y$ and $\zeta_n$ be as in Definition \ref{definition:defRnp1}. Then
\begin{align*}
\lim_{x\rightarrow \zeta_n y} &(x-\zeta_n y)M_{n,p}(x,y,q)=R_{n,p}^1(y,q).
\end{align*}
\end{proposition}

\begin{proof}  This is a straightforward application of Proposition \ref{proposition:H1Thm1.3} followed by identities (\ref{equation:Reciprocal}) and (\ref{equation:1.7}).
\end{proof}

\begin{proposition}  \label{proposition:Rnp-functionaleqn}Let $y$ and $\zeta_n$ be as in Definition \ref{definition:defRnp1}. Then
\begin{align*}
R_{n,p}^1(q^{p(2n+p)}y,q)=y^{-2p}\zeta_n^{-p}q^{np-(p-1)p(2n+p)}R_{n,p}^1(y,q).
\end{align*}
\end{proposition}

\begin{proof}  We have
{\allowdisplaybreaks
\begin{align*}
R_{n,p}^1&(q^{p(2n+p)}y,q)=
\frac{\zeta_nq^{p(2n+p)}y}{n}\Big \{\sum_{r=0}^{n-1} \frac{\zeta_n^{-r}j(q^{n2p+p(r+p)}\zeta_n y;q^n)}{q^{pr^2}j(-q^{n(np+\binom{p}{2})+(r+p)p(2n+p)}(-\zeta_n y)^p;q^{np(2n+p)})} \\
&\ \  -\sum_{r=0}^{n-1}\frac{\zeta_n^{r}j(q^{n2p+p(r+p)}y;q^n)}{q^{pr^2}j(-q^{n(np+\binom{p}{2})+(r+p)p(2n+p)}(-y)^p;q^{np(2n+p)})}\Big \}\\
&=\frac{\zeta_nq^{p(2n+p)}y}{n}\Big \{\sum_{r=p}^{n+p-1} \frac{\zeta_n^{-r+p}j(q^{n2p+pr}\zeta_n y;q^n)}{q^{p(r-p)^2}j(-q^{n(np+\binom{p}{2})+rp(2n+p)}(-\zeta_n y)^p;q^{np(2n+p)})} \\
&\ \ -\sum_{r=p}^{n+p-1}\frac{\zeta_n^{r-p}j(q^{n2p+pr}y;q^n)}{q^{p(r-p)^2}j(-q^{n(np+\binom{p}{2})+rp(2n+p)}(-y)^p;q^{np(2n+p)})}\Big \}\\
&=y^{-2p}\zeta_n^{-p}q^{np-(p-1)p(2n+p)}\frac{\zeta_n y}{n}\Big \{\sum_{r=p}^{n+p-1} \frac{\zeta_n^{-r}j(q^{pr}\zeta_n y;q^n)}{q^{pr^2}j(-q^{n(np+\binom{p}{2})+rp(2n+p)}(-\zeta_n y)^p;q^{np(2n+p)})} \\
&\ \ -\sum_{r=p}^{n+p-1}\frac{\zeta_n^{r}j(q^{pr}y;q^n)}{q^{pr^2}j(-q^{n(np+\binom{p}{2})+rp(2n+p)}(-y)^p;q^{np(2n+p)})}\Big \},
\end{align*}}%
where in the second equality we substituted $r\rightarrow r-p$ in both sums, and in the third equality we applied (\ref{equation:1.8}) and simplified.  With the convention of Proposition \ref{proposition:f-functionaleqn} in mind, we have for generic a function $f$ 
\begin{equation*}
\sum_{r=p}^{n+p-1}f(r)=\sum_{r=p}^{n-1}f(r)+\sum_{r=n}^{n+p-1}f(r).
\end{equation*}
Hence to finish the proof, it suffices to show that for each of the sums within the braces that
\begin{equation*}
\sum_{r=n}^{n+p-1}f(r)=\sum_{r=0}^{p-1}f(r+n)=\sum_{r=0}^{p-1}f(r).
\end{equation*}
Without loss of generality, we consider the first sum within the braces.  After we have shifted $r$,
\begin{align*}
\sum_{r=0}^{p-1} &\frac{\zeta_n^{-r}q^{-p(r+n)^2}j(q^{np+pr}\zeta_n y;q^n)}{j(-q^{np(2n+p)+n(np+\binom{p}{2})+rp(2n+p)}(-\zeta_n y)^p;q^{np(2n+p)})}\\
&=\sum_{r=0}^{p-1} \frac{\zeta_n^{-r}q^{-pr^2}j(q^{pr}\zeta_n y;q^n)}{j(-q^{n(np+\binom{p}{2})+rp(2n+p)}(-\zeta_n y)^p;q^{np(2n+p)})}
\end{align*}
where the equality follows from (\ref{equation:1.8}) and simplifying. 
\end{proof}

\begin{proposition} \label{proposition:Rnp-zeros} Let $k$ be an integer and $\zeta_n$ as in Definition \ref{definition:defRnp1}, then
\begin{equation*}
R_{n,p}^1(\pm \zeta_n^{-1/2}q^{((2k+1)n+kp)/2},q)=0.
\end{equation*}
\end{proposition}
\begin{proof}  Without loss of generality we prove the ``$+$'' case.  We have
\begin{align*}
R_{n,p}^1&(\zeta_n^{-1/2}q^{((2k+1)n+kp)/2},q)\\
&=\frac{\zeta_n \zeta_n^{-1/2}q^{((2k+1)n+kp)/2}}{n}\\
&\ \ \ \ \cdot\Big \{\sum_{r=0}^{n-1} \frac{\zeta_n^{-r}j(q^{pr+((2k+1)n+kp)/2}\zeta_n^{1/2};q^n)}{q^{pr^2}j((-1)^{p+1}q^{n(np+\binom{p}{2})+rp(2n+p)+p((2k+1)n+kp)/2}\zeta_n^{p/2};q^{np(2n+p)})} \\
&\ \ \ \ \ \ \ \ \ -\sum_{r=0}^{n-1}\frac{\zeta_n^{r}j(q^{pr+((2k+1)n+kp)/2}\zeta_n^{-1/2};q^n)}{q^{pr^2}j((-1)^{p+1}q^{n(np+\binom{p}{2})+rp(2n+p)+p((2k+1)n+kp)/2}\zeta_n^{-p/2};q^{np(2n+p)})}\Big \}.
\end{align*}
We show that the two sums are equal.  Using (\ref{equation:1.7}), we rewrite the second sum as
\begin{align*}
\sum_{r=0}^{n-1}\frac{\zeta_n^{r}j(q^{n-pr-((2k+1)n+kp)/2}\zeta_n^{1/2};q^n)}{q^{pr^2}j((-1)^{p+1}q^{n(np+\binom{p}{2})+np-rp(2n+p)-p((2k+1)n+kp)/2}\zeta_n^{p/2};q^{np(2n+p)})}.
\end{align*}
Making the substitution $r\rightarrow n-r-k$, the second sum becomes
\begin{align*}
\sum_{r=1-k}^{n-k}&\frac{\zeta_n^{-r-k}j(q^{-n(p+2k)+pr+((2k+1)n+kp)/2}\zeta_n^{1/2};q^n)}{q^{p(n-r-k)^2}j((-1)^{p+1}q^{-np(2n+p)+n(np+\binom{p}{2})+rp(2n+p)+p((2k+1)n+kp)/2}\zeta_n^{p/2};q^{np(2n+p)})}\\
=&\sum_{r=1-k}^{n-k}\frac{\zeta_n^{-r}j(q^{pr+((2k+1)n+kp)/2}\zeta_n^{1/2};q^n)}{q^{pr^2}j((-1)^{p+1}q^{n(np+\binom{p}{2})+rp(2n+p)+p((2k+1)n+kp)/2}\zeta_n^{p/2};q^{np(2n+p)})},
\end{align*}
where the second equality follows from (\ref{equation:1.8}) and simplifying.  Using the summation convention of Proposition \ref{proposition:f-functionaleqn}, we have for a generic function $f$ that
\begin{equation*}
\sum_{r=1-k}^{n-k}f(r)=\sum_{r=0}^{n-k}f(r)+\sum_{r=1-k}^{-1}f(r)=\sum_{r=0}^{n-k}f(r)+\sum_{r=n-k+1}^{n-1}f(r-n).
\end{equation*}
Hence we need to show
\begin{equation*}
\sum_{r=n-k+1}^{n-1}f(r-n)=\sum_{r=n-k+1}^{n-1}f(r).
\end{equation*}
Focusing on the left-hand side,
\begin{align*}
\sum_{r=n-k+1}^{n-1}&\frac{\zeta_n^{-r+n}j(q^{p(r-n)+((2k+1)n+kp)/2}\zeta_n^{1/2};q^n)}{q^{p(r-n)^2}j((-1)^{p+1}q^{n(np+\binom{p}{2})+(r-n)p(2n+p)+p((2k+1)n+kp)/2}\zeta_n^{p/2};q^{np(2n+p)})}\\
=&\sum_{r=n-k+1}^{n-1}\frac{\zeta_n^{-r}j(q^{-np+pr+((2k+1)n+kp)/2}\zeta_n^{1/2},q^n)}{q^{p(r-n)^2}j((-1)^{p+1}q^{-np(2n+p)+n(np+\binom{p}{2})+rp(2n+p)+p((2k+1)n+kp)/2}\zeta_n^{p/2};q^{np(2n+p)})}\\
=&\sum_{r=n-k+1}^{n-1}\frac{\zeta_n^{-r}j(q^{pr+((2k+1)n+kp)/2}\zeta_n^{1/2};q^n)}{q^{pr^2}j((-1)^{p+1}q^{n(np+\binom{p}{2})+rp(2n+p)+p((2k+1)n+kp)/2}\zeta_n^{p/2};q^{np(2n+p)})},
\end{align*}
where the last equality follows from (\ref{equation:1.8}) and simplifying.  
\end{proof}

\begin{proposition}  \label{proposition:Rnp1-residue}  Let $n$, $p$, and $\zeta_n$ be as in Definition \ref{definition:defRnp1} with $\zeta_n\ne 1$, and let $k$ and $r$ be integers with $0\le r\le n-1$.  $R_{n,p}^1(y,q)$ is meromorphic for $y\ne 0$ and has at most simple poles at points $y_0$, where $y_0$ satisfies one of the following two conditions:
\begin{align*}
\textup{I}.& \ y_0^p=(-1)^{p+1}\zeta_n^{-p}q^{knp(2n+p)-n(np+\binom{p}{2})-rp(2n+p)},\\
\textup{II}.& \ y_0^p=(-1)^{p+1}q^{knp(2n+p)-n(np+\binom{p}{2})-rp(2n+p)}.
\end{align*}
The respective residue at such a $y_0$ is given by
\begin{align*}
\textup{I}.& \ \ \ {(-1)^{k+1}\zeta_n^{1-r}y_0^2q^{np(2n+p)\binom{k}{2}}j(q^{pr}\zeta_n y_0;q^n)}/\big ({npq^{pr^2}J_{np(2n+p)}^3}\big ),\\
\textup{II}.& \ \ \ {(-1)^{k}\zeta_n^{1+r}y_0^2q^{np(2n+p)\binom{k}{2}}j(q^{pr}y_0;q^n)}/\big ({npq^{pr^2}J_{np(2n+p)}^3}\big ).
\end{align*}
\end{proposition}
\begin{proof}  This follows from the definition of $R_{n,p}^1(y,q)$ and Proposition \ref{proposition:H1Thm1.3}.
\end{proof}
\begin{corollary} \label{corollary:Rn21-residue}   Define $v:=kn-r-(n-1)/2$.  For $p=2$, the residues of Proposition \ref{proposition:Rnp1-residue} can be written respectively
\begin{align*}
\textup{I}.& \ \ \ {(-1)^{v-1}\zeta_n^{-1-r}q^{-4(n+1)\binom{v-1}{2}-\tfrac12(n^2-4n-3)}J_{2n,4n}}/\big ( {2nJ_{4n(n+1)}^3}\big ),\\
\textup{II}.& \ \ \ {(-1)^{v}\zeta_n^{1+r}q^{-4(n+1)\binom{v-1}{2}-\tfrac12(n^2-4n-3)}J_{2n,4n}}/\big ({2nJ_{4n(n+1)}^3}\big ).
\end{align*}
\end{corollary}
\begin{proof}  We prove the residue for type \textup{I}; the proof for type \textup{II} is similar.  Using the respective value for $y_0^2$ above and the respective residue from Proposition \ref{proposition:Rnp1-residue}, the residue is
\begin{equation}\label{equation:Rn21I-residue}
{(-1)^k\zeta_{n}^{-1-r}q^{4n(n+1)\binom{k+1}{2}-n(2n+1)-4r(n+1)}j(q^{2r}\zeta_ny_0;q^n)}/\big ({2nq^{2r^2}J_{4n(n+1)}^3}\big ).
\end{equation}
Using (\ref{equation:jsplit}) with $m=2$ yields allows us to expand the $j(q^{2r}\zeta_ny_0,q^n)$ term, to have 
\begin{align}
j(q^{2r}\zeta_ny_0;q^n)&=j(-q^{n+4r}\zeta_n^2y_0^2;q^{4n})-q^{2r}\zeta_ny_0j(-q^{3n+4r}\zeta_n^2y_0^2;q^{4n}).\label{equation:Rn21-split}
\end{align}
We consider the first theta function of the right-hand side of (\ref{equation:Rn21-split}) and substitute for $y_0^2$, to obtain
\begin{align*}
j(-q^{n+4r}\zeta_n^2y_0^2;q^{4n})&=j(q^{4n(k(n+1)-r)-2n^2};q^{4n})
=j(q^{4n(k(n+1)-r-(n+1)/2)}q^{2n};q^{4n})\\
&=(-1)^{(k(n+1)-r-(n+1)/2)}{q^{-4n\binom{k(n+1)-r-(n+1)/2}{2}-2n(k(n+1)-r-(n+1)/2)}}J_{2n,4n},
\end{align*}
where the last equality follows from (\ref{equation:1.8}).  Considering the second theta function of the right-hand side of (\ref{equation:Rn21-split}) and substituting for $y_0^2$, produces
\begin{align*}
j(-q^{3n+4r}\zeta_n^2y_0^2;q^{4n})=j(q^{4n(k(n+1)-r-(n-1)/2))};q^{4n})=0.
\end{align*}
Inserting (\ref{equation:Rn21-split}) into (\ref{equation:Rn21I-residue}) and collecting terms produces the desired result.
\end{proof}

\begin{corollary} \label{corollary:Rn31-residue} For $p=3$, the poles and residues of  Proposition \ref{proposition:Rnp1-residue} can be written respectively 
\begin{align*}
\textup{I}.& \ y_0^3=\zeta_n^{-3}q^{3kn(2n+3)-3n(n+1)-3r(2n+3)} \ (\text{or } y_0=\omega \zeta_n^{-1}q^{kn(2n+3)-n(n+1)-r(2n+3)}),\\
\textup{II}.& \ y_0^3=q^{3kn(2n+3)-3n(n+1)-3r(2n+3)} \ (\text{or } y_0=\omega q^{kn(2n+3)-n(n+1)-r(2n+3)}),
\end{align*}
where $\omega^3=1$ and $\omega\ne 1$.  (The poles where $\omega=1$ are easily seen to be removable.)  Define \\$v:=kn-r-(n-1)/2$.  The respective residue at such a $y_0$ is given by
\begin{align*}
\textup{I}.& \ \ \ {(-1)^{n}\zeta_n^{-1-r}\omega^{1-2v}(1-\omega)q^{-2(2n+3)\binom{v-1}{2}-\tfrac34(n-3)(n+1)}J_{3n}}/\big ({3nJ_{3n(2n+3)}^3}\big ),\\
\textup{II}.& \ \ \ {(-1)^{n+1}\zeta_n^{1+r}\omega^{1-2v}(1-\omega)q^{-2(2n+3)\binom{v-1}{2}-\tfrac34(n-3)(n+1)}J_{3n}}/\big ({3nJ_{3n(2n+3)}^3}\big ).
\end{align*}
\end{corollary}
\begin{proof}  The first part follows from specializing Proposition \ref{proposition:Rnp1-residue} to $p=3$.  We note that when we identify the corresponding pole of $T_{n,3}^1(y,q)$ in Lemma \ref{lemma:Tn31-residue}, the pole must involve the same primitive third root of unity.  We prove the residue for type \textup{I}; the proof for type \textup{II} is similar.  Specializing the residue in Proposition \ref{proposition:Rnp1-residue} to $p=3$ yields
\begin{equation}\label{equation:Rn31I-residue}
{(-1)^{k+1}\zeta_{n}^{1-r}y_0^2q^{3n(2n+3)\binom{k}{2}}j(q^{3r}\zeta_ny_0;q^n)}/\big ({3nq^{3r^2}J_{3n(2n+3)}^3}\big ).
\end{equation}
Using the respective value for $y_0$ from type $I$ above, we have
\begin{equation*}
y_0^2=\omega^2\zeta_n^{-2}q^{2kn(2n+3)-2n(n+1)-2r(2n+3)}.
\end{equation*}
Noting that $j(\omega,q^n)=(1-\omega)J_{3n}$, it follows that
\begin{align*}
j(q^{3r}\zeta_ny_0;q^{n})=&j(\omega q^{n(k(2n+3)-(n+1)-2r)};q^n)
=(-\omega)^{k(2n+3)-(n+1)-2r}\frac{j(\omega;q^n)}{q^{n\binom{k(2n+3)-(n+1)-2r}{2}}}.
\end{align*}
We insert both results back into (\ref{equation:Rn31I-residue}) and collect terms to produce the desired result.
\end{proof}

\begin{corollary} \label{corollary:Rn41-residue}  For $p=4$, the poles and residues of  Proposition \ref{proposition:Rnp1-residue} can be written respectively
\begin{align*}
\textup{I}. & \ y_0^4=-\zeta_n^{-4}q^{4kn(2n+4)-n(4n+6)-4r(2n+4)} \ (\text{or } y_0^2=i\zeta_n^{-2}q^{2kn(2n+4)-n(2n+3)-2r(2n+4)}),\\
\textup{II}. & \ y_0^4=-q^{4kn(2n+4)-n(4n+6)-4r(2n+4)} \ (\textup{or } y_0^2=iq^{2kn(2n+4)-n(2n+3)-2r(2n+4)}),
\end{align*}
where $i^2=-1$.  Define $v:=kn-r-(n-1)/2$.  The respective residue at such a $y_0$ is given by
\begin{align*}
\textup{I}.& \ \ \ \frac{i^{(2\cdot \delta)}(-1)^{(v-1)/2+\delta}}{4nJ_{4n(2n+4)}^3\zeta_n^{1+r}}\cdot q^{-2(2n+4)\binom{v-1}{2}-n^2+n+3} \cdot \Big ( (1-i)J_{4n,16n}+i\zeta_n y_0q^{2-(2n+4)v}J_{8n,16n}\Big ),\\
\textup{II}.&\ \ \ \frac{i^{(2\cdot \delta)}(-1)^{(v+1)/2+\delta}\zeta_n^{1+r}}{4nJ_{4n(2n+4)}^3}\cdot q^{-2(2n+4)\binom{v-1}{2}-n^2+n+3}  \cdot \Big ( (1-i)J_{4n,16n}+i y_0q^{2-(2n+4)v}J_{8n,16n}\Big ),
\end{align*}
where $\delta:=\{(v-1)/2\},$ with $0\le \{x\}<1$ denoting the fraction part of $x$.
\end{corollary}
\begin{proof}The first part follows from specializing Proposition \ref{proposition:Rnp1-residue} to $p=4$.  When we identify the corresponding pole of $T_{n,4}^1(y,q)$ in Lemma \ref{lemma:Tn41-residue}, the pole must involve the same primitive fourth root of unity.  We prove the residue for type \textup{I}; the proof for type \textup{II} is similar.  Specializing the residue in Proposition \ref{proposition:Rnp1-residue} to $p=4$ yields
\begin{equation}
{(-1)^{k+1}i\zeta_n^{-1-r}q^{2kn(2n+4)-n(2n+3)-2r(2n+4)+4n(2n+4)\binom{k}{2}}j(q^{4r}\zeta_ny_0;q^n)}/\big ({4nq^{4r^2}J_{4n(2n+4)}^3}\big ).\label{equation:Rn41-residue}
\end{equation}
Using (\ref{equation:jsplit}) with $m=4$ we can rewrite the $j(q^{4r}\zeta_ny_0,q^n)$ term as  
\begin{align}
j(q^{4r}\zeta_ny_0;q^n)&=j(-q^{6n+16r}\zeta_n^4y_0^4,q^{16n})-q^{4r}\zeta_ny_0j(-q^{10n+16r}\zeta_n^4y_0^4;q^{16n})\label{equation:cat2}\\
&\ \ +q^{n+8r}\zeta_n^2y_0^2j(-q^{14n+16r}\zeta_n^4y_0^4;q^{16n})-q^{3n+12r}\zeta_n^3y_0^3j(-q^{18n+16r}\zeta_n^4y_0^4;q^{16n}).\notag
\end{align}
We have two cases to consider depending on whether $kn-r-(n-1)/2$ is even or odd.  We consider the even case; the odd case is similar.   We work with the sum of the first and third summands of the right-hand side of (\ref{equation:cat2}).  Substituting for $y_0^4$ and $y_0^2,$
\begin{align*}
j(&-q^{6n+16r}\zeta_n^4y_0^4;q^{16n})+q^{n+8r}\zeta_n^2y_0^2j(-q^{14n+16r}\zeta_n^4y_0^4;q^{16n})\\
&=j(q^{16nk+8n(kn-r-(n-1)/2-1)}q^{4n};q^{16n})\\
&\ \ \ \ \ +iq^{8nk+4n(kn-r-(n+1)/2)}j(q^{16nk+8n(kn-r-(n-1)/2)}q^{4n};q^{16n}).
\end{align*}
Under the even assumption, we use (\ref{equation:1.8}) and simplify to rewrite this sum as
\begin{align}
(-1)^{k+1+(kn-r-(n-1)/2)/2)}q^{-16n\binom{k+(kn-r-(n-1)/2)/2}{2}+4n(k-1+(kn-r-(n-1)/2)/2)}(1-i)J_{4n,16n}.\label{equation:13-even}
\end{align}
We consider the sum of the second and fourth summands.  Substituting for $y_0^4$ and $y_0^2,$
\begin{align*}
-q^{4r}\zeta_ny_0&j(-q^{10n+16r}\zeta_n^4y_0^4;q^{16n})-q^{3n+12r}\zeta_n^3y_0^3j(-q^{18n+16r}\zeta_n^4y_0^4;q^{16n})\\
=&-q^{4r}\zeta_ny_0\Big [ j(q^{16nk+8n(kn-r-(n-1)/2)};q^{16n})\\
&+iq^{8nk+4n(kn-r-(n-1)/2)-2n}j(q^{16nk+8n(kn-r-(n-1)/2)}q^{8n};q^{16n})\Big ].
\end{align*}
Again under the even assumption, we use (\ref{equation:1.8}) and simplify to rewrite this sum as
\begin{align}
(-1)^{k+1+(kn-r-(n-1)/2)/2}i\zeta_ny_0q^{-16n\binom{k+(kn-r-(n-1)/2)/2}{2}-2n+4r}J_{8n,16n}.\label{equation:24-even}
\end{align}
Inserting  (\ref{equation:13-even}) and (\ref{equation:24-even}) into (\ref{equation:Rn41-residue}), we obtain
\begin{align*}
(-1&)^{(kn-r-(n-1)/2)/2}i\zeta_n^{-1-r}q^{2kn(2n+4)-n(2n+3)-2r(2n+4)+4n(2n+4)\binom{k}{2}-4r^2}\\
\cdot &\frac{q^{-16n\binom{k+(kn-r-(n-1)/2)/2}{2}+4n(k-1+(kn-r-(n-1)/2)/2)}}{4nJ_{4n(2n+4)}^3}\\
\cdot & \Big \{ (1-i)J_{4n,16n}+i\zeta_n y_0 q^{-4nk+4r+2n-2n(kn-r-(n-1)/2)}J_{8n,16n}\Big \}.
\end{align*}
Simplifying and collecting terms produces the desired result.
\end{proof}


\subsection{Functional equations and poles for $R_{n,p}^2(y,q)$}

\begin{definition} \label{definition:Rnp2-def} Let $n$ and $p$ be positive integers with $(n,p)=1$, $y\in \mathbb{C}^*$ be generic, and $x_0$ be such that $x_0^p=(-1)^{p+1}q^{knp(2n+p)+n(np+\binom{p+1}{2})-rp(2n+p)}$ for some $k,r \in \mathbb{Z}$  with $0\le r\le n-1$. Then
\begin{equation}\label{equation:Rnp2-def}
R_{n,p}^2(y,q):=(-1)^{k+1}q^{np(2n+p)\binom{k+1}{2}-pr^2}\cdot\frac{x_{0}^{kn+n+1-r}}{y^{kn+n-r}}\cdot\frac{j(q^{pr}x_0;q^n)}{p j(x_0^n/y^n;q^{np(2n+p)})}.
\end{equation}
\end{definition}

\begin{proposition} \label{proposition:Mnp-residue2} Let $n$, $p$, $y$, and $x_0$ be as in Definition \ref{definition:Rnp2-def}. Then
\begin{align*}
\lim_{x\rightarrow x_0} &(x-x_0)M_{n,p}(x,y,q)=R_{n,p}^2(y,q).
\end{align*}
\end{proposition}
\begin{proof}
This follows from the definition of $M_{n,p}(x,y,q)$ and Proposition \ref{proposition:mnp2Lerch-residue}.
\end{proof}

\begin{proposition}  \label{proposition:Rnp2-functionaleqn} Let $n,$ $p,$ $y,$ and $x_0$ be as in Definition \ref{definition:Rnp2-def}. Then
\begin{align*}
R_{n,p}^2(q^{p(2n+p)}y,q)=q^{n(np+\binom{p+1}{2})}{(-y)^n}{(-x_0)^{-(n+p)}}R_{n,p}^2(y,q).
\end{align*}
\end{proposition}
\begin{proof}  This follows (\ref{equation:1.8}) and using the value of $x_0^p$ from Definition \ref{definition:Rnp2-def}.
\end{proof}

\begin{proposition}  \label{proposition:Rnp2-residue}  Let $\zeta_n$ be an $n$-th root of unity and let $y$, $x_0$, $k$, and $r$ be as in Definition \ref{definition:Rnp2-def}.  $R_{n,p}^2(y,q)$ is meromorphic for $y\ne 0$ and has at most simple poles at points $y_0$, where $y_0$ satisfies the following condition:
\begin{align*}
\textup{I}. \ y_0=\zeta_n x_0q^{tp(2n+p)}.
\end{align*}
The respective residue at such a $y_0$  at $t=0$ is given by
\begin{align*}
\textup{I}.& \ \ \ {(-1)^{k+1}\zeta_n^{1+r}x_0^2q^{np(2n+p)\binom{k+1}{2}}j(q^{pr}x_0;q^n)}/\big ({npq^{pr^2}J_{np(2n+p)}^3}\big ).
\end{align*}
Given the residue at $t=0$, one can use the functional equation of Proposition \ref{proposition:Rnp2-functionaleqn} to compute the residue for general $t\in \mathbb{Z}.$
\end{proposition}
\begin{proof}  This follows from the Definition \ref{definition:Rnp2-def} and Proposition \ref{proposition:H1Thm1.3}.
\end{proof}

\begin{corollary} \label{corollary:Rnp2-residue} Define $v:=kn-r+(n+1)/2$.  For $p=2,$ $3,$ and $4$ the values of the residues in Proposition \ref{proposition:Rnp2-residue} at $t=0$ can be written,
\begin{align*}
2)& \ \ \ (-1)^{v}\zeta_n^{1+r}q^{-4(n+1)\binom{v-1}{2}-\tfrac12(n^2-4n-3)}
{J_{2n,4n}}/\big ({2nJ_{4n(n+1)}^3}\big ),\\
3)& \ \ \ {(-1)^{n+1}\zeta_n^{1+r}\omega^{-2v+1}(1-\omega)q^{-2(2n+3)\binom{v-1}{2}-\tfrac34(n^2-2n-3)}J_{3n}}/\big ({3nJ_{3n(2n+3)}^3}\big ),\\
4)& \ \ \ \frac{i^{(2\cdot \delta)}(-1)^{(v+1)/2+\delta}\zeta_n^{1+r}}{4nJ_{4n(2n+4)}^3} \cdot {q^{-2(2n+4)\binom{v-1}{2}-n^2+n+3}}
\cdot \Big ( (1-i)J_{4n,16n}+ix_0q^{2-(2n+4)v}J_{8n,16n}\Big ),
\end{align*}
where $\omega^2+\omega+1=0$, $i^2=-1$, and $\delta:=\{(v-1)/2\}$ with $0\le \{\alpha \}<1$ denoting the fractional part of $\alpha$.
\end{corollary}

\begin{proof}  We prove the $p=2$ case.  Using Proposition \ref{proposition:Rnp2-residue} and the value of $x_0^2$ from Definition \ref{definition:Rnp2-def} yields
\begin{align}
(-1)^{k}\zeta_n^{1+r}q^{4n(n+1)\binom{k+1}{2}+4kn(n+1)+n(2n+3)-4r(n+1)-2r^2}{j(q^{2r}x_0;q^n)}/\big ({2nJ_{4n(n+1)}^3}\big ).\label{equation:Rn22I-residue}
\end{align}
We rewrite $j(q^{2r}x_0,q^n)$.  Using (\ref{equation:jsplit}) with $m=2$ yields
\begin{equation}
j(q^{2r}x_0,q^n)=j(-q^{n+4r}x_0^2;q^{4n})-q^{2r}x_0j(-q^{3n+4r}x_0^2;q^{4n}).\label{equation:dog}
\end{equation}
We take the first theta function of the right-hand side of (\ref{equation:dog}) and substitute for $x_0^2$,
\begin{align*}
j(&-q^{n+4r}x_0^2;q^{4n})=j(q^{4n(k(n+1)-r)+2n^2+4n};q^{4n})
=j(q^{4n(k(n+1)-r+(n+1)/2)}q^{2n};q^{4n})\\
&=(-1)^{(k(n+1)-r+(n+1)/2)}{q^{-4n\binom{k(n+1)-r+(n+1)/2}{2}-2n(k(n+1)-r+(n+1)/2)}}J_{2n,4n},
\end{align*}
by (\ref{equation:1.8}).  Taking the second theta function of the right-hand side of (\ref{equation:dog}) and substituting for $x_0^2$ produces
\begin{equation*}
j(-q^{3n+4r}x_0^2;q^{4n})=j(q^{4n(k(n+1)-r+(n+3)/2)};q^{4n})=0.
\end{equation*}
Inserting these into (\ref{equation:Rn22I-residue}) produces the desired result.

 We prove $p=3$ case.  From Definition \ref{definition:Rnp2-def}, we have
\begin{equation}
x_0^3=q^{3kn(2n+3)+3n(n+2)-3r(2n+3)} \ ({\text{or }} x_0=\omega q^{kn(2n+3)+n(n+2)-r(2n+3)}),
\end{equation}
where $\omega^2+\omega+1=0.$  (The pole with $\omega=1$ is easily seen to be removable in $M_{n,3}(x,y,q)$ and $\Theta_{n,3}(x,y,q)$).  When we identify the corresponding pole of $\Theta_{n,3}(x,y,q)$, the pole must involve the same primitive third root of unity.  Specializing Proposition \ref{proposition:Rnp2-residue} to $p=3$ and substituting for $x_0$ yields
\begin{align}
(-1)^{k+1}\zeta_n^{1+r}\omega^{2}
q^{3n(2n+3)\binom{k+1}{2}+2kn(2n+3)+2n(n+2)-2r(2n+3)-3r^2}{j(q^{3r}x_0;q^n)}/\big({3nJ_{3n(2n+3)}^3}\big ).\label{equation:Rn32I-residue}
\end{align}
We rewrite the theta function in the numerator.  Substituting for $x_0$,
\begin{align*}
j(q^{3r}x_0;q^{n})&=j(\omega q^{n(k(2n+3)+n+2-2r)};q^n)\\
&=(-1)^{k(2n+3)+n+2-2r}q^{-n\binom{k(2n+3)+n+2-2r}{2}}\omega^{-(k(2n+3)+n+2-2r)}j(\omega;q^n).&(\text{by }(\ref{equation:1.8}))
\end{align*}
We note that $j(\omega,q^n)=(1-\omega)J_{3n}$, insert everything back into (\ref{equation:Rn32I-residue}) and simplify.

We prove the $p=4$ case.  From Definition \ref{definition:Rnp2-def}, we have
\begin{equation}
x_0^4=-q^{4kn(2n+4)+n(4n+10)-4r(2n+4)}\ ({\text{or }} x_0^2=iq^{2kn(2n+4)+n(2n+5)-2r(2n+4)}),
\end{equation}
where $i^2=-1.$  When we identify the corresponding pole of $\Theta_{n,4}(x,y,q)$, the pole must involve the same primitive fourth root of unity.  Specializing Proposition \ref{proposition:Rnp2-residue} to $p=4$ and substituting for $x_0$ yields
\begin{align}
(-1)^{k+1}i \zeta_n^{1+r}
q^{4n(2n+4)\binom{k+1}{2}+2kn(2n+4)+n(2n+5)-2r(2n+4)-4r^2}{j(q^{4r}x_0;q^n)}/\big ({4nJ_{4n(2n+4)}^3}\big ).\label{equation:Rn42I-residue}
\end{align}
To rewrite $j(q^{4r}x_0,q^n)$, we argue as in the proof of Corollary \ref{corollary:Rn41-residue}.  The result follows.
\end{proof}


\subsection{Functional equations, zeros, and poles for $T_{n,p}^1(y,q)$}\label{section:Tnp1}

\begin{definition} \label{definition:defTn21} Let $n\in \mathbb{N}$ with $(n,2)=1$, $y\in\mathbb{C}^*$ generic, and $\zeta_n\ne 1$ an $n$-th root of unity. Then
\begin{equation}
T_{n,2}^1(y,q):=\frac{y^{3}J_{2n,4n}J_{4(n+1),8(n+1)}j(\zeta_n^{-1},q^{n+2}\zeta_n y^2;q^{4(n+1)})j(q^{2n}/\zeta_n^2y^4;q^{8(n+1)})}
{n\zeta_n^{(n-5)/2}q^{(n^2-3)/2}J_{4n(n+1)}^3j(-q^{n+2}\zeta_n^2 y^2,-q^{n+2}y^2;q^{4(n+1)})}.\label{equation:Tn21-def}
\end{equation}
\end{definition}

\begin{definition}\label{definition:defTn31} Let $n\in \mathbb{N}$ with $(n,3)=1$, $y\in\mathbb{C}^*$ generic, and $\zeta_n\ne 1$ an $n$-th root of unity. Also, define $\delta:=\{(n-1)/2\}$ and $\beta:=\delta\cdot (2n+3)$, with $0\le \{\alpha \}<1$ denoting the fractional part of $\alpha.$  Then
\begin{align}
T_{n,3}^1&(y,q):=q^{n\binom{\delta-(n-3)/2}{2}+(n+3)\big (\delta-(n-3)/2\big )\big (\delta+(n+1)/2\big )+n\binom{\delta+(n+1)/2}{2}-3n}(-\zeta_n y)^{\delta-(n-3)/2}\notag\\
&\ \ \cdot(-y)^{\delta+(n+1)/2}\frac{\zeta_nyJ_{3n}J_{3(2n+3)}j(\zeta_n^{-1};q^{3(2n+3)})j(q^{(n+3)/2+\beta}\zeta_ny,q^{(n+3)/2+\beta}y;q^{2n+3})}
 {nJ_{2n+3}^2J_{3n(2n+3)}^3j(q^{3(n+3)/2+3\beta}\zeta_n^3y^3,q^{3(n+3)/2+3\beta}y^3;q^{3(2n+3)})}\notag\\
&\ \ \cdot  \Big \{ j(q^{7(n+1)/2+4+3\beta}\zeta_n^2y^3,q^{7(n+1)/2+4+3\beta}\zeta_ny^3;q^{3(2n+3)})\notag\\
&\ \ \ \ -{q^{n+3+2\beta}}{\zeta_ny^2}j(q^{11(n+1)/2+5+3\beta}\zeta_n^2y^3,q^{11(n+1)/2+5+3\beta}\zeta_ny^3;q^{3(2n+3)})\Big \}\label{equation:Tn31-def}
.
\end{align}
\end{definition}

\begin{definition}  \label{definition:defTn41}Let $n\in \mathbb{N}$ with $(n,2)=1$, $y\in\mathbb{C}^*$ be generic, and $\zeta_n\ne 1$ an $n$-th root of unity.  Define
{\allowdisplaybreaks \begin{align}
T_{n,4}^1(y,q):&=\frac{\zeta_n^{-(n-5)/2}y^3q^{-(n^2+n-3)}j(\zeta_n^{-1};q^{4(2n+4)})}{nJ_{4n(2n+4)}^3j(-q^{2n+8}\zeta_n^4y^4;q^{4(2n+4)})j(-q^{2n+8}y^4;q^{4(2n+4)})} \notag\\
\cdot \Big [&J_{4n,16n}\cdot \frac{j(q^{6n+16}\zeta_n^2y^4;q^{4(2n+4)})j(q^{n+4}\zeta_n y^2;q^{2(2n+4)})j(-q^{2(2n+4)}\zeta_n^{-1};q^{4(2n+4)})}{J_{2(2n+4)}^3J_{8(2n+4)}} \notag \\
&\ \ \cdot  \Big \{j(-q^{2n+8}\zeta_n^2y^4;q^{4(2n+4)})j(q^{2(2n+4)}\zeta_n^{-2};q^{4(2n+4)})J_{4(2n+4)}^2\notag\\
&\ \ \ \ \ \ +\frac{q^{n+4}\zeta_n^2y^2 j(-q^{6n+16}\zeta_n^2y^4;q^{4(2n+4)})j(q^{4n+8}\zeta_n^{-1},-\zeta_n^{-1};q^{4(2n+4)})^2}{J_{4(2n+4)}}\Big \}\notag\\
&-qJ_{8n,16n}\cdot \frac{j(q^{2n+8}\zeta_n^2y^4;q^{4(2n+4)})j(q^{3n+8}\zeta_ny^2;q^{2(2n+4)})j(-\zeta_n^{-1};q^{4(2n+4)})}{J_{2(2n+4)}^2} \notag\\
&\ \ \cdot  \Big \{\frac{q^{n+1}j(-q^{2n+8}\zeta_n^2y^4;q^{4(2n+4)})j(q^{4n+8}\zeta_n^{-2};q^{4(2n+4)})J_{8(2n+4)}}{yJ_{4(2n+4)} }\notag\\ 
&\ \ \ \ \ \ +\frac{q\zeta_n yj(-q^{6n+16}\zeta_n^2y^4;q^{4(2n+4)})j(q^{4(2n+4)}\zeta_n^{-2};q^{8(2n+4)})^2}{J_{8(2n+4)} }\Big \} \Big ].\label{equation:Tn41-def}
\end{align}}%
\end{definition}

\begin{lemma}  \label{lemma:Tnp1-def}For $p=2$ (resp. $3,4$), let $n$, $y$, and $\zeta_n$ be as in Definition \ref{definition:defTn21} (resp. \ref{definition:defTn31}, \ref{definition:defTn41}). Then
\begin{equation*}
\lim_{x\rightarrow y\zeta_n}(x-y\zeta_n)\Theta_{n,p}(x,y,q)=T_{n,p}^1(y,q).
\end{equation*}
\end{lemma}

\begin{proof}   This is just an application of Proposition \ref{proposition:H1Thm1.3}.
\end{proof}
\begin{remark}  For reference we include the simplified versions of $\Theta_{n,3}(x,q,z)$ for $n$ even
\begin{align*}
\Theta_{n,3}&(x,y,q):=\frac{y^{n/2-1}J_{3n}J_{3(2n+3)}j(y/x;q^{3(2n+3)})j(q^{3n/2+3}x,q^{3n/2+3}y;q^{2n+3})}
{q^{n(3n+2)/4}x^{n/2}J_{2n+3}^2j(y^{n}/x^{n};q^{3n(2n+3)})j(q^{9n/2+9}x^3,q^{9n/2+9}y^3;q^{3(2n+3)})}\\
\cdot\Big \{ &j(q^{5n/2+6}x^2y,q^{5n/2+6}xy^2;q^{3(2n+3)})-\frac{q^{n}}{xy}j(q^{n/2+3}x^2y,q^{n/2+3}xy^2;q^{3(2n+3)})\Big \},
\end{align*}
and for $n$ odd
\begin{align*}
\Theta_{n,3}(x,y,q)&:=\frac{q^{-(n(3n+2)-9)/4}J_{3n}J_{3(2n+3)}j(y/x;q^{3(2n+3)})j(q^{(n+3)/2}x,q^{(n+3)/2}y;q^{2n+3})}
{J_{2n+3}^2j(y^{n}/x^{n};q^{3n(2n+3)})j(q^{3(n+3)/2}x^3,q^{3(n+3)/2}y^3;q^{3(2n+3)})}\\
\cdot & y^{(n+1)/2}x^{-(n-3)/2}\Big \{ j(q^{7(n+1)/2+4}x^2y,q^{7(n+1)/2+4}xy^2;q^{3(2n+3)})\\
&-{q^{n+3}}{xy}j(q^{11(n+1)/2+5}x^2y,q^{11(n+1)/2+5}xy^2;q^{3(2n+3)})\Big \}.
\end{align*}
\end{remark}
\begin{lemma}  \label{lemma:Tnp1-functional} For $p=2$ (resp. $3,4$), let $n$, $y$, and $\zeta_n$ be as in Definition \ref{definition:defTn21} (resp. \ref{definition:defTn31}, \ref{definition:defTn41}). Then
\begin{equation*}
T_{n,p}^1(q^{p(2n+p)}y,q)=y^{-2p}\zeta_{n}^{-p}q^{np-(p-1)p(2n+p)}T_{n,p}^1(y,q).
\end{equation*}
\end{lemma}
\begin{proof}  This is just repeated application of (\ref{equation:1.8}).
\end{proof}

\begin{lemma} \label{lemma:Tnp1-zeros} For $p=2$ (resp. $3,4$), let $n$, $y$, and $\zeta_n$ be as in Definition \ref{definition:defTn21} (resp. \ref{definition:defTn31}, \ref{definition:defTn41}).  Then
\begin{equation*}
T_{n,p}^1(\pm \zeta_n^{-1/2}q^{((2k+1)n+kp)/2},q)=0.
\end{equation*}
\end{lemma}
\begin{proof}  We first note that because $(n,p)=1$, the denominator of $T_{n,p}^1(y,q)$ does not vanish.  We prove the $p=2$ case.   It suffices to consider two terms in the numerator of $T_{n,2}^1(y,q)$.  For $k$ odd, we have
\begin{align*}
j(q^{n+2}\zeta_ny^2;q^{4(n+1)})=j(q^{n+2+(2k+1)n+k2};q^{4(n+1)})=j(q^{2(k+1)(n+1)};q^{4(n+1)})=0,
\end{align*}
and for $k$ even, we have
\begin{align*}
j(q^{2n}/\zeta_n^2y^4;q^{8(n+1)})=j(q^{2n-2((2k+1)n+2k)};q^{8(n+1)})=j(q^{-4k(n+1)};q^{8(n+1)})=0.
\end{align*}

 We prove the $p=3$ case with $n$ even because the equations are more compact.  The proof for the $n$ odd case is similar.  Without loss of generality we prove the ``$+$'' sign case.   We focus on the two products inside the braces of (\ref{equation:Tn31-def}) and show that their sum is zero.  We unwind all four theta functions with (\ref{equation:1.8}) and pull out the common factor $-q^{-2n-6}\zeta_n^{-2}y^{-4}$ to obtain
\begin{align*}
j(q^{5n/2+6}& \zeta_n^2y^3,q^{5n/2+6}\zeta_n y^3;q^{3(2n+3)})
-\frac{q^{n}}{\zeta_n y^2}j(q^{n/2+3}\zeta_n^2y^3,q^{n/2+3}\zeta_n y^3;q^{3(2n+3)})\\
&=j(q^{5n/2+6+3[(2k+1)n+3k]/2}\zeta_n^{1/2},q^{5n/2+6+3[(2k+1)n+3k]/2}\zeta_n^{-1/2};q^{3(2n+3)})\\
&\ \ -q^{-2kn-3k}j(q^{n/2+3+3[(2k+1)n+3k]/2}\zeta_n^{1/2},q^{n/2+3+3[(2k+1)n+3k]/2}\zeta_n^{-1/2} ;q^{3(2n+3)}),
\end{align*}
where the equality follows from substituting for $y$.  We rewrite each term of the second product to make the second product look like the first product.  Working with the first term of the second product,
\begin{align*}
j&(q^{n/2+3+3[(2k+1)n+3k]/2}\zeta_n^{1/2};q^{3(2n+3)})
=j(q^{11n/2+6-3[(2k+1)n+3k]/2}\zeta_n^{-1/2};q^{3(2n+3)}) \\
&=j(q^{-3k(2n+3)+5n/2+6+3[(2k+1)n+3k]/2}\zeta_n^{-1/2};q^{3(2n+3)})\\
&=(-1)^kq^{-3(2n+3)\binom{k+1}{2}}q^{k(5n/2+6+3[(2k+1)n+3k]/2)}\zeta_n^{-k/2}j(q^{5n/2+6+3[(2k+1)n+3k]/2}\zeta_n^{-1/2};q^{3(2n+3)}),
\end{align*}
where the first equality follows from (\ref{equation:1.7}), the second equality is just a rearrangement, and the third equality follows from (\ref{equation:1.8}) and simplifying.  Examining the second term of the second product and arguing as above,
\begin{align*}
j(&q^{n/2+3+3[(2k+1)n+3k]/2}\zeta_n^{-1/2};q^{3(2n+3)})\\
&=(-1)^kq^{-3(2n+3)\binom{k+1}{2}}q^{k(5n/2+6+3[(2k+1)n+3k]/2)}\zeta_n^{k/2}j(q^{5n/2+6+3[(2k+1)n+3k]/2}\zeta_n^{1/2};q^{3(2n+3)}),
\end{align*}
and the result follows.

We prove the $p=4$ case.   It suffices to show that for $y=\pm \zeta_n^{-1/2}q^{((2k+1)n+k4)/2}$, that
\begin{align*}
j(q^{6n+16}\zeta_n^2y^4;q^{4(2n+4)})&j(q^{n+4}\zeta_n y^2;q^{2(2n+4)})\\
&=j(q^{2n+8}\zeta_n^2y^4;q^{4(2n+4)})j(q^{3n+8}\zeta_ny^2;q^{2(2n+4)})=0.
\end{align*}
We show the first product of theta functions is equal to zero; the second product is similar.  We have
\begin{align*}
j(q^{6n+16}\zeta_n^2y^4;q^{4(2n+4)})&j(q^{n+4}\zeta_n y^2;q^{2(2n+4)})\\
&=j(q^{(k+2)2(2n+4)};q^{4(2n+4)})j(q^{(k+1)(2n+4)};q^{2(2n+4)})=0,
\end{align*}
where the last equality follows from the fact that when $k$ is even that the first theta function is zero, and that when $k$ is odd that the second theta function is zero.
\end{proof}

\begin{lemma}  \label{lemma:Tn21-residue} Let $n$ and $\zeta_n$ be as in Definition \ref{definition:defTn21}.  $T_{n,2}^1(y,q)$ is meromorphic for $y\ne 0$ and has simple poles at points $y_0$, where $y_0$ satisfies one of the following two conditions:
\begin{align*}
\textup{I}.& \ y_0^2=-\zeta_n^{-2}q^{4\ell (n+1)-n-2}, \\
\textup{II}.& \  y_0^2=-q^{4\ell (n+1)-n-2},
\end{align*}
where $\ell\in \mathbb{Z}$.  The respective residue at such a $y_0$ is given by,
\begin{align*}
\textup{I}.& \ \ \ {(-1)^{\ell+1}\zeta_n^{-2+\ell-(n-1)/2}q^{-4(n+1)\binom{\ell-1}{2}-\tfrac12(n^2-4n-3)}J_{2n,4n}}/\big ({2nJ_{4n(n+1)}^3}\big ),\\
\textup{II}.& \ \ \ {(-1)^{\ell}\zeta_n^{-\ell+1-(n-1)/2}q^{-4(n+1)\binom{\ell-1}{2}-\tfrac12(n^2-4n-3)}J_{2n,4n}}/\big ({2nJ_{4n(n+1)}^3}\big ).
\end{align*}
\end{lemma}
\begin{proof} The first part follows from Definition \ref{definition:defTn21}.  We prove the residue for poles of type \textup{I}.  The proof of the residue for poles of type \textup{II} is similar.  We take such a $y_0$ and use Proposition \ref{proposition:H1Thm1.3} to obtain
{\allowdisplaybreaks \begin{align}
\lim_{y\rightarrow y_0}(y-y_0)T_{n,2}^1(y,q)=&\frac{(-1)^{\ell+1}y_0^{4}q^{4(n+1)\binom{\ell}{2}-(n^2-3)/2}}{2n\zeta_n^{(n-5)/2}}\cdot
\frac{J_{2n,4n}J_{4(n+1),8(n+1)}}{J_{4n(n+1)}^3J_{4(n+1)}^3}\notag\\
&\cdot
\frac{j(\zeta_n^{-1};q^{4(n+1)})j(q^{n+2}\zeta_n y_0^2;q^{4(n+1)})j(q^{2n}/\zeta_n^2y_0^4;q^{8(n+1)})}{j(-q^{n+2}y_0^2;q^{4(n+1)})}.\label{equation:Tn21I-residue}
\end{align}}%
We rewrite the quotient of theta functions in the second line of (\ref{equation:Tn21I-residue}).  For the first numerator term,
\begin{align*}
j(\zeta_n^{-1};q^{4(n+1)})&=j(q^{4(n+1)}\zeta_n;q^{4(n+1)})=-\zeta_n^{-1}j(\zeta_n;q^{4(n+1)}).&(\text{by }(\ref{equation:1.7}), (\ref{equation:1.8}))
\end{align*}
Substituting for $y_0^2$ and using (\ref{equation:1.8}), the second term in the numerator can be written
\begin{align*}
j(q^{n+2} \zeta_n y_0^2;q^{4(n+1)})&=j(-\zeta_n^{-1}q^{4\ell (n+1)};q^{4(n+1)})=(-1)^{\ell}q^{-4(n+1)\binom{\ell}{2}}(-\zeta_n^{-1})^{-\ell}j(-\zeta_n^{-1};q^{4(n+1)})\\
&=\zeta_n^{\ell-1}q^{-4(n+1)\binom{\ell}{2}}j(-\zeta_n;q^{4(n+1)}).
\end{align*}
Again using (\ref{equation:1.8}), the third term in the numerator becomes
\begin{align*}
j(q^{2n}  /\zeta_n^2y_0^4;q^{8(n+1)})&=j(\zeta_n^2q^{-8\ell(n+1)+4(n+1)};q^{8(n+1)})\\
&=(-1)^{\ell}\zeta_n^{2\ell}q^{-4(n+1)\ell^2}j(\zeta_n^2q^{4(n+1)};q^{8(n+1)}).
\end{align*}
Also using (\ref{equation:1.8}), the term in the denominator can be written
\begin{align*}
j(-q^{n+2}y_0^2;q^{4(n+1)})&=j(\zeta_n^{-2}q^{4\ell (n+1)};q^{4(n+1)})
=(-1)^{\ell}q^{-4(n+1)\binom{\ell}{2}}(\zeta_n^{-2})^{-\ell}j(\zeta_n^{-2};q^{4(n+1)})\\
&=(-1)^{\ell+1}\zeta_n^{2(\ell-1)}q^{-4(n+1)\binom{\ell}{2}}j(\zeta_n^{2};q^{4(n+1)}).
\end{align*}
Furthermore, we can use elementary theta function properties to show
\begin{align*}
\frac{j(\zeta_n;q^{4(n+1)})j(-\zeta_n;q^{4(n+1)})j(\zeta_n^2q^{4(n+1)};q^{8(n+1)})}{j(\zeta_n^2;q^{4(n+1)})}=&\frac{j(\zeta_n^2;q^{8(n+1)})j(\zeta_n^2q^{4(n+1)};q^{8(n+1)})J_{4(n+1)}^2}{j(\zeta_n^2;q^{4(n+1)})J_{8(n+1)}}\\
=&J_{4(n+1)}J_{8(n+1)}={J_{4(n+1)}^3}/{J_{4(n+1),8(n+1)}}.
\end{align*}
Substituting the four rewritten theta functions into (\ref{equation:Tn21I-residue}) and simplifying produces the desired result.
\end{proof}

\begin{lemma}  \label{lemma:Tn31-residue} Let $n$ and $\zeta_n$ be as in Definition \ref{definition:defTn31}.  $T_{n,3}^1(y,q)$ is meromorphic for $y\ne 0$ and has simple poles at points $y_0$, where $y_0$ satisfies one of the following two conditions:
\begin{align*}
\textup{I}. & \ y_0^3=\zeta_n^{-3}q^{3\ell^* (2n+3)-3(n+3)/2} \ (\text{or } y_0=\omega\zeta_n^{-1}q^{\ell^* (2n+3)-(n+3)/2}),\\
\textup{II}. & \ y_0^3=q^{3\ell^* (2n+3)-3(n+3)/2} \ (\text{or } y_0=\omega q^{\ell^* (2n+3)-(n+3)/2}) ,
\end{align*}
where $\omega^3=1$ with $\omega\ne 1$ and where $\ell^*:=\ell-\{(n-1)/2\}$ with $\ell\in \mathbb{Z}$ and $0\le\{x\}<1$ denoting the fractional part of $x$.  (The poles where $\omega=1$ are easily seen to be removable.)  The respective residue at such a $y_0$ is given by
\begin{align*}
\textup{I}.& \ \ \ {(-1)^{n}\zeta_n^{-2+\ell^*-(n-1)/2}\omega^{1-2\ell^*}(1-\omega)q^{-2(2n+3)\binom{\ell^*-1}{2}-\tfrac34(n-3)(n+1)}J_{3n}}/\big ({3nJ_{3n(2n+3)}^3}\big ),\\
\textup{II}.& \ \ \ {(-1)^{n+1}\zeta_n^{1-\ell^*-(n-1)/2}\omega^{1-2\ell^*}(1-\omega)q^{-2(2n+3)\binom{\ell^*-1}{2}-\tfrac34(n-3)(n+1)}J_{3n}}/\big ({3nJ_{3n(2n+3)}^3}\big ).
\end{align*}
\end{lemma}

\begin{proof} [Proof of Lemma \ref{lemma:Tn31-residue}] We prove the $n$ even case because the equations involved are more compact.  The $n$ odd case is similar.  The first part follows from Definition \ref{definition:defTn31}.  When we consider the corresponding pole for $R_{n,3}^1(y,q),$ we must choose the same third root of unity.  We prove the residue for poles of type \textup{I}.  The proof for poles of type \textup{II} is similar.  Using Proposition \ref{proposition:H1Thm1.3} and (\ref{equation:1.8}) yields
\begin{align}
\lim_{y\rightarrow y_0}&(y-y_0)T_{n,3}^1(y,q)=\frac{(-1)^{\ell+1}\zeta_n^{1-n/2}y_0 q^{3(2n+3)\binom{\ell}{2}-n(3n+2)/4}}{3n}\cdot
\frac{J_{3n}J_{3(2n+3)}}{J_{3(2n+3)}^3J_{3n(2n+3)}^3J_{2n+3}^2} \notag\\
\cdot&
\frac{j(\zeta_n^{-1};q^{3(2n+3)})j(q^{3n/2+3}\zeta_n y_0,q^{3n/2+3}y_0;q^{2n+3})}
{j(q^{9n/2+9}y_0^3;q^{3(2n+3)})} \label{equation:Tn31I-residue}
\\
\cdot & \Big \{ j(q^{5n/2+6}\zeta_n^2y_0^3,q^{5n/2+6}\zeta_n y_0^3;q^{3(2n+3)})-\frac{q^{n}}{\zeta_n y_0^2}j(q^{n/2+3}\zeta_n^2y_0^3,q^{n/2+3}\zeta_n y_0^3;q^{3(2n+3)})\Big \}.\notag
\end{align}
Substituting in for $y_0$, the first line of (\ref{equation:Tn31I-residue}) can be written
\begin{equation}
{(-1)^{\ell+1}\omega \zeta_n^{-n/2}q^{3(2n+3)\binom{\ell}{2} +\ell(2n+3)-3n/2-3}}\cdot {J_{3n}}/{3nq^{n(3n+2)/4}J_{3(2n+3)}^2J_{3n(2n+3)}^3J_{2n+3}^2}.\label{equation:genfn3-firstline}
\end{equation}
We rewrite the quotient of theta functions in the second line of (\ref{equation:Tn31I-residue}).  The first theta function in the numerator becomes
\begin{align*}
j(\zeta_n^{-1};q^{3(2n+3)})=j(q^{3(2n+3)}\zeta_n;q^{3(2n+3)})=-\zeta_n^{-1}j(\zeta_n;q^{3(2n+3)}).
\end{align*}
Substituting for $y_0$, the second theta function in the numerator can be written
\begin{align*}
j(q^{3n/2+3}\zeta_n y_0;q^{2n+3})&=j(\omega q^{\ell (2n+3)};q^{2n+3})=(-1)^{\ell}q^{-(2n+3)\binom{\ell}{2}}\omega^{-\ell}j(\omega;q^{2n+3})&({\text{by }} (\ref{equation:1.8}))\\
&=(-1)^{\ell}\omega^{-\ell}(1-\omega)q^{-(2n+3)\binom{\ell}{2}}J_{3(2n+3)}.
\end{align*}
Using (\ref{equation:1.8}), the third theta function in the numerator can be written
\begin{align*}
j(q^{3n/2+3}y_0;q^{2n+3})&=j(\omega \zeta_n^{-1}q^{\ell(2n+3)};q^{2n+3})
=(-1)^{\ell}q^{-(2n+3)\binom{\ell}{2}}(\omega \zeta_n^{-1})^{-\ell}j(\omega \zeta_n^{-1};q^{2n+3}) \\
&=(-1)^{\ell+1}\omega^{1-\ell}\zeta_n^{\ell-1}q^{-(2n+3)\binom{\ell}{2}}j(\omega^2 \zeta_n;q^{2n+3}).
\end{align*}
Again using (\ref{equation:1.8}), the theta function in the denominator becomes
\begin{align*}
j(q^{9n/2+9}y_0^3;q^{3(2n+3)})&=j(\zeta_n^{-3}q^{3\ell (2n+3)};q^{3(2n+3)})
=q^{-3(2n+3)\binom{\ell}{2}}(-\zeta_n^{-3})^{-\ell}j(\zeta_n^{-3};q^{3(2n+3)})\\
&=(-1)^{\ell+1}\zeta_n^{3(\ell-1)}q^{-3(2n+3)\binom{\ell}{2}}j(\zeta_n^{3};q^{3(2n+3)}).
\end{align*}
Assembling the pieces, the second line of (\ref{equation:Tn31I-residue}) can thus be written
\begin{equation}\label{equation:genfn3-secondline}
{(-1)^{\ell+1}\omega^{1-2\ell}(1-\omega)\zeta_n^{1-2\ell}q^{(2n+3)\binom{\ell}{2}}J_{3(2n+3)}j(\zeta_n;q^{3(2n+3)})j(\omega^2\zeta_n;q^{2n+3})}/{j(\zeta_n^3;q^{3(2n+3)})}.
\end{equation}
We rewrite the expression in braces from (\ref{equation:Tn31I-residue}).  The first theta function of the first product becomes
\begin{align*}
j(q^{5n/2+6}\zeta_n^2y_0^3;&q^{3(2n+3)})=j(\zeta_n^{-1}q^{3\ell(2n+3)-(2n+3)};q^{3(2n+3)})\\
&=(-1)^{\ell}q^{-3(2n+3)\binom{\ell}{2}}(\zeta_n^{-1}q^{-(2n+3)})^{-\ell} j(\zeta_n^{-1}q^{-(2n+3)};q^{3(2n+3)})&({\text{by }} (\ref{equation:1.8}))\\
&=(-1)^{\ell+1}\zeta_n^{\ell-1}q^{-3(2n+3)\binom{\ell}{2}+(\ell-1)(2n+3)}j(\zeta_nq^{2n+3};q^{3(2n+3)}).
\end{align*}
The arguments for the other three theta functions are similar.  We have
\begin{align*}
j(q^{5n/2+6}\zeta_ny_0^3;q^{3(2n+3)})
&=(-1)^{\ell+1}\zeta_n^{2(\ell-1)}q^{-3(2n+3)\binom{\ell}{2}+(\ell-1)(2n+3)}j(\zeta_n^2 q^{2n+3};q^{3(2n+3)}),\\
j(q^{n/2+3}\zeta_n^2y_0^3;q^{3(2n+3)})
&=(-1)^{\ell+1}\zeta_n^{\ell-1}q^{-3(2n+3)\binom{\ell}{2}+2(\ell-1)(2n+3)}j(\zeta_n q^{2(2n+3)};q^{3(2n+3)}),\\
j(q^{n/2+3}\zeta_n y_0^3;q^{3(2n+3)})
&=(-1)^{\ell+1}\zeta_n^{2(\ell-1)}q^{-3(2n+3)\binom{\ell}{2}+2(\ell-1)(2n+3)}j(\zeta_n^2 q^{2(2n+3)};q^{3(2n+3)}).
\end{align*}
Noting that ${q^n}/({\zeta_n y_0^2})=\omega \zeta_n q^{n-2\ell(2n+3)+3n+6}=\omega \zeta_n q^{-2(\ell-1)(2n+3)}$, the braced expression becomes
\begin{align*}
\zeta_n^{3(\ell-1)}&q^{-6(2n+3)\binom{\ell}{2}+2(\ell-1)(2n+3)}
\cdot\Big ( j(\zeta_nq^{2n+3},\zeta_n^2q^{2n+3};q^{3(2n+3)})\notag \\
&-\omega \zeta_n j(\zeta_nq^{2(2n+3)},\zeta_n^2q^{2(2n+3)};q^{3(2n+3)})\Big ).
\end{align*}
Using Proposition \ref{proposition:prop-f141} with the substitutions $\omega=\omega^2$, $q= q^{2n+3}$, and $y= \zeta_n q^{2n+3}$, we obtain
\begin{equation}
\zeta_n^{3(\ell-1)}q^{-6(2n+3)\binom{\ell}{2}+2(\ell-1)(2n+3)}
j(\omega \zeta_n;q^{2n+3})
\cdot \frac{ j(\zeta_n q^{2(2n+3)},\zeta_n q^{2n+3};q^{3(2n+3)})}{J_{3(2n+3)}}. \label{equation:genfn3-thirdline}
\end{equation}
Putting (\ref{equation:genfn3-firstline}), (\ref{equation:genfn3-secondline}), and (\ref{equation:genfn3-thirdline}) together yields
\begin{align*}
\lim_{y\rightarrow y_0}&(y-y_0)T_{n,3}^1(y,q)
=\frac{\zeta_n^{-2+\ell-n/2}\omega^{\ell-1}(1-\omega)q^{-(2n+3)(\ell^2-4\ell+3)-\tfrac34n^2}}{3n}\cdot\frac{J_{3n}}{J_{3(2n+3)}^2J_{3n(2n+3)}^3J_{2n+3}^2} \\
&\cdot\frac{j(\zeta_n,\zeta_nq^{2n+3},\zeta_nq^{2(2n+3)};q^{3(2n+3)})j(\omega \zeta_n,\omega^2\zeta_n;q^{2n+3})}{j(\zeta_n^3,q^{3(2n+3)})}\cdot
\end{align*} 
Using identities (\ref{equation:1.10}) and (\ref{equation:1.12}), the above second line becomes $J_{3(2n+3)}^2J_{2n+3}^2$, and the result follows.
\end{proof}

\begin{lemma}  \label{lemma:Tn41-residue} Let $n$ and $\zeta_n$ be as in Definition \ref{definition:defTn41}.  $T_{n,4}^1(y,q)$ is meromorphic for $y\ne 0$ and has simple poles at points $y_0$, where $y_0$ satisfies one of the following two conditions:
\begin{align*}
\textup{I}. & \ y_0^4=-\zeta_n^{-4}q^{4\ell (2n+4)-(2n+8)} \ (\text{or } y_0^2=i\zeta_n^{-2}q^{2\ell (2n+4)-(n+4)}),\\
\textup{II}. & \  y_0^4=-q^{4\ell (2n+4)-(2n+8)} \ (\text{or } y_0^2=iq^{2\ell (2n+4)-(n+4)}),
\end{align*}
where $\ell\in \mathbb{Z}$ and $i^2=-1$.  The respective residue at such a $y_0$ is given by
\begin{align*}
\textup{I}.& \ \ \ {i^{1-\ell}\zeta_n^{\ell-2-(n-1)/2}}q^{-2(2n+4)\binom{\ell-1}{2}-n^2+n+3}
 \frac{ (1-i)J_{4n,16n}+i\zeta_n y_0q^{2-\ell (2n+4)}J_{8n,16n}}{4nJ_{4n(2n+4)}^3},\\
\textup{II}.& \ \ \ {-i^{1-\ell}\zeta_n^{-\ell+1-(n-1)/2}}q^{-2(2n+4)\binom{\ell-1}{2}-n^2+n+3}
\frac{ (1-i)J_{4n,16n}+i y_0q^{2-\ell (2n+4)}J_{8n,16n}}{4nJ_{4n(2n+4)}^3} .
\end{align*}
\end{lemma}
\begin{proof}  The first part follows from Definition \ref{definition:defTn41}.  When we identify the corresponding pole of $R_{n,4}^1(y,q)$ in Corollary \ref{corollary:Rn41-residue}, the pole must involve the same primitive fourth root of unity. We prove the residue for poles of type \textup{I}.  The proof of the residue for poles of type \textup{II} is similar.  We take such a $y_0$ and use Proposition \ref{proposition:H1Thm1.3} to obtain
{\allowdisplaybreaks \begin{align}
\lim_{y\rightarrow y_0}(y-y_0)&T_{n,4}^1(y,q)=\frac{(-1)^{1+\ell}\zeta_n^{-(n-5)/2}y_0^4q^{4(2n+4)\binom{\ell}{2}-(n^2+n-3)}j(\zeta_n^{-1};q^{4(2n+4)})}{4nJ_{4n(2n+4)}^3J_{4(2n+4)}^3j(-q^{2n+8}y_0^4;q^{4(2n+4)})} \notag\\
\cdot \Big [&J_{4n,16n}\cdot \frac{j(q^{6n+16}\zeta_n^2y_0^4;q^{4(2n+4)})j(q^{n+4}\zeta_n y_0^2;q^{2(2n+4)})j(-q^{2(2n+4)}\zeta_n^{-1};q^{4(2n+4)})}{J_{2(2n+4)}^3J_{8(2n+4)}} \notag \\
&\cdot  \Big \{j(-q^{2n+8}\zeta_n^2y_0^4;q^{4(2n+4)})j(q^{2(2n+4)}\zeta_n^{-2};q^{4(2n+4)})J_{4(2n+4)}^2\notag\\
&+\frac{q^{n+4}\zeta_n^2y_0^2 j(-q^{6n+16}\zeta_n^2y_0^4;q^{4(2n+4)})j(q^{4n+8}\zeta_n^{-1};q^{4(2n+4)})^2j(-\zeta_n^{-1};q^{4(2n+4)})^2}{J_{4(2n+4)}}\Big \}\notag\\
-&qy_0^{-1}J_{8n,16n}\cdot \frac{j(q^{2n+8}\zeta_n^2y_0^4;q^{4(2n+4)})j(q^{3n+8}\zeta_ny_0^2;q^{2(2n+4)})j(-\zeta_n^{-1};q^{4(2n+4)})}{J_{2(2n+4)}^2} \notag\\
&\cdot  \Big \{\frac{q^{n+1}j(-q^{2n+8}\zeta_n^2y_0^4;q^{4(2n+4)})j(q^{4n+8}\zeta_n^{-2};q^{4(2n+4)})J_{8(2n+4)}}{J_{4(2n+4)} }\notag\\ 
&+\frac{q\zeta_n y_0^2j(-q^{6n+16}\zeta_n^2y_0^4;q^{4(2n+4)})j(q^{4(2n+4)}\zeta_n^{-2};q^{8(2n+4)})^2}{J_{8(2n+4)} }\Big \} \Big ].\label{equation:Tn41-residue}
\end{align}}%
We rewrite the above residue line by line.  We focus on the first line of (\ref{equation:Tn41-residue}).  Using (\ref{equation:1.8}),
\begin{equation*}
j(\zeta_n^{-1};q^{4(2n+4)})=-\zeta_n^{-1}j(\zeta_n;q^{4(2n+4)}).
\end{equation*}
Substituting for $y_0^4$, using  (\ref{equation:1.8}) and simplifying, we obtain
\begin{equation*}
j(-q^{2n+8}y_0^4;q^{4(2n+4)})=(-1)^{1+\ell}\zeta_n^{4(\ell-1)}q^{-4(2n+4)\binom{\ell}{2}}j(\zeta_n^4;q^{4(2n+4)}).
\end{equation*}
Hence the first line can be rewritten
\begin{equation}
{\zeta_n^{-4\ell+1-(n-1)/2}q^{(2n+4)4\ell^2-n^2-3n-5} j(\zeta_n;q^{4(2n+4)})}/{4nJ_{4n(2n+4)}^3J_{4(2n+4)}^3j(\zeta_n^4;q^{4(2n+4)})}.\label{equation:Tn41-Q}
\end{equation}
We work on the second line of (\ref{equation:Tn41-residue}).  Substituting for $y_0^4$ or $y_0^2$ and using (\ref{equation:1.8}) yields
\begin{align*}
j(q^{6n+16}\zeta_n^2y_0^4;q^{4(2n+4)})&=\zeta_n^{2\ell}q^{-2\ell^2(2n+4)}j(-q^{2(2n+4)}\zeta_n^2;q^{4(2n+4)}),\\
j(q^{n+4}\zeta_ny_0^2;q^{2(2n+4)})&=(-1)^{\ell+1}i^{1-\ell}\zeta_n^{\ell-1}q^{-2(2n+4)\binom{\ell}{2}}j(-i\zeta_n;q^{2(2n+4)}),\\
j(-q^{2(2n+4)}\zeta_n^{-1};q^{4(2n+4)})&=j(-\zeta_nq^{2(2n+4)};q^{4(2n+4)}).
\end{align*}
Hence the second line can be rewritten
\begin{equation}
\frac{(-1)^{1+\ell}i^{1-\ell}\zeta_n^{3\ell-1}J_{4n,16n}j(-i\zeta_n;q^{2(2n+4)})j(-\zeta_nq^{2(2n+4)},-q^{2(2n+4)}\zeta_n^2;q^{4(2n+4)})}
{q^{(2n+4)(3\ell^2-\ell)}J_{2(2n+4)}^3J_{8(2n+4)}}.\label{equation:Tn41-A}
\end{equation}
We focus on the sum of the third and fourth lines of (\ref{equation:Tn41-residue}).  Substituting for $y_0^4$ and using  (\ref{equation:1.8}),
\begin{align}
(-1)^{\ell +1}\zeta_n^{2\ell-2}q^{-4(2n+4)\binom{\ell}{2}}&j(q^{2(2n+4)}\zeta_n^2;q^{4(2n+4)})\label{equation:cat3}\\
\cdot &\Big \{ j(\zeta_n^2;q^{4(2n+4)})J_{4(2n+4)}^2-i\frac{j(q^{2(2n+4)}\zeta_n,-\zeta_n;q^{4(2n+4)})^2}{J_{4(2n+4)}}\Big \}.\notag
\end{align}
Using (\ref{equation:1.12}) and then (\ref{equation:1.10}), we have
\begin{align*}
j(\zeta_n^2;q^{4(2n+4)})&=j(\zeta_n,-\zeta_n;q^{2(2n+4)})\cdot {J_{4(2n+4)}}/{J_{2(2n+4)}^2}\\
=j(\zeta_n&,\zeta_nq^{2(2n+4)},-\zeta_n,-\zeta_nq^{2(2n+4)};q^{4(2n+4)}) /{J_{4(2n+4)}^3}.
\end{align*}
Hence the above expression in braces in (\ref{equation:cat3}) becomes
\begin{align*}
\frac{j(q^{2(2n+4)}\zeta_n,-\zeta_n;q^{4(2n+4)})}{J_{4(2n+4)}}\cdot \Big \{ j(\zeta_n,-\zeta_nq^{2(2n+4)};q^{4(2n+4)}) -ij(q^{2(2n+4)}\zeta_n,-\zeta_n;q^{4(2n+4)})\Big \}.
\end{align*} 
Using (\ref{equation:H1Thm1.1}) with the substitutions, $x=i, \ y=i\zeta_n, \ q=q^{2(2n+4)},$ we have that
\begin{align*}
 j(\zeta_n;&q^{4(2n+4)})j(-\zeta_nq^{2(2n+4)};q^{4(2n+4)})-ij(q^{2(2n+4)}\zeta_n;q^{4(2n+4)})j(-\zeta_n;q^{4(2n+4)})\\
 &=j(i;q^{2(2n+4)})j(i\zeta_n;q^{2(2n+4)})
 =j(i\zeta_n;q^{2(2n+4)})(1-i){J_{2(2n+4)}J_{8(2n+4)}}/{J_{4(2n+4)}}.
\end{align*}
Thus the sum of the third and fourth lines can be written
\begin{align}
(1-i)\cdot (-1)^{\ell +1}\zeta_n^{2\ell-2}&q^{-4(2n+4)\binom{\ell}{2}}\frac{J_{2(2n+4)}J_{8(2n+4)}}{J_{4(2n+4)}^2}\notag\\
&\cdot j(q^{2(2n+4)}\zeta_n,-\zeta_n,q^{2(2n+4)}\zeta_n^2;q^{4(2n+4)})j(i\zeta_n;q^{2(2n+4)}).\label{equation:Tn41-alpha}
\end{align}
We proceed to the fifth line of (\ref{equation:Tn41-residue}).  We note that $y_0^{-1}=y_0\cdot y_0^{-2}=-i\zeta_n^2y_0q^{-2\ell(2n+4)+(n+4)}$. Substituting for $y_0^4$, using  (\ref{equation:1.8}), and simplifying, we obtain
\begin{align*}
j(q^{2n+8}\zeta_n^2y_0^4;q^{4(2n+4)})&=\zeta_n^{2(\ell-1)}q^{-4(2n+4)\binom{\ell}{2}}j(-\zeta_n^2;q^{4(2n+4)}),\\
j(q^{3n+8}\zeta_ny_0^2;q^{2(2n+4)})&=(-1)^{\ell}i^{-\ell}\zeta_n^{\ell}q^{-\ell^2(2n+4)}j(-i\zeta_nq^{2n+4};q^{2(2n+4)}),\\
j(-\zeta_n^{-1};q^{4(2n+4)})&=\zeta_n^{-1}j(-\zeta_n;q^{4(2n+4)}).
\end{align*}
Hence the fifth line can be rewritten
\begin{align}
(-1)^{\ell}i^{1-\ell}\zeta_n^{3\ell-1}y_0q^{n+5-(2n+4)3\ell^2}\cdot &\frac{J_{8n,16n}}{J_{2(2n+4)}^2}j(-\zeta_n^2,-\zeta_n;q^{4(2n+4)})j(-i\zeta_nq^{2n+4};q^{2(2n+4)}).\label{equation:Tn41-B}
\end{align}
We rewrite the sum of the sixth and seventh lines of (\ref{equation:Tn41-residue}).  Substituting for $y_0^4$ and $y_0^2$, using  (\ref{equation:1.8}) and simplifying, we obtain
\begin{align}
(-1)^{\ell+1}&\zeta_n^{2(\ell-1)}q^{-4(2n+4)\binom{\ell}{2}+n+1}j(\zeta_n^2q^{2(2n+4)};q^{4(2n+4)})\label{equation:cat4}\\
&\cdot \Big \{ j(\zeta_n^2;q^{4(2n+4)}){J_{8(2n+4)}}/{J_{4(2n+4)}}-i \zeta_n q^{-2n-4}j(\zeta_n^2q^{4(2n+4)};q^{8(2n+4)})^2/{J_{8(2n+4)}}\Big \}.\notag
\end{align}
Using (\ref{equation:1.10}) yields
\begin{equation*}
j(\zeta_n^2;q^{4(2n+4)})=j(\zeta_n^2;q^{8(2n+4)})j(\zeta_n^2q^{4(2n+4)};q^{8(2n+4)})\cdot {J_{4(2n+4)}}/{J_{8(2n+4)}^2},
\end{equation*}
thus the expression in braces in (\ref{equation:cat4}) becomes
\begin{align*}
{j(\zeta_n^2q^{4(2n+4)};q^{8(2n+4)})}\cdot& \Big \{ j(\zeta_n^2;q^{8(2n+4)})
-i \zeta_n q^{-2n-4}j(\zeta_n^2q^{4(2n+4)};q^{8(2n+4)})\Big\}/{J_{8(2n+4)}}\\
&={j(\zeta_n^2q^{4(2n+4)};q^{8(2n+4)})}\cdot j(i\zeta_nq^{-(2n+4)};q^{2(2n+4)})/{J_{8(2n+4)}}\\
&=-i\zeta_nq^{-(2n+4)}{j(\zeta_n^2q^{4(2n+4)};q^{8(2n+4)})}\cdot j(i\zeta_nq^{2n+4};q^{2(2n+4)})/{J_{8(2n+4)}},
\end{align*}
where the first equality follows from (\ref{equation:jsplit}) with $m=2$, and the second equality follows from (\ref{equation:1.8}).   Hence the sum of the sixth and seventh lines can be written
\begin{align}
\frac{(-1)^{\ell}i\zeta_n^{2\ell-1}q^{-4(2n+4)\binom{\ell}{2}-n-3}}{J_{8(2n+4)}}j(\zeta_n^2q^{2(2n+4)};q^{4(2n+4)})j(\zeta_n^2q^{4(2n+4)};q^{8(2n+4)}) j(i\zeta_nq^{2n+4};q^{2(2n+4)}).\label{equation:Tn41-beta}
\end{align}
Combining (\ref{equation:Tn41-Q}), (\ref{equation:Tn41-A}), (\ref{equation:Tn41-alpha}), (\ref{equation:Tn41-B}), and (\ref{equation:Tn41-beta}), we have
\begin{align*}
&\lim_{y\rightarrow y_0}(y-y_0)T_{n,4}^1(y,q)
=\frac{\zeta_n^{-4\ell+1-(n-1)/2}q^{(2n+4)4\ell^2-n^2-3n-5} j(\zeta_n;q^{4(2n+4)})}{4nJ_{4n(2n+4)}^3J_{4(2n+4)}^3j(\zeta_n^4;q^{4(2n+4)})}\\
&\cdot \Big [\frac{(-1)^{1+\ell}i^{1-\ell}\zeta_n^{3\ell-1}J_{4n,16n}j(-q^{2(2n+4)}\zeta_n^2;q^{4(2n+4)})j(-i\zeta_n;q^{2(2n+4)})j(-\zeta_nq^{2(2n+4)};q^{4(2n+4)})}{q^{(2n+4)(3\ell^2-\ell)}J_{2(2n+4)}^3J_{8(2n+4)}}\\
&\cdot(1-i)\cdot  (-1)^{\ell +1}\zeta_n^{2\ell-2}q^{-4(2n+4)\binom{\ell}{2}}\frac{J_{2(2n+4)}J_{8(2n+4)}}{J_{4(2n+4)}^2}\cdot \notag \\
&\cdot j(q^{2(2n+4)}\zeta_n,-\zeta_n,q^{2(2n+4)}\zeta_n^2;q^{4(2n+4)})j(i\zeta_n;q^{2(2n+4)})\\
&+(-1)^{\ell}i^{1-\ell}\zeta_n^{3\ell-1}y_0q^{-(2n+4)3\ell^2+n+5}\cdot \frac{J_{8n,16n}}{J_{2(2n+4)}^2}j(-\zeta_n^2,-\zeta_n;q^{4(2n+4)})j(-i\zeta_nq^{2n+4};q^{2(2n+4)})\\
&\cdot\frac{(-1)^{\ell}i\zeta_n^{2\ell-1}}{q^{4(2n+4)\binom{\ell}{2}+n+3}J_{8(2n+4)}}j(\zeta_n^2q^{2(2n+4)};q^{4(2n+4)})j(\zeta_n^2q^{4(2n+4)};q^{8(2n+4)}) j(i\zeta_nq^{2n+4};q^{2(2n+4)})\Big ].
\end{align*}
Collecting terms and simplifying using (\ref{equation:1.10}) and (\ref{equation:1.12}) produces the desired result.
\end{proof}


\subsection{Functional equations and poles for $T_{n,p}^2(y,q)$}

\begin{definition} \label{definition:defTn22}Let $n\in \mathbb{N}$ with $(n,2)=1$, $y\in \mathbb{C}^*$ generic, and $x_0$ such that $x_0^2=-q^{4\ell (n+1)-n-2}$ for some $\ell \in \mathbb{Z}$.  Then
\begin{align*}
T_{n,2}^2(y,q):&=
\frac{(-1)^{\ell+1}q^{4(n+1)\binom{\ell}{2}}y^{(n+1)/2}}{2q^{(n^2-3)/2}x_{0}^{(n-5)/2}}\cdot
\frac{J_{2n,4n}J_{4(n+1),8(n+1)}}{J_{4(n+1)}^3}\\
&\ \ \ \ \cdot\frac{j(y/x_0,q^{n+2}x_0y;q^{4(n+1)})j(q^{2n}/x_0^2y^2;q^{8(n+1)})}{j(y^n/x_0^n;q^{4n(n+1)})j(-q^{n+2}y^2;q^{4(n+1)})}.
\end{align*}
\end{definition}

\begin{definition}  \label{definition:defTn32} Let $n\in \mathbb{N}$ with $(n,3)=1$, $y\in\mathbb{C}^*$ generic, and $x_0$ such that $x_0^3=q^{3\ell (2n+3)-3(n+3)/2-3\beta}$ for some $\ell \in \mathbb{Z}$.  Also, define $\delta:=\{(n-1)/2\}$ and $\beta:=\delta\cdot (2n+3)$, with $0\le \{\alpha \}<1$ denoting the fractional part of $\alpha.$  Then
\begin{align}
&T_{n,3}^2(y,q):=q^{n\binom{\delta-(n-3)/2}{2}+(n+3)\big (\delta-(n-3)/2\big )\big (\delta+(n+1)/2\big )+n\binom{\delta+(n+1)/2}{2}-3n}(-x_0)^{\delta-(n-3)/2}\notag\\
&\cdot\frac{(-1)^{\ell+1}q^{3(2n+3)\binom{\ell}{2}} x_0(-y)^{\delta+(n+1)/2} J_{3n}j(y/x_0;q^{3(2n+3)})j(q^{(n+3)/2+\beta}x_0,q^{(n+3)/2+\beta}y;q^{2n+3})}
 {3J_{2n+3}^2J_{3(2n+3)}^2j(y^{n}/x_0^{n};q^{3n(2n+3)})j(q^{3(n+3)/2+3\beta}y^3;q^{3(2n+3)})}\notag\\
& \cdot  \Big \{ j(q^{7(n+1)/2+4+3\beta}x_0^2y,q^{7(n+1)/2+4+3\beta}x_0y^2;q^{3(2n+3)})\label{equation:Tn32-def}\\
&\ \  -{q^{n+3+2\beta}}{x_0y}j(q^{11(n+1)/2+5+3\beta}x_0^2y,q^{11(n+1)/2+5+3\beta}x_0y^2;q^{3(2n+3)})\Big \}.\notag
\end{align}
\end{definition}

\begin{definition}  \label{definition:defTn42} Let $n\in \mathbb{N}$ with $(n,2)=1$, $y\in \mathbb{C}^*$ generic, and $x_0$ such that $x_0^4=-q^{4\ell (2n+4)-(2n+8)}$ for some $\ell \in \mathbb{Z}$. Then
{\allowdisplaybreaks
 \begin{align}
T_{n,4}^2(y,q):&=\frac{(-1)^{1+\ell}x_0^{-(n-5)/2}y^{(n+1)/2}q^{4(2n+4)\binom{\ell}{2}-(n^2+n-3)}j(y/x_0;q^{4(2n+4)})}{4J_{4(2n+4)}^3j(y^n/x_0^n;q^{4n(2n+4)})j(-q^{2n+8}y^4;q^{4(2n+4)})} \notag\\
&  \cdot \Big [J_{4n,16n}\cdot \frac{j(q^{6n+16}x_0^2y^2;q^{4(2n+4)})j(q^{n+4}x_0 y;q^{2(2n+4)})j(-q^{2(2n+4)}y/x_0;q^{4(2n+4)})}{J_{2(2n+4)}^3J_{8(2n+4)}} \notag \\
&\ \ \ \ \cdot  \Big \{j(-q^{2n+8}x_0^2y^2;q^{4(2n+4)})j(q^{2(2n+4)}y^2/x_0^2;q^{4(2n+4)})J_{4(2n+4)}^2\notag\\
&\ \ \ \ +\frac{q^{n+4}x_0^2 j(-q^{6n+16}x_0^2y^2;q^{4(2n+4)})j(q^{4n+8}y/x_0;q^{4(2n+4)})^2j(-y/x_0;q^{4(2n+4)})^2}{J_{4(2n+4)}}\Big \}\notag\\
&\ \ -qJ_{8n,16n}\cdot \frac{j(q^{2n+8}x_0^2y^2;q^{4(2n+4)})j(q^{3n+8}x_0y;q^{2(2n+4)})j(-y/x_0;q^{4(2n+4)})}{J_{2(2n+4)}^2} \notag\\
&\ \ \ \ \cdot  \Big \{\frac{q^{n+1}j(-q^{2n+8}x_0^2y^2;q^{4(2n+4)})j(q^{4n+8}y^2/x_0^2;q^{4(2n+4)})J_{8(2n+4)}}{yJ_{4(2n+4)} }\notag\\ 
&\ \ \ \ +\frac{qx_0j(-q^{6n+16}x_0^2y^2;q^{4(2n+4)})j(q^{4(2n+4)}y^2/x_0^{2};q^{8(2n+4)})^2}{J_{8(2n+4)} }\Big \} \Big ].\label{equation:Tn42-def}
\end{align}}
\end{definition}

\begin{lemma} \label{lemma:Tnp2-def} For $p=2$ (resp. $3,4$), let $n$, $y$, and $x_0$ be as Definition \ref{definition:defTn22} (resp. \ref{definition:defTn32}, \ref{definition:defTn42}).  Then
\begin{equation*}
\lim_{x\rightarrow x_0}(x-x_0)\Theta_{n,p}(x,y,q)=T_{n,p}^2(y,q).
\end{equation*}
\end{lemma}
\begin{proof}
This follows from Proposition \ref{proposition:H1Thm1.3}.  The $n$ even and $n$ odd specializations of $\Theta_{n,3}(x,y,q)$ can be found in Section \ref{section:Tnp1}.
\end{proof}
\begin{lemma}  \label{lemma:Tnp2-functional}For $p=2$ (resp. $3,4$), let $n$, $y$, and $x_0$ be as Definition \ref{definition:defTn22} (resp. \ref{definition:defTn32}, \ref{definition:defTn42}).  Then
\begin{equation*}
T_{n,p}^2(q^{p(2n+p)}y,q)=q^{n(np+\binom{p+1}{2})}{(-y)^n}{(-x_0)^{-(n+p)}}T_{n,p}^2(y,q).
\end{equation*}
\end{lemma}
\begin{proof}  This is just repeated application of (\ref{equation:1.8}).
\end{proof}

\begin{lemma}  \label{lemma:Tn22-residue} Let $n$ and $x_0$ be as in Definition \ref{definition:defTn22} and $\zeta_n$ an $n$-th root of unity.  $T_{n,2}^2(y,q)$ is meromorphic for $y\ne 0$ and has poles at points $y_0$, where $y_0$ satisfies at least of the following two conditions:
\begin{align*}
\textup{I}. & \ y_0=\zeta_n x_0q^{4t (n+1)},\\
\textup{II}. & \ y_0^2=-q^{4t (n+1)-n-2}.
\end{align*}
If $y_0$ satisifies \textup{I} or \textup{II} exclusively, then it is a simple pole.  When this is the case, the residues are as follows.  For poles of type \textup{I}, the respective residue at such a $y_0$ for $t=0$ is
\begin{align*}
\textup{I}.& \ \ \ (-1)^{\ell}\zeta_n^{-\ell +1+(n+1)/2}q^{-4(n+1)\binom{\ell-1}{2}-\tfrac12(n^2-4n-3)}{J_{2n,4n}}/\big ({2nJ_{4n(n+1)}^3}\big ).
\end{align*}
Given the residue at $t=0$, one can use the functional equation of Lemma \ref{lemma:Tnp2-functional} to compute the residue for general $t\in \mathbb{Z}$.  Poles of type \textup{II} are removable.
\end{lemma}
\begin{proof}  We prove the residue for poles of type \textup{I} where $\zeta_n\ne 1$.  The proof where $\zeta_n=1$ is similar. Using Proposition \ref{proposition:H1Thm1.3}, we obtain
\begin{align}
\lim_{y\rightarrow x_0 \zeta_n} (y-\zeta_n x_0) T_{n,2}^2(y,q)&=\frac{(-1)^{\ell}\zeta_n^{(n+3)/2}x_0^4}{2n}\cdot \frac{q^{4(n+1)\binom{\ell}{2}}}{q^{(n^2-3)/2}}\cdot \frac{J_{2n,4n}J_{4(n+1),8(n+1)}}{J_{4(n+1)}^3J_{4n(n+1)}^3}\label{equation:Tn22-2nd-line}\\
&\ \ \ \ \cdot \frac{j(\zeta_n;q^{4(n+1)})j(q^{n+2}\zeta_nx_0^2;q^{4(n+1)})j(q^{2n}\zeta_n^{-2}x_0^{-4};q^{8(n+1)})}{j(-q^{n+2}\zeta_n^2x_0^2;q^{4(n+1)})}.\notag
\end{align}
We rewrite the first line of (\ref{equation:Tn22-2nd-line}).  Substituting $x_0^2=-q^{4\ell(n+1)-n-2}$ from Definition \ref{definition:defTn22} yields
\begin{equation}\label{equation-Tn22-residue-firstline}
\frac{(-1)^{\ell}\zeta_n^{(n+3)/2}}{2n}\cdot \frac{q^{4(n+1)\binom{\ell}{2}+8\ell (n+1)-2n-4}}{q^{(n^2-3)/2}}\cdot \frac{J_{2n,4n}J_{4(n+1),8(n+1)}}{J_{4(n+1)}^3J_{4n(n+1)}^3}.
\end{equation}
We rewrite the quotient in the second line of (\ref{equation:Tn22-2nd-line}).  The first theta function can be left as is.  Substituting for $x_0^2$, the second theta function can be written
\begin{align*}
j(-q^{n+2}\zeta_nx_0^2;q^{4(n+1)})&=j(-\zeta_nq^{4\ell(n+1)};q^{4(n+1)})
=\zeta_n^{-\ell}q^{-4(n+1)\binom{\ell}{2}}j(-\zeta_n;q^{4(n+1)}),
\end{align*}
by (\ref{equation:1.8}). The third theta function can be written
\begin{align*}
j(q^{2n}\zeta_n^{-2}x_0^{-4};q^{8(n+1)})&=j(\zeta_n^{-2}q^{-8\ell (n+1)+4(n+1)};q^{8(n+1)})\\
&=(-1)^{\ell}q^{-8(n+1)\binom{\ell+1}{2}}(\zeta_n^{-2}q^{4(n+1)})^{\ell}j(\zeta_n^{-2}q^{4(n+1)};q^{8(n+1)})&({\text{by }}(\ref{equation:1.8}))\\
&=(-1)^{\ell}\zeta_n^{-2\ell}q^{-8(n+1)\binom{\ell+1}{2}+4\ell (n+1)}j(\zeta_n^{2}q^{4(n+1)};q^{8(n+1)}).&({\text{by }}(\ref{equation:1.7}))
\end{align*}
The theta function in the denominator becomes
\begin{align*}
j(-q^{n+2}\zeta_n^2x_0^2;q^{4(n+1)})&=j(\zeta_n^2q^{4\ell (n+1)};q^{4(n+1)})
=(-1)^{\ell}\zeta_n^{-2\ell}q^{-4(n+1)\binom{\ell}{2}}j(\zeta_n^2;q^{4(n+1)}),
\end{align*}
by (\ref{equation:1.8}).  Combining the four terms, the second line of (\ref{equation:Tn22-2nd-line}) can thus be written
\begin{align}
\zeta_n^{-\ell}&q^{-8(n+1)\binom{\ell+1}{2}+4\ell (n+1)}
\frac{j(\zeta_n;q^{4(n+1)})j(-\zeta_n;q^{4(n+1)})j(\zeta_n^{2}q^{4(n+1)};q^{8(n+1)})}{j(\zeta_n^2;q^{4(n+1)})}\notag\\
&=\zeta_n^{-\ell}q^{-8(n+1)\binom{\ell+1}{2}+4\ell (n+1)}\frac{j(\zeta_n^2;q^{8(n+1)})j(\zeta_n^{2}q^{4(n+1)};q^{8(n+1)})}{j(\zeta_n^2;q^{4(n+1)})}\cdot \frac{J_{4(n+1)}^2}{J_{8(n+1)}}&({\text{by }}(\ref{equation:1.9}))\notag\\
&=\zeta_n^{-\ell}q^{-8(n+1)\binom{\ell+1}{2}+4\ell (n+1)}J_{4(n+1)}J_{8(n+1)},\label{equation-Tn22-residue-secondline}
\end{align}
where the last equality follows from (\ref{equation:1.10}) and simplifying.  Combining (\ref{equation-Tn22-residue-firstline}) and (\ref{equation-Tn22-residue-secondline}) produces the residue.

We prove that poles of type \textup{II} are removable.  Because we are excluding poles of type \textup{I}, it suffices to show that the numerator of $T_{n,2}^2(y_0,q)$ evaluates to zero.  Here we have that
\begin{align*}
x_0^2=&-q^{4\ell(n+1)-n-2}{\text{ and }}
y_0^2=-q^{4t(n+1)-n-2}.
\end{align*}
This implies that $x_0 y_0=\pm q^{2(t+\ell)(n+1)-n-2}$.  Hence if $t+\ell$ is odd, then a numerator term vanishes
\begin{equation*}
j(q^{2n}/x_0^2y_0^2;q^{8(n+1)})=j(q^{-4(t+\ell-1)(n+1)};q^{8(n+1)})=0.
\end{equation*}
If $t+\ell$ is even, we have two cases to deal with.  If $x_0 y_0=q^{2(t+\ell)(n+1)-n-2}$,
then in the numerator,
\begin{equation*}
j(q^{n+2}x_0y_0;q^{4(n+1)})=j(q^{2(t+\ell)(n+1)};q^{4(n+1)})=0.
\end{equation*}
If $x_0 y_0=-q^{2(t+\ell)(n+1)-n-2}$, then upon multiplying both sides by $x_0$ and using the value of $x_0^2$, we have that $y_0=x_0q^{2(t-\ell)(n+1)}$; however, this is a  pole of type \textup{I}, which is excluded.
\end{proof}

\begin{lemma}  \label{lemma:Tn32-residue} Let $n$ and $x_0$ be as in Definition \ref{definition:defTn32} and $\zeta_n$ be an $n$-th root of unity.  We write $x_0=\omega q^{\ell^*(2n+3)-(n+3)/2}$ where $\ell\in \mathbb{Z}$, $\ell^*:=\ell-\{(n-1)/2 \}$ with $0\le\{\alpha\}<1$ denoting the fraction part of $\alpha$, and $\omega^3=1$.  We can assume $\omega\ne 1$ because $\omega=1$ is a removable pole for $\Theta_{n,2}(x,y,q)$. $T_{n,3}^2(y,q)$ is meromorphic for $y\ne 0$ and has poles at points $y_0$, where $y_0$ satisfies at least one of the following two conditions:
\begin{align*}
\textup{I}.& \  y_0=\zeta_n x_0q^{3t (2n+3)},\\
\textup{II}.& \  y_0^3=q^{3t^*(2n+3)-3(n+3)/2} \ (\text{or } y_0=\omega^s q^{t^*(2n+3)-(n+3)/2}\text{ where }s\in \{0,1,2\}),
\end{align*}
where $t\in \mathbb{Z}$ and $t^*:=t-\{(n-1)/2\}$.  If $y_0$ satisfies \textup{I} or \textup{II} exclusively, then it is a simple pole.  When this is the case, the residues are as follows.  For poles of type \textup{I}, the residue at such a $y_0$ for $t=0$ is
\begin{align*}
\textup{I}.& \ \ \ (-1)^{n+1}\zeta_n^{1-\ell^* +(n+1)/2}\omega^{1-2\ell^*}(1-\omega)q^{-2(2n+3)\binom{\ell^*-1}{2}-\tfrac34(n^2-2n-3)}{J_{3n}}/\big ({3nJ_{3n(2n+3)}^3}\big ).
\end{align*}
Given the residue at $t=0$, one can use the functional equation of Lemma \ref{lemma:Tnp2-functional} to compute the residue for general $t\in \mathbb{Z}.$  Poles of type \textup{II} are removable.
\end{lemma}

\begin{proof}[Proof of Lemma \ref{lemma:Tn32-residue}.]  We prove the $n$ even case because the equations are more compact.  The proof for $n$ odd is similar.  The first part follows from Definition \ref{definition:defTn32}.  When we compute the residue for the corresponding pole of $R_{n,3}^2(y,q),$ we must choose the same primitive third root of unity.

We prove the residue for poles of type \textup{I} for $\zeta_n\ne 1$.  The case where $\zeta_n=1$ is similar.   Using Proposition \ref{proposition:H1Thm1.3} and simplifying with (\ref{equation:1.8}) yields
\begin{align}
\lim_{y\rightarrow \zeta_n x_0}&(y-\zeta_n x_0)T_{n,3}^2(y,q)=\frac{(-1)^{\ell}\zeta_n^{n/2}x_0 q^{3(2n+3)\binom{\ell}{2}-n(3n+2)/4}}{3n}\cdot
\frac{J_{3n}}{J_{3(2n+3)}^2J_{3n(2n+3)}^3J_{2n+3}^2} \notag\\
\cdot&
\frac{j(\zeta_n;q^{3(2n+3)})j(q^{3n/2+3}x_0,q^{3n/2+3}\zeta_n x_0;q^{2n+3})}
{j(q^{9n/2+9}\zeta_n^3x_0^3;q^{3(2n+3)})}\label{equation:Tn32I-residue}\\
\cdot & \Big \{ j(q^{5n/2+6}\zeta_n x_0^3,q^{5n/2+6}\zeta_n^2 x_0^3;q^{3(2n+3)})
-\frac{q^{n}}{\zeta_n x_0^2}j(q^{n/2+3}\zeta_n x_0^3,q^{n/2+3}\zeta_n^2 x_0^3;q^{3(2n+3)})\Big \}.\notag
\end{align}
Substituting in for $x_0$ from Definition \ref{definition:defTn32}, the first line of (\ref{equation:Tn32I-residue}) can be written
\begin{equation}
{(-1)^{\ell}\zeta_n^{n/2}\omega q^{3(2n+3)\binom{\ell}{2} +\ell(2n+3)-3n/2-3}}\cdot {J_{3n}}/{3nq^{n(3n+2)/4}J_{3(2n+3)}^2J_{3n(2n+3)}^3J_{2n+3}^2}.\label{equation:genfn3-2-firstline}
\end{equation}
We rewrite the quotient of theta functions in the second line of (\ref{equation:Tn32I-residue}).  The first theta function in the numerator can be left as is.  Substituting for $x_0$, the second theta function in the numerator is
\begin{align*}
j(q^{3n/2+3}x_0;q^{2n+3})&=j(\omega q^{\ell (2n+3)};q^{2n+3})=(-1)^{\ell}q^{-(2n+3)\binom{\ell}{2}}\omega^{-\ell}j(\omega;q^{2n+3})&({\text{by }} (\ref{equation:1.8}))\\
&=(-1)^{\ell}\omega^{-\ell}(1-\omega)q^{-(2n+3)\binom{\ell}{2}}J_{3(2n+3)}.
\end{align*}
Using (\ref{equation:1.8}), the third theta function in the numerator becomes
\begin{align*}
j(q^{3n/2+3}\zeta_n x_0;q^{2n+3})&=j(\omega \zeta_n q^{\ell(2n+3)};q^{2n+3})
=(-1)^{\ell}\omega^{-\ell}\zeta_n^{-\ell}q^{-(2n+3)\binom{\ell}{2}}j(\omega \zeta_n;q^{2n+3}).
\end{align*}
Again using (\ref{equation:1.8}), the theta function in the denominator can be rewritten
\begin{align*}
j(q^{9n/2+9}\zeta_n^3x_0^3;q^{3(2n+3)})&=j(\zeta_n^{3}q^{3\ell (2n+3)};q^{3(2n+3)})
=(-1)^{\ell}\zeta_n^{-3\ell}q^{-3(2n+3)\binom{\ell}{2}}j(\zeta_n^{3};q^{3(2n+3)}).
\end{align*}
Assembling the pieces, the second line of (\ref{equation:Tn32I-residue}) can thus be rewritten
\begin{equation}\label{equation:genfn3-2-secondline}
{(-1)^{\ell}\zeta_n^{2\ell}\omega^{-2\ell}(1-\omega)q^{(2n+3)\binom{\ell}{2}}J_{3(2n+3)}j(\zeta_n;q^{3(2n+3)})j(\omega\zeta_n;q^{2n+3})}/{j(\zeta_n^3;q^{3(2n+3)})}.
\end{equation}
We rewrite the expression in braces from (\ref{equation:Tn32I-residue}).  The first theta function of the first product is
\begin{align*}
j(q^{5n/2+6}\zeta_n x_0^3;&q^{3(2n+3)})=j(\zeta_n q^{3\ell(2n+3)-(2n+3)};q^{3(2n+3)})\\
&=(-1)^{\ell}q^{-3(2n+3)\binom{\ell}{2}}(\zeta_n q^{-(2n+3)})^{-\ell} j(\zeta_n q^{-(2n+3)};q^{3(2n+3)})&({\text{by }} (\ref{equation:1.8}))\\
&=(-1)^{\ell+1}\zeta_n^{1-\ell}q^{-3(2n+3)\binom{\ell}{2}+(\ell-1)(2n+3)}j(\zeta_nq^{2(2n+3)};q^{3(2n+3)}).
\end{align*}
The arguments for the next three theta functions are similar.  We have
\begin{align*}
j(q^{5n/2+6}\zeta_n^2x_0^3;q^{3(2n+3)})
&=(-1)^{\ell+1}\zeta_n^{2-2\ell}q^{-3(2n+3)\binom{\ell}{2}+(\ell-1)(2n+3)}j(\zeta_n^2 q^{2(2n+3)};q^{3(2n+3)}),\\
j(q^{n/2+3}\zeta_n x_0^3;q^{3(2n+3)})
&=(-1)^{\ell+1}\zeta_n^{1-\ell}q^{-3(2n+3)\binom{\ell}{2}+2(\ell-1)(2n+3)}j(\zeta_n q^{2n+3};q^{3(2n+3)}),\\
j(q^{n/2+3}\zeta_n^2 x_0^3;q^{3(2n+3)})
&=(-1)^{\ell+1}\zeta_n^{2-2\ell}q^{-3(2n+3)\binom{\ell}{2}+2(\ell-1)(2n+3)}j(\zeta_n^2 q^{2n+3};q^{3(2n+3)}).
\end{align*}
Noting that ${q^n}/({\zeta_n x_0^{2}})=\omega \zeta_n^{-1} q^{-2\ell(2n+3)+2(2n+3)}$, the expression in the braces from (\ref{equation:Tn32I-residue}) becomes
\begin{align*}
\zeta_n^{3-3\ell}&q^{-6(2n+3)\binom{\ell}{2}+2(\ell-1)(2n+3)}
\cdot\Big \{ j(\zeta_nq^{2(2n+3)},\zeta_n^2q^{2(2n+3)};q^{3(2n+3)})\notag \\
&-\omega \zeta_n^{-1} j(\zeta_nq^{2n+3},\zeta_n^2q^{2n+3};q^{3(2n+3)})\Big \}.
\end{align*}
We rewrite the above using Proposition \ref{proposition:prop-f141} with the substitutions $\omega=\omega$, $q= q^{2n+3}$, $y= \zeta_n q^{2n+3}$.  Simplifying with (\ref{equation:1.8}) yields
\begin{equation}
-\zeta_n^{2-3\ell}\omega q^{-6(2n+3)\binom{\ell}{2}+2(\ell-1)(2n+3)}
j(\omega^2 \zeta_n;q^{2n+3}) 
\cdot \frac{j(\zeta_n q^{2(2n+3)},\zeta_n q^{2n+3};q^{3(2n+3)})}{J_{3(2n+3)}}. \label{equation:genfn3-2-thirdline}
\end{equation}
Putting (\ref{equation:genfn3-2-firstline}), (\ref{equation:genfn3-2-secondline}), and (\ref{equation:genfn3-2-thirdline}) together, we obtain
\begin{align*}
\lim_{y\rightarrow y_0}&(y-y_0)T_{n,3}^2(y,q)
=-\frac{\zeta_n^{2-\ell+n/2}\omega^{2-2\ell}(1-\omega)q^{-(2n+3)(\ell^2-4\ell+3)-3n^2/4}}{3n}\\
&\cdot\frac{J_{3n}}{J_{3(2n+3)}^2J_{3n(2n+3)}^3J_{2n+3}^2} \cdot\frac{j(\zeta_n,\zeta_nq^{2n+3},\zeta_nq^{2(2n+3)};q^{3(2n+3)})j(\omega \zeta_n,\omega^2\zeta_n;q^{2n+3})}{j(\zeta_n^3;q^{3(2n+3)})}.
\end{align*} 
Using properties (\ref{equation:1.10}) and (\ref{equation:1.12}), the second line of (\ref{equation:Tn32I-residue}) becomes  $J_{3(2n+3)}^2J_{2n+3}^2$.  The result follows.

We prove that poles of type \textup{II} are removable.  The $s=0$ case is straightforward.  For $s=1,2$, because we are excluding poles of type \textup{I}, it suffices to show that the expression in braces in $T_{n,3}^2(y,q)$ in $(\ref{equation:Tn32-def})$ evaluates to zero.   We unwind the theta functions in the braced expression of $(\ref{equation:Tn32-def})$ with (\ref{equation:1.8}) and pull out a common factor.  Substituting for $x_0$, we see that we must show that the following vanishes
\begin{align}
&j(\omega^2yq^{2\ell(2n+3)-n/2},\omega y^2q^{\ell(2n+3)+n+3};q^{3(2n+3)})\label{equation:dog2}\\
&-\omega^2y^{-1}q^{-\ell(2n+3)+5n/2+3}j(\omega^2yq^{2\ell(2n+3)-5n/2-3},\omega y^2q^{\ell(2n+3)-n};q^{3(2n+3)}).\notag
\end{align} 
Using Proposition \ref{proposition:prop-f141} with the substitutions $\omega=\omega$, $q= q^{2n+3}$, $y=\omega^2 y q^{2\ell(2n+3)-5n/2-3}$ yields
{\allowdisplaybreaks
\begin{align}
&\frac{j(\omega yq^{2\ell(2n+3)-5n/2-3},q^{2n+3})}{J_{3(2n+3)}}\cdot j(\omega^2yq^{2\ell(2n+3)-n/2},\omega^2yq^{2\ell(2n+3)-5n/2-3};q^{3(2n+3)})\notag\\
&=yq^{2\ell(2n+3)-5n/2-3}j(\omega^2yq^{2\ell(2n+3)-5n/2-3},\omega y^2q^{4\ell(2n+3)-n};q^{3(2n+3)})\notag \\
&\ \ \ \ +j(\omega^2yq^{2\ell(2n+3)-n/2},\omega y^2q^{4\ell(2n+3)-5n-6};q^{3(2n+3)})\notag \\
&=(-1)^{\ell+1}\omega^{1-\ell}y^{2-2\ell}q^{-3(2n+3)\binom{\ell}{2}-\ell^2(2n+3)+7\ell n +9\ell -5n-6}\label{equation:Tn32-even-prod}\\
&\ \ \ \ \cdot\Big \{ j(\omega^2yq^{2\ell(2n+3)-n/2},\omega y^2q^{\ell(2n+3)+n+3};q^{3(2n+3)})\notag\\
&\ \ \ \ -\omega^2y^{-1}q^{-\ell(2n+3)+5n/2+3}j(\omega^2yq^{2\ell(2n+3)-5n/2-3},\omega y^2q^{\ell(2n+3)-n};q^{3(2n+3)})\Big \},\notag
\end{align} }%
where the last equality follows from (\ref{equation:1.8}) and simplifying.  To show that (\ref{equation:dog2}) evaluates to zero for $y_0$, we consider the extreme left-hand side of (\ref{equation:Tn32-even-prod}).  For $s=2$, the first theta function in the numerator is
\begin{equation*}
j(\omega yq^{2\ell(2n+3)-5n/2-3},q^{2n+3})=j(q^{(2\ell+t)(2n+3)-2(2n+3)},q^{2n+3})=0.
\end{equation*}
For $s=1$ we have subcases.  When $2\ell+t\equiv 1 \pmod 3$, the second theta function in the numerator is
\begin{equation*}
j(\omega^2yq^{2\ell(2n+3)-n/2};q^{3(2n+3)})=j(q^{(2\ell+t)(2n+3)-(2n+3)};q^{3(2n+3)})=0.
\end{equation*}
When $2\ell+t\equiv 2 \pmod 3$, the third theta function in the numerator becomes
\begin{equation*}
j(\omega^2yq^{2\ell(2n+3)-5n/2-3};q^{3(2n+3)})=j(q^{(2\ell+t)(2n+3)-2(2n+3)};q^{3(2n+3)})=0.
\end{equation*}
When $2\ell+t\equiv 0 \pmod 3$, then $y_0=x_0q^{(t-\ell)(2n+3)}$ is pole of type \textup{I}, which we have excluded.
\end{proof}
\begin{lemma}  \label{lemma:Tn42-residue} Let $n$ and $x_0$ be as in Defintion \ref{definition:defTn42} and $\zeta_n$ and $n$-th root of unity.  We write $x_0^2=iq^{2\ell(2n+4)-(n+4)}$ where $\ell\in\mathbb{Z}$ and $i^2=-1$.  $T_{n,4}^2(y,q)$ is meromorphic for $y\ne 0$ and has poles at points $y_0$, where $y_0$ satisfies at least one of the following two conditions:
\begin{align*}
\textup{I}. & \ y_0=\zeta_n x_0q^{4t (2n+4)}, \\
\textup{II}. & \ y_0^4=-q^{4t (2n+4)-(2n+8)} \ (\textup{or } y_0^2=\pm iq^{2t(2n+4)-(n+4)}),
\end{align*}
where $t\in\mathbb{Z}$.  If $y_0$ satisfies \textup{I} or \textup{II} exclusively, then it is a simple pole.  When this is the case, the residues are as follows.  For poles of type \textup{I}, the residue at such a $y_0$ for $t=0$ is
\begin{align*}
\textup{I}.& \ \ \ \frac{i^{-1-\ell}\zeta_n^{-\ell+1+(n+1)/2}}{4nJ_{4n(2n+4)}^3}q^{-2(2n+4)\binom{\ell-1}{2}-n^2+n+3}
 \cdot \Big \{ (1-i)J_{4n,16n}+i x_0q^{2-\ell (2n+4)}J_{8n,16n}\Big \}.
\end{align*}
Given the residue at $t=0$, one can use the functional equation of Lemma \ref{lemma:Tnp2-functional} to compute the residue for general $t\in \mathbb{Z}.$  Poles of type \textup{II} are removable.
\end{lemma}
\begin{proof}   The first part follows from the Definition of \ref{definition:defTn42}.  When we identify the corresponding pole of $R_{n,4}^2(y,q)$, we must have the same primitive fourth root of unity.

We prove the residue for poles of type \textup{I} where $\zeta_n\ne 1$.  The proof when $\zeta_n=1$ is similar.   Using Proposition \ref{proposition:H1Thm1.3}, we have
{\allowdisplaybreaks 
\begin{align*}
\lim_{y\rightarrow \zeta_nx_0}&(y-\zeta_nx_0)T_{n,4}^2(y,q)=\frac{(-1)^{\ell}\zeta_n^{1+(n+1)/2}x_0^{4}q^{4(2n+4)\binom{\ell}{2}-(n^2+n-3)}j(\zeta_n;q^{4(2n+4)})}{4nJ_{4(2n+4)}^3J_{4n(2n+4)}^3j(-q^{2n+8}\zeta_n^4x_0^4;q^{4(2n+4)})}\notag\\
\cdot \Big [&J_{4n,16n}\cdot \frac{j(q^{6n+16}\zeta_n^2x_0^4;q^{4(2n+4)})j(q^{n+4}\zeta_nx_0^2;q^{2(2n+4)})j(-q^{2(2n+4)}\zeta_n;q^{4(2n+4)})}{J_{2(2n+4)}^3J_{8(2n+4)}} \notag \\
&\cdot  \Big \{j(-q^{2n+8}\zeta_n^2x_0^4;q^{4(2n+4)})j(q^{2(2n+4)}\zeta_n^2;q^{4(2n+4)})J_{4(2n+4)}^2\notag\\
&+\frac{q^{n+4}x_0^2 j(-q^{6n+16}\zeta_n^2x_0^4;q^{4(2n+4)})j(q^{4n+8}\zeta_n;q^{4(2n+4)})^2j(-\zeta_n;q^{4(2n+4)})^2}{J_{4(2n+4)}}\Big \}\notag\\
-&q\zeta_n^{-1}x_0^{-1}J_{8n,16n}\cdot \frac{j(q^{2n+8}\zeta_n^2x_0^4,q^{4(2n+4)})j(q^{3n+8}\zeta_nx_0^2,q^{2(2n+4)})j(-\zeta_n,q^{4(2n+4)})}{J_{2(2n+4)}^2} \notag\\
&\cdot  \Big \{\frac{q^{n+1}j(-q^{2n+8}\zeta_n^2x_0^4;q^{4(2n+4)})j(q^{4n+8}\zeta_n^2;q^{4(2n+4)})J_{8(2n+4)}}{J_{4(2n+4)} }\notag\\ 
&+\frac{q\zeta_n x_0^2j(-q^{6n+16}\zeta_n^2x_0^4;q^{4(2n+4)})j(q^{4(2n+4)}\zeta_n^{2};q^{8(2n+4)})^2}{J_{8(2n+4)} }\Big \} \Big ].
\end{align*}}%
Arguing as we did in the proof of Lemma \ref{lemma:Tn41-residue} yields the result. 

We prove that poles of type \textup{II} are removable.  Because poles of type \textup{I} are excluded, it suffices to show that the expression in brackets in Definition \ref{definition:defTn42} evaluates to zero.
Because of the functional equation satisfied by $T_{n,4}^2(y,q)$ in Proposition \ref{lemma:Tnp2-functional}, we only need to consider poles of type \textup{II} where $t=0,1,2,3$.  We prove the case $y_0^2=iq^{-(n+4)}$; the other seven cases are similar so the proofs will be omitted.  The term in brackets in Definition \ref{definition:defTn42} can thus be written
{\allowdisplaybreaks
 \begin{align}
 \Big [&J_{4n,16n}\cdot \frac{j(q^{6n+16}x_0^2y_0^2;q^{4(2n+4)})j(q^{n+4}x_0 y_0;q^{2(2n+4)})j(-q^{2(2n+4)}y_0/x_0;q^{4(2n+4)})}{J_{2(2n+4)}^3J_{8(2n+4)}} \notag \\
&\cdot  \Big \{j(-q^{2n+8}x_0^2y_0^2;q^{4(2n+4)})j(q^{2(2n+4)}y_0^2/x_0^2;q^{4(2n+4)})J_{4(2n+4)}^2\notag\\
&+\frac{q^{n+4}x_0^2 j(-q^{6n+16}x_0^2y_0^2;q^{4(2n+4)})j(q^{4n+8}y_0/x_0;q^{4(2n+4)})^2j(-y_0/x_0;q^{4(2n+4)})^2}{J_{4(2n+4)}}\Big \}\notag\\
-&qy_0^{-1}J_{8n,16n}\cdot \frac{j(q^{2n+8}x_0^2y_0^2,q^{4(2n+4)})j(q^{3n+8}x_0y_0,q^{2(2n+4)})j(-y_0/x_0,q^{4(2n+4)})}{J_{2(2n+4)}^2}\notag\\
&\cdot  \Big \{\frac{q^{n+1}j(-q^{2n+8}x_0^2y_0^2;q^{4(2n+4)})j(q^{4n+8}y_0^2/x_0^2;q^{4(2n+4)})J_{8(2n+4)}}{J_{4(2n+4)} }\notag\\ 
&+\frac{qx_0y_0j(-q^{6n+16}x_0^2y_0^2;q^{4(2n+4)})j(q^{4(2n+4)}y_0^2/x_0^{2};q^{8(2n+4)})^2}{J_{8(2n+4)} }\Big \} \Big ].\label{equation:Tn42-residue-II}
\end{align}}
Because $y_0^2=iq^{-(n+4)}$ and $x_0^2=i q^{2\ell(2n+4)-(n+4)}$, we have
\begin{align}
y_0/x_0=&\pm q^{-\ell(2n+4)}\label{equation:eq2},\\
x_0y_0=&\pm i q^{\ell(2n+4)-(n+4)},\label{equation:eq4}
\end{align}
where the sign in (\ref{equation:eq2}) determines the respective sign in (\ref{equation:eq4}).   Inserting everything into (\ref{equation:Tn42-residue-II}) produces
{\allowdisplaybreaks
\begin{align}
 \Big [&J_{4n,16n}\cdot \frac{j(-q^{2(\ell+1)(2n+4)};q^{4(2n+4)})j(\pm i q^{\ell(2n+4)};q^{2(2n+4)})j(\mp q^{(2-\ell)(2n+4)};q^{4(2n+4)})}{J_{2(2n+4)}^3J_{8(2n+4)}} \notag \\
&\cdot  \Big \{j(q^{2\ell(2n+4)};q^{4(2n+4)})j(q^{2(1-\ell)(2n+4)};q^{4(2n+4)})J_{4(2n+4)}^2\notag\\
&+\frac{q^{n+4}x_0^2 j(q^{2(1+\ell)(2n+4)};q^{4(2n+4)})j(\pm q^{(2-\ell)(2n+4)};q^{4(2n+4)})^2j(\mp q^{-\ell(2n+4)};q^{4(2n+4)})^2}{J_{4(2n+4)}}\Big \}\notag\\
-&qy_0^{-1}J_{8n,16n}\cdot \frac{j(-q^{2\ell(2n+4)};q^{4(2n+4)})j(\pm iq^{(\ell+1)(2n+4)};q^{2(2n+4)})j(\mp q^{-\ell(2n+4)};q^{4(2n+4)})}{J_{2(2n+4)}^2}\notag\\
&\cdot  \Big \{\frac{q^{n+1}j(q^{2\ell(2n+4)};q^{4(2n+4)})j(q^{2(1-\ell)(2n+4)};q^{4(2n+4)})J_{8(2n+4)}}{J_{4(2n+4)} }\notag\\ 
&\pm i \cdot \frac{q^{\ell(2n+4)-n-3}j(q^{2(1+\ell)(2n+4)};q^{4(2n+4)})j(q^{2(2-\ell)(2n+4)};q^{8(2n+4)})^2}{J_{8(2n+4)} }\Big \} \Big ],\label{equation:Tn42-residue-II-b}
\end{align}}%
where the sign choice follows from (\ref{equation:eq2}).  We now have several subcases to consider depending on the parity of $\ell$ and the sign choice in (\ref{equation:eq2}).   Suppose $\ell$ is odd, then expressions in the first and second set of braces are both zero, i.e. both products within each set of braces are zero.  Suppose $\ell$ is even and we take the plus sign choice in (\ref{equation:eq2}).   Here we must have that $\ell \equiv 2 \pmod 4$, otherwise we would have a pole of type \textup{I} which we have excluded.  Indeed, we would have $y_0=x_0q^{4(\ell/4)(2n+4)}$.  So we have that $\ell \equiv 2 \pmod 4$.  Here both products within each set of braces are zero.  Suppose $\ell$ is even and we take the minus sign choice in (\ref{equation:eq2}).  If $\ell\equiv 0 \pmod 4$, then both products within the first set of braces is zero and the coefficient of the second set of braces is zero.   If $\ell \equiv 2 \pmod 4$, then the coefficient of the first set of braces is zero and the both products within the second set of braces are zero.  For the other seven cases, when an expression within braces evaluates to zero it is not always that both products are zero.  One may have to argue as the proof of Lemma \ref{lemma:Tn41-residue} and use (\ref{equation:jsplit}) with $m=2$ or use $(\ref{equation:H1Thm1.1})$ to combine the two products into a single product which is easily seen to be zero.
\end{proof}


\section{Proof of Corollary \ref{corollary:cor-P1-KP}}\label{section:P1-KP-corollaries}

We prove (\ref{equation:KP-(5.22)}), which is equivalent to $f_{5,5,1}(q^5,q^2,q)=J_2J_{10}.$  Specializing Theorem \ref{theorem:main-acdivb},
\begin{align*}
&f_{5,5,1}(x,y,q)=j(x;q^5)m(-q^4x^{-1}y,q^4,-1)+j(y;q)m(-q^{10}xy^{-5},q^{20},-1)\\
&-\frac{1}{4\overline{J}_{4,16}\overline{J}_{20,80}}\sum_{d=0}^{4}
q^{2d(d+1)}j\big (q^{4+4d}y;q^{5}\big )  j\big (-q^{16-4d}xy^{-1};q^{20}\big ) \frac{J_{20}^3j\big (q^{14+4d}y^{-4};q^{20}\big )}
{j\big (q^{10}xy^{-5},q^{4+4d}x^{-1}y;q^{20}\big )},
\end{align*}
Substituting in for $x,y$ yields
\begin{align*}
f_{5,5,1}&(q^5,q^2,q)=0+0-\frac{1}{4\overline{J}_{4,16}\overline{J}_{20,80}}\sum_{d=0}^4q^{2d(d+1)}\frac{j(q^{6+4d};q^5)j(-q^{19-4d};q^{20})J_{20}^3j(q^{6+4d};q^{20})}{J_{5,20}j(q^{1+4d};q^{20})}\\
&=-\frac{1}{4\overline{J}_{4,16}\overline{J}_{20,80}}\Big [ \frac{J_{6,5}\overline{J}_{1,20}J_{20}^3J_{6,20}}{J_{5,20}J_{1,20}}+0+q^{12}\frac{J_{14,5}\overline{J}_{11,20}J_{20}^3J_{14,20}}{J_{5,20}J_{9,20}}\\
&\ \ \ \ +q^{24}\frac{J_{18,5}\overline{J}_{7,20}J_{20}^3J_{18,20}}{J_{5,20}J_{13,20}}+q^{40}\frac{J_{22,5}\overline{J}_{3,20}J_{20}^3J_{22,20}}{J_{5,20}J_{17,20}}\Big ]\\
&=\frac{1}{4\overline{J}_{4,16}\overline{J}_{20,80}}\cdot \frac{J_{20}^3}{J_{5,20}}\cdot \Big [ q^{-1}J_{1,5}J_{6,20}\Big \{  \frac{\overline{J}_{1,20}}{J_{1,20}}-\frac{\overline{J}_{9,20}}{J_{9,20}}\Big \}+ J_{2,5}J_{2,20}\Big \{  \frac{\overline{J}_{7,20}}{J_{7,20}}+\frac{\overline{J}_{3,20}}{J_{3,20}}\Big \} \Big ],
\end{align*}
where the last line follows from (\ref{equation:1.8}).  Applying (\ref{equation:H1Thm1.2A}) to the expression in the first set of braces and (\ref{equation:H1Thm1.2B}) to the expression in the second set of braces yields
\begin{align*}
f_{5,5,1}&(q^5,q^2,q)=\frac{1}{2\overline{J}_{4,16}\overline{J}_{20,80}}\cdot \frac{J_{20}^3}{J_{5,20}}\cdot J_{10,40}\cdot \Big [ \frac{J_{1,5}J_{6,20}}{J_{1,20}J_{9,20}}\cdot J_{8,40}+ \frac{J_{2,5}J_{2,20}}{J_{7,20}J_{3,20}}\cdot J_{16,40} \Big ]\\
&=\frac{1}{2\overline{J}_{4,16}\overline{J}_{20,80}}\cdot \frac{J_{20}^3}{J_{5,20}}\cdot J_{10,40}\cdot \frac{J_{10}}{J_{20}^2}\cdot \frac{J_5}{J_{10}^2}\Big [ \frac{J_{1,10}J_{6,10}J_{6,20}}{J_{1,10}}\cdot J_{8,40}+ \frac{J_{2,10}J_{7,10}J_{2,20}}{J_{3,10}}\cdot J_{16,40} \Big ],
\end{align*}
where the last line follows from (\ref{equation:1.10}).  Simplifying,
{\allowdisplaybreaks
\begin{align*}
f_{5,5,1}&(q^5,q^2,q)=\frac{J_{4}J_{10}}{2J_{8}^2J_{40}}\Big [ J_{6,10}J_{6,20} J_{8,40}+ J_{2,10}J_{2,20} J_{16,40} \Big ]\\
&=\frac{J_{4}J_{10}}{2J_{8}^2J_{40}}\cdot \frac{J_{40}}{J_{20}^2}\cdot \Big [ J_{6,10}J_{6,20} J_{4,20}\overline{J}_{4,20}+ J_{2,10}J_{2,20} J_{8,20}\overline{J}_{8,20} \Big ]\\
&=\frac{J_{4}J_{10}}{2J_{8}^2J_{40}}\cdot \frac{J_{40}}{J_{20}^2}\cdot \frac{J_{20}^2}{J_{10}}\cdot \Big [ J_{6,10}^2\overline{J}_{4,20}+ J_{2,10}^2\overline{J}_{8,20} \Big ]
=\frac{J_{4}}{2J_{8}^2} \Big [ J_{6,10}^2\overline{J}_{4,20}+ J_{2,10}^2\overline{J}_{8,20} \Big ],
\end{align*}}%
where the second and third equalities follow from (\ref{equation:1.12}) and (\ref{equation:1.10}) respectively, and the last equality follows from simplifying.
We recall the identities $J_{1,5}j(q^4;-q^5)=J_{2,4}\overline{J}_{2,10}$ and $J_{2,5}j(q^2;-q^5)=J_{2,4}\overline{J}_{4,10}$ and use them to obtain
{\allowdisplaybreaks
\begin{align*}
&f_{5,5,1}(q^5,q^2,q)=\frac{J_{4}}{2J_{8}^2} \Big [ \frac{J_{4,10}^2J_{2,10}j(q^8;-q^{10})}{J_{4,8}}+ \frac{J_{2,10}^2J_{4,10}j(q^4;-q^{10})}{J_{4,8}} \Big ]\\
&=\frac{J_{4}}{2J_{8}^2}\cdot  \frac{J_{2,10}J_{4,10}}{J_{4,8}}\Big [ J_{4,10}j(q^8;-q^{10})+ J_{2,10}j(q^4;-q^{10}) \Big ]\\
&=\frac{J_{4}}{2J_{8}^2}\cdot  \frac{J_{2,10}J_{4,10}}{J_{4,8}}\Big [ J_{4,10}\frac{J_{8,20}\overline{J}_{18,20}}{J_{10,40}}+ J_{2,10}\frac{J_{4,20}\overline{J}_{14,20}}{J_{10,40}} \Big ]
=\frac{J_2J_{10}}{2J_8J_{40}}\Big [ J_{6,20}\overline{J}_{2,20}+J_{2,20}\overline{J}_{6,20}\Big ],
\end{align*}}
where the third equality follows from (\ref{equation:1.11}), and the last follows (\ref{equation:1.10}) and the identity $J_{1,5}J_{2,5}=J_1J_5$.  The result then follows from applying (\ref{equation:H1Thm1.2B}) to the expression in brackets.

\section{Proofs of the corollaries to the four subtheorems}\label{section:cor-proofs}

The following proposition will facilitate the proofs of the corollaries.

\begin{proposition}\label{proposition:prop-singshift}  Let $\ell\in \mathbb{Z}$, $p\in \{ 1,2,3,4\}$ and $n\in \mathbb{N}$ with $(n,p)=1$.  For generic $x,y\in\mathbb{C}^*$
\begin{equation*}
f_{n,n+p,n}(x,y,q)=g_{n,n+p,n}(x,y,q,q^{\ell np}y^n/x^n,q^{-\ell np}x^n/y^n)-(-x)^{\ell}q^{n\binom{\ell}{2}}\Theta_{n,p}(q^{\ell n}x,q^{\ell (n+p)}y,q).
\end{equation*}
\end{proposition}
\begin{proof}  Recall the summation convention of Proposition \ref{proposition:f-functionaleqn}.  Using Proposition \ref{proposition:f-functionaleqn} with $k=0$, we have
\begin{align*}
f_{n,n+p,n}(x,y,q)&=(-x)^{\ell}q^{n\binom{\ell}{2}}\Big [g_{n,n+p,n}(q^{n\ell}x,q^{(n+p)\ell}y,q,q^{\ell np}y^n/x^n,q^{-\ell np}x^n/y^n)\\
&\ \ \ \ -\Theta_{n,p}(q^{\ell n}x,q^{\ell (n+p)}y,q)\Big ]+\sum_{m=0}^{\ell-1}(-x)^mq^{n\binom{m}{2}}j(q^{m(n+p)}y;q^n).
\end{align*}
We focus on the $g_{n,n+p,n}$ term and write it out to obtain
{\allowdisplaybreaks \begin{align}
\sum_{r=0}^{n-1}&\frac{q^{-\ell pr}x^r}{q^{pr^2}y^r}j(q^{pr+(n+p)\ell}y;q^n)
m\Big (-q^{n(np+\binom{p+1}{2})-(r+\ell)p(2n+p)}\frac{(-x)^n}{(-y)^{n+p}},q^{np(2n+p)},\frac{q^{\ell pn}y^n}{x^n}\Big )\label{equation:singshift-I}\\
&+\sum_{r=0}^{n-1}\frac{q^{\ell pr}y^r}{q^{pr^2}x^r}j(q^{pr+n\ell}x;q^n)
m\Big (-q^{n(np+\binom{p+1}{2})-rp(2n+p)}\frac{(-y)^n}{(-x)^{n+p}},q^{np(2n+p)},\frac{q^{-\ell pn}x^n}{y^n}\Big ).\notag
\end{align}}%
We unwind both theta functions with (\ref{equation:1.8}) and then replace $r$ with $r-\ell$ in the first sum to have
\begin{align}
&\frac{(-1)^{\ell}}{x^{\ell}q^{n\binom{\ell}{2}}}\Big [\sum_{r=\ell}^{n+\ell-1}\frac{x^r}{q^{pr^2}y^r}j(q^{pr}y;q^n)
m\Big (-q^{n(np+\binom{p+1}{2})-rp(2n+p)}\frac{(-x)^n}{(-y)^{n+p}},q^{np(2n+p)},\frac{q^{\ell pn}y^n}{x^n}\Big )\notag\\
&+\sum_{r=0}^{n-1}\frac{y^r}{q^{pr^2}x^r}j(q^{pr}x;q^n)
m\Big (-q^{n(np+\binom{p+1}{2})-rp(2n+p)}\frac{(-y)^n}{(-x)^{n+p}},q^{np(2n+p)},\frac{q^{-\ell pn}x^n}{y^n}\Big )\Big ].\label{equation:singshift-II}
\end{align}
The second sum in (\ref{equation:singshift-II}) is exactly what we want.  We rewrite the first sum in (\ref{equation:singshift-II}) using  $\sum_{r=\ell}^{n+\ell-1}=\sum_{r=\ell}^{n-1}+\sum_{r=n}^{n+\ell-1}:$ 
\begin{align*}
\sum_{r=0}^{\ell-1}\frac{x^r}{q^{pr^2}y^r}j(q^{pr}y;q^n)
\Big [m\Big (-q^{n(np+\binom{p+1}{2})-rp(2n+p)}\frac{(-x)^n}{(-y)^{n+p}},q^{np(2n+p)},\frac{q^{\ell pn}y^n}{x^n}\Big )-1\Big ],
\end{align*}
where we have replaced $r$ with $r+n$, unwound the theta function with (\ref{equation:1.8}), and then used (\ref{equation:mxqz-altdef1}). The result follows.
\end{proof}

\subsection{Proof of Corollary \ref{corollary:f232-tenth}}

Rewriting the respective Hecke-type sum formulas from \cite{C1}, \cite{C2}:
{\allowdisplaybreaks
\begin{align}
J_{1,2}\phi(q)&=f_{2,3,2}(q^2,q^2,q),\\
J_{1,2}\psi(q)&=-q^2f_{2,3,2}(q^4,q^4,q),\\
\overline{J}_{1,4}X(q)&=f_{2,3,2}(-q^3,-q^3,q^2),\\
\overline{J}_{1,4}(2-\chi(q))&=qf_{2,3,2}(-q^{-1},-q^{-1},q^2).
\end{align}}
Specializing Theorem \ref{theorem:genfn1} to $n=2$ and using Proposition \ref{proposition:prop-singshift} yields
\begin{corollary} Let $\ell\in \mathbb{Z}$.  For generic $x,y\in \mathbb{C}^*$
\begin{align}
f_{2,3,2}(x,y,q)=\sum_{r=0}^{1}\Big [ \frac{x^r}{q^{r^2}y^r}j(q^{r}y;q^2)&
m \Big (\frac{q^{6-5r}x^2}{y^{3}},q^{10},\frac{q^{2\ell}y^2}{x^2} \Big )\label{equation:f232} \\
&+\frac{y^r}{q^{r^2}x^r}j(q^{r}x;q^2)
m\Big (\frac{q^{6-5r}{y^2}}{x^{3}},q^{10},\frac{x^2}{q^{2\ell}y^2}\Big )\Big ].\notag
\end{align}
\end{corollary}
\noindent To prove (\ref{equation:10th-phi(q)}), we use (\ref{equation:f232}) with $\ell=1$
\begin{align}
\phi(q)&=f_{2,3,2}(q^2,q^2,q)/J_{1,2}=-q^{-2}m(q^{-1},q^{10},q^2)-q^{-2}m(q^{-1},q^{10},q^{-2})\notag\\
&=-q^{-1}m(q,q^{10},q^{-2})-q^{-1}m(q,q^{10},q^{2})&(\text{by } (\ref{equation:mxqz-flip}))\notag\\
&=-q^{-1}m(q,q^{10},q)-q^{-1}m(q,q^{10},q^2).&(\text{by Cor \ref{corollary:mxqz-flip-xz}})\notag
\end{align}
For (\ref{equation:10th-psi(q)}), we use (\ref{equation:f232}) with $\ell=2$.   To prove (\ref{equation:10th-BigX(q)}), we use (\ref{equation:f232}) with $\ell=1$
\begin{align*}
X(q)&=f_{2,3,2}(-q^3,-q^3,q^2)/\overline{J}_{1,4}\\
&=m(-q^9,q^{20},q^{4}) +q^{-3}m(-q^{-1},q^{20},q^{4})
+m(-q^{9},q^{20},q^{-4})+ q^{-3}m(-q^{-1},q^{20},q^{-4}) \\
&=m(-q^9,q^{20},q^{4}) +q^{-3}m(-q^{-1},q^{20},q^{4})
+m(-q^{9},q^{20},q^{16})+ q^{-3}m(-q^{-1},q^{20},q^{16})\\
&=m(-q^2,q^5,q)+m(-q^2,q^5,q^4),
\end{align*}
where the the last two equalities follow from (\ref{equation:mxqz-fnq-z}) and Corollary \ref{corollary:msplit-n=2} respectively.  For (\ref{equation:10th-chi(q)}), we use (\ref{equation:f232}) with $\ell=2$ and argue as above and use (\ref{equation:mxqz-fnq-x}) as well as (\ref{equation:mxqz-flip})

\subsection{Proof of Corollary \ref{corollary:f343-seventh}}
We rewrite the respective Hecke-type sums from \cite{H2}:
\begin{align}
J_1{\mathcal{F}}_0(q)&=f_{3,4,3}(q^2,q^2,q)\label{equation:F0-7th},\\
J_1{\mathcal{F}}_1(q)&=qf_{3,4,3}(q^4,q^4,q)\label{equation:F1-7th},\\
J_1{\mathcal{F}}_2(q)&=f_{3,4,3}(q^3,q^3,q)\label{equation:F2-7th}.
\end{align}
Specializing Theorem \ref{theorem:genfn1} to $n=3$ and using Proposition \ref{proposition:prop-singshift} yields
\begin{corollary}  Let $\ell\in \mathbb{Z}$.  For generic $x,y\in \mathbb{C}^*$
\begin{align}
f_{3,4,3}(x,y,q)=\sum_{r=0}^{2}\Big [ \frac{x^r}{q^{r^2}y^r}j(q^{r}y;q^3)&
m \Big (\frac{q^{12-7r}{x^3}}{y^{4}},q^{21},\frac{q^{3\ell}y^3}{x^3} \Big )\label{equation:f343}\\
&+\frac{y^r}{q^{r^2}x^r}j(q^{r}x;q^3)
m\Big (\frac{q^{12-7r}{y^3}}{x^{4}},q^{21},\frac{x^3}{q^{3\ell}y^3}\Big )\Big ].\notag
\end{align}
\end{corollary}

\noindent To prove (\ref{equation:7th-F0(q)}), we use (\ref{equation:f343}) with $\ell=3$
\begin{align*}
{\mathcal{F}}_0(q)&=f_{3,4,3}(q^2,q^2,q)/J_1\\
&=m(q^{10},q^{21},q^9)-q^{-5}m(q^{-4},q^{21},q^{9})+m(q^{10},q^{21},q^{-9})-q^{-5}m(q^{-4},q^{21},q^{-9})\\
&=m(q^{10},q^{21},q^9)-q^{-1}m(q^4,q^{21},q^{-9})+m(q^{10},q^{21},q^{-9})-q^{-1}m(q^4,q^{21},q^{9}),\notag
\end{align*}
where the last line follows from  (\ref{equation:mxqz-flip}).  For (\ref{equation:7th-F1(q)}) (resp. (\ref{equation:7th-F2(q)})) we use (\ref{equation:f343}) with $\ell=1$ (resp. $\ell=2$).
\subsection{Proof of Corollary \ref{corollary:f373-fifth}}
We rewrite the respective Hecke-type sum formulations from \cite{A}:
\begin{align}
J_1 f_0(q)&=f_{3,7,3}(q^2,q^2,q)+q^3f_{3,7,3}(q^7,q^7,q)=f_{3,7,3}(q^{5/8},-q^{5/8},-q^{1/4})\label{equation:f0-hecke},\\
J_1 f_1(q)&=f_{3,7,3}(q^3,q^3,q)+q^4f_{3,7,3}(q^8,q^8,q)=f_{3,7,3}(q^{9/8},-q^{9/8},-q^{1/4}),\\
J_2 F_0(q)&=f_{3,7,3}(q^4,q^6,q^2)-q^7f_{3,7,3}(q^{14},q^{16},q^2)\notag\\
&=\tfrac{1}{4}\Re \{ f_{3,7,3}(q^{7/16},q^{15/16},iq^{1/8}) +f_{3,7,3}(-q^{7/16},-q^{15/16},-iq^{1/8})+2J_2\},\\
J_2 F_1(q)&=f_{3,7,3}(q^6,q^8,q^2)-q^9f_{3,7,3}(q^{16},q^{18},q^2)\notag\\
&=\tfrac{1}{4}\Re \{ f_{3,7,3}(q^{15/16},q^{23/16},iq^{1/8}) +f_{3,7,3}(-q^{15/16},-q^{23/16},-iq^{1/8})\},
\end{align}
where we also used Propositions \ref{proposition:fabc-mod2}, \ref{proposition:H7eq1.14}, and Corollary \ref{corollary:fabc-funceqnspecial}.  Specializing Theorem \ref{theorem:genfn4} to $n=3$ gives
\begin{corollary}  \label{corollary:f373} For generic $x,y\in \mathbb{C}^*$
\begin{align*}
f_{3,7,3}(x,y,q)=\sum_{r=0}^2\Big [&\frac{x^r}{q^{4r^2}y^r}j(q^{4r}y;q^3)m(-q^{66-40r}x^3/y^7,q^{120},y^3/x^3)\\
&+\frac{y^r}{q^{4r^2}x^r}j(q^{4r}x;q^3)m(-q^{66-40r}y^3/x^7,q^{120},x^3/y^3)\Big ]
-\Theta_{3,4}(x,y,q),
\end{align*}
where
\begin{align*}
\Theta_{3,4}(x,y,q):&=\frac{q^{-9}y^2j(y/x;q^{40})}{j(y^3/x^3;q^{120})j(-q^{14}x^4,-q^{14}y^4;q^{40})}\Big \{ J_{12,48}  S_1-qJ_{24,48} S_2\Big \},
\end{align*}
with
\begin{align*}
S_1:&=\frac{j(q^{34}x^2y^2;q^{40})j(q^{7}xy;q^{20})j(-q^{20}y/x;q^{40})}{J_{20}^3J_{80}} \\
&\ \ \ \ \cdot  \Big \{j(-q^{14}x^2y^2,q^{20}y^2/x^2;q^{40})J_{40}^2
+\frac{q^{7}x^2 j(-q^{34}x^2y^2;q^{40})j(q^{20}y/x,-y/x;q^{40})^2}{J_{40}}\Big \},\\
S_2:&=\frac{j(q^{14}x^2y^2;q^{40})j(q^{17}xy;q^{20})j(-y/x;q^{40})}{J_{20}^2} \\
&\ \ \ \ \cdot  \Big \{\frac{q^{4}j(-q^{14}x^2y^2;q^{40})j(q^{20}y^2/x^2;q^{40})J_{80}}{yJ_{40} }
+\frac{qxj(-q^{34}x^2y^2;q^{40})j(q^{40}y^2/x^2;q^{80})^2}{J_{80} }\Big \}.
\end{align*}
\end{corollary}

\noindent For (\ref{equation:5th-f0(q)}), we use (\ref{equation:f0-hecke}) to first obtain the $g_{3,7,3}(x,y,q,y^3/x^3,x^3/y^3)$ expression
\begin{align*}
g_{3,7,3}(q^{5/8},-q^{5/8},-q^{1/4},-1,-1)
&=\Big [j(-q^{5/8};-q^{3/4})+j(q^{5/8};-q^{3/4})\Big ]m(q^{14},q^{30},-1)\\
&\ \ -q^{-1}\Big [ j(-q^{13/8};-q^{3/4}) + j(q^{13/8};-q^{3/4}) \Big ]m(q^{4},q^{30},-1)\\
&\ \   +q^{-4}\Big [ j(-q^{21/8};-q^{3/4}) + j(q^{21/8};-q^{3/4}) \Big ]m(q^{-6},q^{30},-1)\\
&= 2J_1m(q^{14},q^{30},-1)+2q^{-2}J_1m(q^{4},q^{30},-1),
\end{align*}
where the last line follows by applying (\ref{equation:jsplit}) with $m=2$ to each theta function within brackets.   Upon substitution of the values of $x,y,q$ from (\ref{equation:f0-hecke}) into $\Theta_{3,4}(x,y,q)$, three of the four theta quotients vanish
\begin{align*}
\Theta_{3,4}(q^{5/8},-q^{5/8},-q^{1/4})&=\frac{-q^{-1}\overline{J}_{0,10}J_{3,12}j(q^{11};q^{10})J_{3,5}J_{5,10}^2J_{10}^2}{\overline{J}_{0,30}\overline{J}_{6,10}J_{5}^3J_{20}}
=\frac{q^{-2}\overline{J}_{0,10}J_{3,12}J_5J_1}{\overline{J}_{0,30}{J}_{8,20}},
\end{align*}
where the last equality follows from $J_1J_5=J_{1,5}J_{2,5}$ and simplifying. Hence,
\begin{align*}
f_0(q)&= 2m(q^{14},q^{30},-1)+2q^{-2}m(q^{4},q^{30},-1)-\frac{q^{-2}\overline{J}_{0,10}J_{3,12}J_5}{\overline{J}_{0,30}{J}_{8,20}}\\
&=2m(q^{14},q^{30},q^4)+2q^{-2}m(q^{4},q^{30},q^4)\\
&\ \ \ \ +2\frac{J_{30}^3\overline{J}_{4,30}}{\overline{J}_{0,30}J_{4,30}}\Big [ \frac{\overline{J}_{18,30}}{\overline{J}_{14,30}J_{18,30}}+\frac{q^{-2}\overline{J}_{8,30}}{\overline{J}_{4,30}J_{8,30}}\Big ]
-\frac{\overline{J}_{0,10}J_{3,12}J_5}{q^2\overline{J}_{0,30}{J}_{8,20}}&(\text{by Thm \ref{theorem:changing-z-theorem}})\\
&=2m(q^{14},q^{30},q^4)+2q^{-2}m(q^{4},q^{30},q^4)
+\frac{2\overline{J}_{6,30}J_{10}^2}{q^2\overline{J}_{0,30}J_{2,10}}
-\frac{\overline{J}_{0,10}J_{3,12}J_5}{q^2\overline{J}_{0,30}{J}_{8,20}},
\end{align*}
where the last line follows from Proposition \ref{proposition:CHcorollary} with $q=q^{30}, a=-q^{11},b=q^3,c=q^5,d=q^7$ and identity (\ref{equation:1.10}).  It remains to show
\begin{equation}
\frac{2q^{-2}\overline{J}_{6,30}J_{10}^2}{\overline{J}_{0,30}J_{2,10}}
-\frac{q^{-2}\overline{J}_{0,10}J_{3,12}J_5}{\overline{J}_{0,30}{J}_{8,20}}=\frac{J_{5,10}J_{2,5}}{J_1},\label{equation:f0-state}
\end{equation}
which is equivalent to
\begin{equation}
\overline{J}_{5,20}\overline{J}_{6,30}-q^2\overline{J}_{1,5}\overline{J}_{30,120}=J_{2,20}J_{3,12}.\label{equation:f0-state-2}
\end{equation}
Using $J_{1,4}\overline{J}_{1,10}=J_{2,5}J_{4,20},$ (\ref{equation:f0-state-2}) is then equivalent to (\ref{equation:f0f1}) with the $x=1.$  For (\ref{equation:5th-f1(q)}), we argue as above, using Proposition \ref{proposition:CHcorollary} with $q=q^{30}, a=-q^{7},b=q,c=q^5,d=q^9$, to reduce the problem to
\begin{align}
\frac{2J_{10}^2}{\overline{J}_{0,30}}\cdot \frac{q^{-3}\overline{J}_{12,30}}{J_{4,10}}=&\frac{\overline{J}_{0,10}J_{3,12}J_{5}}{\overline{J}_{0,30}}\cdot \frac{q^{-3}}{J_{4,20}}+\frac{J_{5,10}}{J_1}\cdot J_{1,5},\label{equation:f1-state}
\end{align}
which is equivalent to 
\begin{equation}
\overline{J}_{5,20}\overline{J}_{12,30}-q^3\overline{J}_{2,5}\overline{J}_{30,120}=J_{6,20}J_{3,12}.\label{equation:f1-state-2}
\end{equation}
Using the fact that $J_{1,4}\overline{J}_{3,10}=J_{1,5}J_{8,20},$ (\ref{equation:f1-state-2}) is then equivalent to (\ref{equation:f0f1}) with the $x=q^6.$

For (\ref{equation:5th-BigF0(q)}) and (\ref{equation:5th-BigF1(q)}) we argue as above but use  (\ref{equation:F0F1}).

\section*{Acknowledgements}
\noindent{\bf To be entered later.}

\end{document}